\documentclass[math,plain,dissertation]{puthesis}

\usepackage{amsmath}

\usepackage{multicol}

\usepackage{subfig}

\title{Some Quantitative Results in Real Algebraic Geometry}

\author{Salvador Barone}{Barone, Salvador}

\pudegree{Doctor of Philosophy}{Ph.D.}{August}{2013}

\majorprof{Saugata Basu}

\campus{West Lafayette}

\usepackage{graphicx,color,epsf}
\usepackage [cmtip,arrow]{xy}
\xyoption{all}

\usepackage{subfig}
\usepackage{float}

\usepackage {pb-diagram,pb-xy}

\usepackage{amsmath,amssymb}
\newtheorem{lemma}[theorem]{Lemma}
\newtheorem{corollary}[theorem]{Corollary}
\newtheorem{notationdefinition}[theorem]{Notation and definition}

\newtheorem{example}[theorem]{Example}
\newtheorem{notation}[theorem]{Notation}
\newtheorem{property}[theorem]{Property}

\newtheorem{remark}[theorem]{Remark}

\definecolor{DarkBlue}{rgb}{0,0.1,0.55}

\numberwithin{equation}{section}

\newcommand {\hide}[1]{}

\newcommand {\junk}[1]{}

\newcommand {\ZZ} {{\rm Zer}}
\newcommand {\RR} {{\mathcal R}}

\newcommand {\V} {\mathbf{V}}

\newcommand {\la}   {{\langle}}
\newcommand {\ra}   {{\rangle}}
\newcommand {\eps} {{\varepsilon}}

\newcommand {\Id} {\mbox{\rm Id}}

\newcommand {\PP}     {\mathbb{P}} 

\newcommand {\T}      {{\mbox{\rm T}}}

\newcommand {\Zer} {{\rm Zer}}
\newcommand {\Reali} {{\rm Reali}}
\newcommand {\Cr} {{\rm Cr}}
\newcommand {\Pos} {{\rm Pos}}
\newcommand {\Def} {{\rm Def}}
\newcommand {\card} {{\rm card}}

\def\addots{\mathinner{\mkern1mu
\raise1pt\vbox{\kern7pt\hbox{.}}
\mkern2mu\raise4pt\hbox{.}\mkern2mu
\raise7pt\hbox{.}\mkern1mu}}

\newcommand{\HH}  {\mbox{\rm H}}

\newcommand{\defeq}{\;{\stackrel{\text{\tiny def}}{=}}\;}

\newcommand{\X}{\mathbf{X}}

\newcommand{\case}[1]{\textbf{Case #1:}}
\newcommand{\zz}{\begin{flushright}$\Box$\end{flushright}}
\newcommand{\blist}
    {\begin{list}
    {\textbf{\arabic{enumi}.}}
    {\setlength{\leftmargin}{0.365 in}
    \setlength{\rightmargin}{0.365 in}
    \usecounter{enumi}}}
\newcommand{\alist}
    {\begin{list}
    {\textbf{(\alph{enumii})}}
    {\setlength{\leftmargin}{0.365 in}
    \setlength{\rightmargin}{0.365 in}
    \usecounter{enumii}}}
\newcommand{\rlist}
    {\begin{list}
    {\textbf{(\roman{enumiii})}}
    {\setlength{\leftmargin}{0.365 in}
    \setlength{\rightmargin}{0.365 in}
    \usecounter{enumiii}}}
\newcommand{\Rlist}
    {\begin{list}
    {\textbf{(\Roman{enumii})}}
    {\setlength{\leftmargin}{0.365 in}
    \setlength{\rightmargin}{0.365 in}
    \usecounter{enumii}}}
\newcommand{\elist}{\end{list}}

\newcommand{\Ext}{\text{Ext}}
\newcommand{\ep}{\varepsilon}

\usepackage{graphicx,color,epsf}
\usepackage [cmtip,arrow]{xy}
\xyoption{all}

\usepackage{subfig}
\usepackage{float}

\usepackage {pb-diagram,pb-xy}

\usepackage{amsmath,amssymb}

\newcommand{\limit}{\mathrm{limit}}

\newcommand{\halfspace}
    {\quad \hspace{-0.12 in}  \hspace{-0.11 in} \quad }
\newcommand{\halfhalfspace}
    {\quad \hspace{-0.14 in}  \hspace{-0.15 in} \quad }

\newcommand{\halfbackspace}
    {\hspace{-0.06 in}}

\newcommand{\fraction}[2]
    {\textstyle{\frac{#1}{#2}}}

\newcommand{\C}{\textrm{C}}
\newcommand{\s}{\textrm{ sign}}
\newcommand{\Z}{\mathbb{Z}}
\newcommand{\re}{\mathbb{R}}
\newcommand{\N}{\mathbb{N}}

\newcommand{\TT}{\mathbb{T}}
\newcommand {\Sphere}{\mbox{${\bf S}$}}               

\newcommand{\R}{\textnormal{\textrm{R}}}

\newcommand{\wee}{\wedge}

\begin{document}

\volume

%
%
%
%
%


\begin{acknowledgments}

I am indebted to my wife Jess for her immeasurable philanthropy, my children Mia and Sal Jr. for being a constant source of entertainment, and my parents for their continued support and constant encouragement.  My deepest gratitude goes to my advisor, Saugata Basu, who provided all of these things and without whose thorough instruction and discerning mentoring this thesis would not have been possible.  Also, my thanks goes to my other committee members: Leonard Lipshitz, Ben McReynolds, and a special thanks to Andrei Gabrielov. 

\end{acknowledgments}


\tableofcontents


\listoffigures








\begin{abstract}

Real algebraic geometry is the study of semi-algebraic sets, subsets of $\R^k$ defined by boolean combinations of polynomial equalities and inequalities.  The focus of this thesis is to study quantitative results in real algebraic geometry, primarily upper bounds on the topological complexity of semi-algebraic sets as measured, for example, by their Betti numbers.  
Another quantitative measure of topological complexity which we study is the number of  homotopy types of semi-algebraic sets of bounded description complexity.  The description complexity of a semi-algebraic set depends on the context, but it is simply some measure of the complexity of the polynomials in the formula defining the semi-algebraic set (e.g., the degree and number of variables of a polynomial, the so called dense format). We also provide a description of the Hausdorff limit of a one-parameter family of semi-algebraic sets, up to homotopy type
which the author hopes will find further use in the future.  
Finally, in the last chapter of this thesis, we prove a decomposition theorem similar to the well-known cylindrical decomposition for semi-algebraic sets but which has the advantage of requiring significantly less cells in the decomposition.  The setting of this last result is not in semi-algebraic geometry, but in the more general setting of o-minimal structures.

\end{abstract}

\include{introduction}

\chapter{Preliminaries}{\label{ch:prelim}}

In this chapter we provide some background theory for the results of Chapter \ref{ch:refined} and Chapter \ref{ch:homotopy}.  The content of Chapter \ref{ch:semi} is self-contained except for the definition of real closed fields (Definition \ref{defn:real_closed}).

\section{Real closed fields and extensions}

We first define real closed fields and list some of their important properties.  A field $\R$ is an \textit{ordered field} if it is totally ordered and the order is compatible with the field operations, namely, for every $a,b,c\in \R$ we have 
$$\begin{aligned}
a\leq b &\implies a+c \leq b+c, \text{ and} \\
a\geq 0 \wee b\geq 0 &\implies ab \geq 0. 
\end{aligned}$$

\begin{definition}{\label{defn:real_closed}}
An ordered field $\R$ is real closed if $\R$ has no algebraic ordered extensions.
\end{definition}

There are several equivalent definitions for a field to be real closed (see \cite{BPRbook2}).

\begin{theorem}[Theorem 2.11 \cite{BPRbook2}]
Let $\R$ be a field.  Then the following are equivalent:
\alist 
\item $\R$ is real closed.
\item $\R[i]=\R[T]/(T^2+1)$ is an algebraically closed field.
\item $\R$ has the intermediate value property.
\item The non-negative elements of $\R$ are exactly squares of the form $a^2$, and every polynomial in $\R[X]$ of odd degree has a root in $\R$. 
\elist
\end{theorem}

The prototypical example of a real closed field is $\mathbb{R}$ the field of real numbers.  However, the use of an arbitrary real closed field as a base field is pervasive in real algebraic geometry. The reason for working over an arbitrary real closed field rather than $\mathbb{R}$ is two-fold.  First, for mathematical reasons it is desirable to prove results using the smallest number of assumptions as possible in order to make the results stronger.  More importantly is the fact that in the proof of many results it is extremely useful to consider ordered field extensions of the starting base field and thus it is useful for theorems to hold over arbitrary real closed fields (see Section \ref{subsec:real_extensions}).

\subsection{Real extensions, algebraic Puiseux series}{\label{subsec:real_extensions}}

For any $\R$ ordered field, a \textit{real extension} of $\R$ is an ordered field $\R'$ such that $\R\subseteq \R'$ and the ordering of $\R'$ restricted to $\R$ is just the order on $\R$ (note that if $\R$ is real closed this implies that the extension is not algebraic).  
One of the main tools we will use is a certain type of real extension, namely, the extension of a real closed field by an infinitesimal. 

Let $\R$ be a real closed field.  A \textit{Puiseux series} with coefficients in $\R$ is a formal Laurent series in the indeterminate $\ep^{1/q}$, which takes the form $\sum_{i\geq k} a_i \ep^{i/q}$, $i\in \mathbb{Z}$, $a_i\in \R$, and $q$ a positive integer.  Under the ordering $0<\ep<a$ for any $a\in \R$, the field of Puiseux series is real closed extension of $\R$ (see \cite{BPRbook2}) in which $\ep$ is an infinitesimal. We will primarily be concerned with the real closed subfield $\R\langle \ep \rangle$ of \textit{algebraic Puiseux series}, those Puiseux series which are algebraic over $\R(\ep)$.  

An important concept regarding algebraic Puiseux series which will be discussed in greater detail in Chapter \ref{ch:refined} (Section \ref{sec:puiseux}) is the notion of \textit{limits}.  A Puiseux series $\xi$ is \textit{bounded over} $\R$ if $|\xi|<a$ for some $a\in \R$.  We denote the set of Puiseux series bounded over $\R$ by $\R\langle \ep \rangle_b$. 
We define the $\lim_\ep$ map to be the ring homomorphism from $\R\langle \ep \rangle_b$ to $\R$ which sends a Puiseux series $\sum_{i\geq 0} a_i \ep^{i/q}$ to $a_0$.  The map $\lim_\ep$ simply replaces $\ep$ by $0$ in a  Puiseux series which has no terms corresponding to negative indices.


\section{Semi-algebraic sets and formulas}

We now formally define semi-algebraic sets, using the definition found in \cite{BPRbook2}.  The \textit{semi-algebraic sets} of $\R^k$ are the smallest collection of subsets of $\R^k$ defined by polynomial equalities or polynomial strict inequalites, \textit{i.e.,} sets of the form $\{x\in \R^k| \; P(x)=0\}$ or $\{x\in \R^k|\; P(x)>0\}$ respectively, which is closed under the boolean operations (finite intersections, finite unions, and taking complements).  Semi-algebraic sets can be described by \textit{formulas} involving polynomial equalities and inequalities, and to make this precise we next define what we mean by a formula.  




In the language of logic, a \textit{quantifier-free formula} is defined inductively: an \textit{atom} is a quantifier-free formula, and if $\alpha,\beta$ are two quantifier-free formulas then $\alpha \wedge \beta$, $\alpha \vee \beta$, and $\neg \alpha$ are quantifier-free formulas.  All of the variables in a quantifier-free formula are called \textit{free variables}.  In the case of semi-algebraic sets, the atoms are polynomial equalities $P=0$ and strict inequalites $P>0$, where $P\in \R[X_1,\ldots,X_k]$, and the free variables are $X_1,\ldots,X_k$.  The free variables of a formula can be quantified by either the existential quantifier $\exists X_i$ or the universal quantifier $\forall X_i$, and a free variable once quantified is no longer free but a \textit{bound variable}.  

Using elementary logic, it is clear that any formula can be expressed using only the unary connective $\neg$, the binary connective $\wedge$, and the existential quantifier $\exists$.  These logical symbols have a corresponding meaning in the geometry of semi-algebraic sets, namely the set theoretic operations of complement, intersection, and projection, respectively.  For example, 
$$\{x\in \R^k|\; P(x)\neq 0\} = \R^k \smallsetminus \{x\in \R^k|\; P(x)=0\},$$ 
$$\{x\in \R^k|\; P(x)>0 \wedge Q(x)=0\}=\{x\in \R^k|\; P(x)>0\}\cap \{x\in \R^k|\; Q(x)=0\},$$ and $$\{x'\in \R^{k-1}|\; (\exists X_k)P(x',X_k)=0\}=\pi(\{x\in \R^k|\; P(x)=0\}),$$ where $\pi:\R^k\to \R^{k-1}$ is the projection which forgets the last coordinate. It is an important consequence of the Tarski-Seidenberg Theorem that when $\R$ is real closed that the projection of a semi-algebraic set in $\R^k$ is again semi-algebraic.  Said another way, any formula defined over $\R$ with quantifiers is equivalent to a quantifier-free formula in the free variables.  So, any formula where the atoms are polynomial equalities and inequalities over a real closed field defines a semi-algebraic set.

One important type of semi-algebraic set are the \textit{basic} semi-algebraic sets.  
Let $\R$ be a real closed field and suppose $S$ is a semi-algebraic subset of $\R^k$ defined by a conjunction of equalities and strict  inequalities of $k$-variate polynomials.  Such a semi-algebraic set is called a \textit{basic semi-algebraic set}.  Since any quantifier-free formula can be put into disjunctive normal form, as a disjunction of conjunctions, any semi-algebraic set is the union of basic semi-algebraic sets.  In symbols, any set of the form 
$$\{x\in \R^k|\; P(x)=0 \wedge \bigwedge_{Q\in \mathcal{Q}} Q(x)>0\}$$ is a basic semi-algebraic set, for $\mathcal{Q}$ a finite subset of $\R[X_1,\ldots,X_k]$.  


Although semi-algebraic sets are described by formulas, it is obvious that many formulas can define the same semi-algebraic set (try to think of several formulas that define the empty set).  On the other hand, one formula can describe different semi-algebraic sets over different real closed fields if one considers field extensions.
If $S$ is a semi-algebraic subset of $\R^k$ defined by $\phi$, and $\R\subseteq \R'$ is a real closed extension, then $\Ext(S,\R')$ is the semi-algebraic subset of $\R'^k$ defined by $\phi$.





\subsection{Tame geometry}
Although it will not be 
relevant until Chapter \ref{ch:semi}, we take a moment to step back to discuss an important property which semi-algebraic geometry shares with other branches of real algebraic geometry: the property for a geometry to be \textit{tame}.  
In a precise mathematical sense, the geometry of semi-algebraic sets is tame.  Bizzare geometric phenomenon like the space filling curve or the topologists \emph{sin} curve $y=\sin(1/x)$ can not occur in tame geometries.  While many quantitative results (e.g., on the Betti numbers of semi-algebraic sets, homotopy types, etc.) do not have immediate analogues in other tame geometries, many results about semi-algebraic sets can be deduced as special cases of results which hold in all tame geometries.  

In Chapter \ref{ch:semi} we will prove a result in a more general framework than the theory of semi-algebraic sets.  In particular, we will prove a result which holds over an arbitrary o-minimal structure (see Section \ref{subsec:prelim}). 

\section{Homotopy types, Betti numbers, and homology of semi-algebraic sets}

In this section we provide the definitions which we will need in Chapters \ref{ch:refined} and \ref{ch:homotopy}.

\subsection{Homotopy types}

We recall some definitions from \cite{BPRbook2} (Section 6.3).  Let $X,Y$ be two closed and bounded semi-algebraic sets and $f,g:X\to Y$ two semi-algebraic functions (a function $f:X\to Y$ is semi-algebraic if its graph $\Gamma(f)=\{(x,y)|\; f(x)=y\}$ is a semi-algebraic subset of $X\times Y$).  We say that $f$ and $g$ are \textit{semi-algebraically homotopic} if there exists a semi-algebraic and continuous function $F:X\times [0,1] \to Y$ such that $F(x,0)=f(x)$ and $F(x,1)=g(x)$.  If there exists semi-algebraic continuous maps $f:X\to Y$ and $g:Y\to X$ such that $f\circ g$ is semi-algebraically homotopic to the identity map of $Y$ and $f\circ g$ is semi-algebraically homotopic to the identity map of $X$, then we say that $X$ and $Y$ are \textit{semi-algebraically homotopy equivalent}.  It is clear that semi-algebraic homotopy equivalence is an equivalence relation.  If $X\supseteq Y$ and there exists a semi-algebraic and continuous function $h:X\times [0,1] \to X$ such that $h(-,0)$ is the identity function of $X$ and $h(-,1)$ takes values in $Y$, then we say that $h$ is a \textit{deformation retract} of $X$ to $Y$, and in this case it is clear (using techniques of elementary topology, see, e.g., \cite{munkres2000topology}) that $X$ and $Y$ are semi-algebraically homotopy equivalent.  

\subsection{Betti numbers of semi-algebraic sets}



In this section we provide a brief introduction to the Betti numbers of semi-algebraic sets, see \cite{BPRbook2} for a detailed introduction.  The definition for the Betti numbers of closed and bounded semi-algebraic set $S$ is to use the fact that any closed and semi-algebraic semi-algebraic set can be triangulated, it is homeomorphic to a simplicial complex, and then to define the homology groups of the semi-algebraic set $S$ as the simplicial homology groups of the associated simplicial complex $K$.  For locally closed semi-algebraic sets, one defines the Betti numbers using Borel-Moore homology, which has the advantage of being additive and coincides with the simplicial complex definition of Betti numbers in the case of closed semi-algebraic sets.  In order to be precise in this exposition we provide the following definitions which already appear in \cite{BPRbook2} (see in particular Sections 5.6 and 5.7, and Chapter 6).

Let $a_0,\ldots,a_d$ be affinely independent points of $\R^k$, so that they are not contained in any $d-1$ dimensional affine subspace.  The \textit{$d$-simplex} with vertices $a_0,\ldots,a_d$ is the set $$\{\lambda_0 a_0+\ldots + \lambda_d a_d |\; \sum_{i=0}^d \lambda_i=1 \text{ and } \lambda_0,\ldots,\lambda_d\geq 0\},$$
and an \textit{$e$-face} of the $d$-simplex $[a_0,\ldots,a_d]$ is any $e$-simplex $[b_0,\ldots,b_e]$ such that $$\{b_0,\ldots,b_e\}\subseteq \{a_0,\ldots,a_d\}.$$
A \textit{simplicial complex} $K$ in $\R^k$ is a finite set of simplicies in $\R^k$ such that if $s,s'\in K$ then every face of $s$ is in $K$ and $s\cap s'$ is a common face of both $s$ and $s'$.  

It is proved in \cite{BPRbook2} (Theorem 5.43) that any closed and bounded semi-algebraic set can be triangulated.

\begin{theorem}[Theorem 5.43 \cite{BPRbook2}] \label{thm:triangulation} 
Let $S\subseteq \R^k$ be a closed and bounded semi algebraic set.  Then there exists a simplicial complex $K$ of $\R^k$ and a semi-algebraic homeomorphism $h:K\to S$, called the triangulation of $S$. 
\end{theorem}

For a closed and bounded semi-algebraic set $S$ let $K$ be as in Theorem \ref{thm:triangulation}.  We define $$
b_i(S):=b_i(K)$$ where $b_i(K)$ is the rank of the simplicial homology group $H_p(K)$ as a finite vector space over $\mathbb{Q}$ (see, for example, \cite{BPRbook2} Section 6.1 for a detailed introduction to the homology of simplicial complexes) or sometimes taking $\mathbb{Z}$, $\mathbb{Z}_2$ instead as the ring of coefficients, depending on the context in which we are working. The number of connected components, the Euler characteristic, and the sum of the Betti numbers of $S$ are thus defined respectively as $$
\begin{aligned}b_0(S)&,\\ \chi(S):=&\sum_{i=0}^k (-1)^i b_i(S),\\ b(S):=&\sum_{i=0}^k b_i(S).\end{aligned}$$

For a semi-algebraic set $S$ which is closed but not necessarily bounded, we can use the conical structure at infinity of semi-algebraic sets to define the Betti numbers of $S$.  The following corollary follows immediately from the semi-algebraic triviality theorem (see Chapter 5, Theorem \ref{thm1} for an o-minimal version of the triviality theorem and \cite{BPRbook2} for a proof of the corollary).

\begin{notation}
For any $r\in \R$, we denote by $B_k(0,r)$ semi-algebraic subset of $\R^k$ defined by the inequality $x_1^2+\ldots+x_k^2\leq r^2$.  We will slightly abuse notation at times and use this notation to refer also to the semi-algebraic subset of $\R'^k$ defined by the same inequality for other real closed fields $\R'$ and $r\in \R'$, and we hope that this does not cause any confusion. 
\end{notation}

\begin{theorem}[Corollary 5.49  \cite{BPRbook2}, Conic structure at infinity] \label{thm:conic}
Let $S$ be a closed semi-algebraic subset of $\R^k$.  Then, there exists $r\in \R$, $r>0$, such that for every $r'\geq r$, there is a semi-algebraic deformation retraction from $S$ to $S_{r'}=S\cap \overline{B_k(0,r')}$ and a semi-algebraic deformation retraction from $S_{r'}$ to $S_{r}$.  

\end{theorem}
If $S$ is a closed semi-algebraic set then we define $b(S):=b(S\cap \overline{B_k(0,r)})$ for some sufficiently large $r$ satisfying the conclusion of Theorem \ref{thm:conic}, which is well defined since homology is stable under semi-algebraic homotopy equivalence (Theorem 6.42 \cite{BPRbook2}).  

Finally, we define the Betti numbers of a certain type of semi-algebraic set which we next describe and which will be an important definition of Chapter \ref{ch:refined}.  Let ${\mathcal P}$ be a finite subset of $\R[X_1,\ldots,X_k]$.
A  {\em sign condition} on
${\mathcal P}$ is an element of $\{0,1,- 1\}^{\mathcal P}$.

The {\em realization} of the sign condition $\sigma$ in a semi-algebraic
set $V \subset \R^k$ is the semi-algebraic set
\begin{equation}
\label{eqn:R(Z)}
\RR(\sigma,V)= \{x\in V\;\mid\;
\bigwedge_{P\in{\mathcal P}} \s({P}(x))=\sigma(P) \}.
\end{equation}
We define the Betti numbers of (the realization of) sign conditions as follows (following Section 6.3 of \cite{BPRbook2}).  Consider the field $\R\langle \delta \rangle$ of algebraic Puiseux series in $\delta$, and let $\overline{\Reali}(\sigma)$ be the semi-algebraic subset of $\R\langle \delta \rangle^k$ defined by $$\sum_{1\leq i\leq k} X^2_i \leq 1/\delta \wedge \bigwedge_{\sigma(P)=0} P=0 \wedge \bigwedge_{\sigma(P)=-1} P\leq -\delta \wedge \bigwedge_{\sigma(P)=1} P\geq \delta.$$ 
Note that $\overline{\Reali}(\sigma)$ is closed and bounded.  As in \cite{BPRbook2}, we define $b(\Reali(\sigma)):=b(\overline{\Reali}(\sigma))$, which is a good definition since the set $\overline{\Reali}(\sigma)$ is a semi-algebraic deformation retraction of the extension of $\Reali(\sigma)$ to $\R\langle \delta \rangle$ (Proposition 6.45 \cite{BPRbook2}) and the homology of closed and bounded semi-algebraic sets is invariant under homotopy equivalence (Theorem 6.42 \cite{BPRbook2}) and furthermore the homology groups of $S$ and those of its extension to a bigger real closed field are isomorphic.  
For an excellent and concise exposition of homology of sign conditions and semi-algebraic sets, see \cite{BPR8}.




\section{CW complexes}
We will need a few facts from the homotopy theory of
finite CW-complexes in Chapter \ref{ch:homotopy}.

We first prove a basic result about $p$-equivalences (Definition \ref{def:p-equivalence}).
It is clear that $p$-equivalence is not an equivalence relation
(e.g., for any $p\geq 0$,
the map taking $\Sphere^p$ to a point is a $p$-equivalence,
but no map from a point into $\Sphere^p$ is one).
However, we have the following.

\begin{proposition}
\label{prop:top_basic}
Let $A,B,C$ be
finite CW-complexes
with $\dim(A),\dim(B) \leq k$ and
suppose that $C$ is $p$-equivalent to $A$ and $B$ for some $p >k$.
Then, $A$ and $B$ are homotopy equivalent.
\end{proposition}

The proof of Proposition \ref{prop:top_basic} will rely on the
following well-known lemmas.


\begin{lemma}\cite[page 182, Theorem 7.16]{Whitehead}
\label{lem:top_basic1}
Let $X,Y$ be CW-complexes  and $f:X\rightarrow Y$ a $p$-equivalence. Then,
for each CW-complex $M$, $\dim(M) \leq p$,
the induced map
\[
f_*: [M,X] \rightarrow [M,Y]
\]
is surjective.
\end{lemma}

\begin{lemma}\cite[page 69]{Viro-homotopy}
\label{lem:top_basic2}
If $A$ and $B$ are finite
CW-complexes, with $dim(A) < p$ and $dim(B) \leq p$, then
every $p$-equivalence from $A$ to $B$ is a homotopy equivalence.
\end{lemma}

\begin{proof}[Proof of Proposition \ref{prop:top_basic}]
See \cite{Barone-Basu11b}.
\hide{
Suppose $f:C\rightarrow A$ and $g:C \rightarrow B$ are two $p$-equivalences.
Applying  Lemma \ref{lem:top_basic1} with $X=C$, $M=Y=A$, we have
that the homotopy class of the identity map
$1_A$ has a preimage, $[h]$,  under $f_*$,
for some $h \in [A,C]$.
Then, for each $ a \in A$, and $i \geq 0$,
\[
f_*\circ h_* : \boldsymbol{\pi}_i(A,a) \rightarrow \boldsymbol{\pi}_i(A, f\circ h(a)),
\]
is bijective.
In particular, since $f$ is a $p$-equivalence,
this implies that $h_*: \boldsymbol{\pi}_i(A,a) \rightarrow \boldsymbol{\pi}_i(C,h(a))$ is
bijective
for $0 \leq i < p$. Composing $h$ with $g$, and noting that $g$ is
also a $p$-equivalence we get that the map
$(g \circ h)_*:\boldsymbol{\pi}_i(A,a) \rightarrow
\boldsymbol{\pi}_i(B, g\circ h (a))$ is bijective
for $0 \leq i < p$.
Now, applying Lemma \ref{lem:top_basic2} we get that $g \circ h$ is a
homotopy equivalence.}
\end{proof}

\hide{
We introduce some more notation.

\begin{notation}
For any
$\R\in \mathbb{R}_+$,
we denote by $B_k(0,R) \subset \mathbb{R}^k$,
the open ball of radius $R$ centered at the origin.
\end{notation}

\begin{notation}
For
$P \in \mathbb{R}[X_1,\ldots,X_k]$,
we denote by $\Zer(P,\mathbb{R}^k)$ the real
algebraic set defined by $P=0$.
\end{notation}

\begin{notation}{\label{not:reali}}
For any first order formula $\Phi$
with $k$ free variables, we denote by $\Reali(\Phi)$ the semi-algebraic
subset of $\mathbb{R}^k$ defined by $\Phi$.
\hide{
Additionally, if $\mathcal{P}\subset \mathbb{R}[X_1,\dots,X_k]$
consists of the polynomials appearing in $\Phi$, then we call $\Phi$ a $\mathcal{P}$-formula.
}
\end{notation}

}

\chapter{Refined bounds}{\label{ch:refined}}





In this chapter\footnote{The results of this chapter have already been described in a joint work with S. Basu \cite{Barone-Basu11a}, to which we refer for some of the proofs.}
we prove that the number of semi-algebraically connected components of sign conditions on a variety is bounded singly exponentially in terms of the degree of the polynomials defining the sign condition, the degree of the polynomials defining the variety, the dimension of the variety, and the dimension of the ambient space.

Let $\R$ be a real closed field, $\mathcal{P},\mathcal{Q} \subset
\R[X_1,\ldots,X_k]$ finite subsets of polynomials,
with the degrees
of the polynomials in $\mathcal{P}$ (resp. $\mathcal{Q}$) bounded by
$d$ (resp. $d_0$).
Let $V \subset \R^k$ be the real algebraic variety
defined by the polynomials in $\mathcal{Q}$ and suppose that the real
dimension of $V$ is bounded by $k'$. We prove that the number of
semi-algebraically connected components of the realizations of all
realizable sign conditions of the family $\mathcal{P}$ on $V$ is bounded
by
$$
\displaylines{
\sum_{j=0}^{k'}4^j{s +1\choose j}F_{d,d_0,k,k'}(j),}$$
where $s = \card \; \mathcal{P}$,
and $$F_{d,d_0,k,k'}(j)=
\textstyle\binom{k+1}{k-k'+j+1} \;(2d_0)^{k-k'}d^j\; \max\{2d_0,d \}^{k'-j}
+2(k-j+1)
.$$

In case $2 d_0 \leq d$, the above bound can be written simply as
$$
\displaylines{
\sum_{j = 0}^{k'} {s+1 \choose j}d^{k'} d_0^{k-k'}  O(1)^{k}
= (sd)^{k'} d_0^{k-k'} O(1)^k
}
$$
(in this form the bound was suggested by J. Matousek \cite{Matousek_private}).
Our result improves in certain cases (when $d_0 \ll d$)
the best known bound of
$$
\sum_{1 \leq j \leq k'}
     \binom{s}{j} 4^{j}  d(2d-1)^{k-1}
$$
on the same number proved in
\cite{BPR8}
in the case  $d=d_0$.

The distinction between the bound $d_0$ on the degrees
of the polynomials defining the variety $V$ and the bound $d$ on the degrees
of the polynomials in $\mathcal{P}$ that appears in the new bound
is motivated by several applications in discrete geometry
\cite{Guth-Katz,Matousek11b,Solymosi-Tao,Zahl}.

\section{Introduction}
Let $\R$ be a real closed field.
We denote by $\C$ the algebraic closure of
$\R$.
Let ${\mathcal P}$ be a finite subset of $\R[X_1,\ldots,X_k]$.
A  {\em sign condition} on
${\mathcal P}$ is an element of $\{0,1,- 1\}^{\mathcal P}$.

The {\em realization} of the sign condition $\sigma$ in a semi-algebraic
set $V \subset \R^k$ is the semi-algebraic set

\begin{equation}
\label{eqn:R(Z)}
\RR(\sigma,V)= \{x\in V\;\mid\;
\bigwedge_{P\in{\mathcal P}} \s({P}(x))=\sigma(P) \}.
\end{equation}

More generally, given any first order formula $\Phi(X_1,\ldots,X_k)$,
the realization of $\Phi$
in a semi-algebraic
set $V \subset \R^k$ is the semi-algebraic set

\begin{equation}
\label{eqn:R(Z)}
\RR(\Phi,V)= \{x\in V\;\mid\;
\Phi(x)\}.
\end{equation}

We  denote
the set of zeros
of ${\mathcal P}$ in $\R^k$  (resp. in $\C^k$) by
$$
\displaylines{
\ZZ({\mathcal P},\R^k)=\{x\in \R^k\;\mid\;\bigwedge_{P\in{\mathcal P}}P(x)= 0\}
}
$$

$$
\displaylines{
(\mbox{resp. }
\ZZ({\mathcal P},\C^k)=\{x\in \C^k\;\mid\;\bigwedge_{P\in{\mathcal P}}P(x)= 0\}
).
}
$$

The main problem considered in this chapter is to obtain a tight bound
on the number of semi-algebraically connected components of all realizable
sign conditions of a family of polynomials
$\mathcal{P} \subset \R[X_1,\ldots,X_k]$
in a variety $\ZZ(\mathcal{Q},\R^k)$ having
dimension $k' \leq k$, in terms of $s = \card \;\mathcal{P}, k,k'$ and
the degrees of the polynomials in $\mathcal{P}$ and $\mathcal{Q}$.

\subsection{Main result}

In this chapter we prove a bound on the number of
semi-algebraically
connected components over all realizable sign conditions of a family
of polynomials
in a variety.
However, unlike in the 
previously known bounds
the role of the degrees of the polynomials defining the variety $V$
is distinguished from the role of the degrees of the polynomials in
the family $\mathcal{P}$.
This added flexibility seems to be necessary in certain applications
of these bounds in combinatorial geometry (notably in the recent paper
by Solymosi and Tao \cite{Solymosi-Tao}). We give another application
in the theory of geometric permutations in Section \ref{sec:applications}.

Our main result is the following theorem.
\begin{theorem}
\label{thm:main}
Let $\R$ be a real closed field, and
let $\mathcal{Q},\mathcal{P} \subset \R[X_1,\dots,X_k]$ be
finite subsets of polynomials
such that $\deg(Q)\leq d_0$ for all $Q\in \mathcal{Q}$,
$\deg P =d_P$ for all $P\in \mathcal{P}$,
and the real dimension of $\Zer(\mathcal{Q},\R^k)$ is
$k' \leq k$. Suppose also that $\card\; \mathcal{P} = s$,
and for $\mathcal{I}\subseteq \mathcal{P}$
let ${d}_\mathcal{I}=\prod_{P\in \mathcal{I}}d_P$.
Then,
$$
\displaylines{
\sum_{\sigma \in \{0,1,-1\}^{\mathcal{P}}}
b_0(\RR(\sigma,\ZZ(\mathcal{Q},\R^k)))
}
$$
is at most
$$
\displaylines{
\sum_{{\mathcal{I}\subset \mathcal{P}}\atop {\#\mathcal{I}\leq k'}}4^{\#\mathcal{I}}\left(
\textstyle\binom{k+1}{k-k'+\#\mathcal{I}+1} \;(2d_0)^{k-k'}{d}_\mathcal{I}\;\max_{P\in \mathcal{I}}\{2d_0,d_P \}^{k'-\#\mathcal{I}}+2(k-\#\mathcal{I}+1)\right).
}
$$
In particular, if $d_P\leq d$ for all $P\in \mathcal{P}$,
we have that
$$
\displaylines{
\sum_{\sigma \in \{0,1,-1\}^{\mathcal{P}}}
b_0(\RR(\sigma,\ZZ(\mathcal{Q},\R^k)))
}
$$
is at most
$$
\displaylines{
\sum_{j=0}^{k'}4^j{s +1\choose j}\left(
\textstyle\binom{k+1}{k-k'+j+1} \;(2d_0)^{k-k'}d^j\;\max\{2d_0,d \}^{k'-j}
+2(k-j+1)\right).
}
$$
\end{theorem}

\subsection{A few remarks}
\begin{remark}
The bound in Theorem \ref{thm:main} is tight (up to a factor of
$ O(1)^k$).
It is instructive to examine the two extreme cases, when $k'=0$ and $k' = k-1$
respectively. When, $k'=0$, the variety $\ZZ(\mathcal{Q},\R^k)$ is zero
dimensional,
and is a union of at most $O(d_0)^k$ isolated points. The
bound in Theorem \ref{thm:main} reduces to $O(d_0)^{k}$ in this case,
and is thus tight.

When $k' = k-1$ and $2 d_0 \leq d$,
the bound in Theorem \ref{thm:main}
is equal to
$$
\sum_{j=0}^{k-1}4^j{s +1 \choose j}\left(
\binom{k+1}{j+2}\; 2d_0d^{k-1}+2(k-j+1)\right) = d_0 O(s d)^{k-1}.
$$
The following example shows that this is the best possible (again up to
$O(1)^k$).

\begin{example}
Let $\mathcal{P}$ be the set of $s$ polynomials in $X_1,\ldots,X_{k}$ each
of which is a product of $d$ generic linear forms. Let
$\mathcal{Q} = \{Q\}$, where
$$
Q = \prod_{1 \leq i \leq d_0} (X_{k} - i).
$$
It is easy to see that in this case
the number of semi-algebraically
connected components of all realizable
strict
sign conditions
of $\mathcal{P}$
(i.e.\
sign conditions  belonging to $\{-1,+1\}^{\mathcal{P}}$)
on $\ZZ(\mathcal{Q},\R^k)$ is
equal to
$$
\displaylines{
d_0 \sum_{i = 0}^{k-1} {sd \choose i} =
d_0 (\Omega(sd))^{k-1},
}
$$
since the intersection of
$
\displaystyle{
\bigcup_{P \in \mathcal{P}}\ZZ(P, \R^{k})
}
$
with the hyperplane defined by $X_{k} = i$ for each $i$, $i= 1,\ldots,d_0$, is
homeomorphic to an union of
of $s d$ generic hyperplanes in $\R^{k-1}$, and the number of
connected components of the complement of
the union of $s d$ generic hyperplanes in $\R^{k-1}$ is precisely
$\sum_{i = 0}^{k-1} {sd \choose i}$.
\end{example}
\end{remark}

\begin{remark}
Most bounds on the number of semi-algebraically connected components
of real algebraic varieties are stated in terms of the maximum of
the degrees of the polynomials defining the variety (rather than in terms
of the degree sequence). One reason behind this is the well-known fact that a
``Bezout type'' theorem is not true for real algebraic varieties. The
number of semi-algebraically connected components (indeed even isolated
zeros) of a set of polynomials $\{P_1,\ldots,P_m \} \subset
\R[X_1,\ldots,X_k,X_{k+1}]$
with degrees $d_1,\ldots,d_m$ can
be greater than the product $d_1\cdots d_m$, as can be seen in the
following example.

\begin{example}
\label{eg:counterexample}
Let $\mathcal{P} =
\{P_1,\ldots,P_m\}\subset \R[X_1,\ldots,X_{k+1}]$ be defined as follows.
$$
\displaylines{
P_1 = \sum_{i=1}^{k} \prod_{j=1}^d (X_i - j)^2 \cr
P_j = \prod_{i=1}^{m- j+2}(X_{k+1} - i), 2\leq j \leq m.
}
$$
Let $\mathcal{P}_i = \{P_1,\ldots,P_i\}$.
Notice that for each $i, 1\leq i < m$, $\ZZ(\mathcal{P}_i,\R^{k+1})$
strictly contains $\ZZ(\mathcal{P}_{i+1},\R^{k+1})$.
Moreover, $b_0(\ZZ(\mathcal{P},\R^{k+1})) = 2d^k$, while the product of the
degrees of the polynomials in $\mathcal{P}$ is $2 d m!$. Clearly,
for $d$ large enough
$2d^k > 2 d m!$.
\end{example}
\end{remark}

\begin{remark}
Most of the previously known bounds on the Betti numbers of
realizations of sign conditions relied ultimately on the
 Ole{\u\i}nik-Petrovski{\u\i}-Thom-Milnor bounds on the Betti numbers of real
varieties. Since in the proofs of these bounds
the finite family of polynomials defining a given real variety is replaced
by a single polynomial by taking a sum of squares, it is not possible
to separate
out the different roles played by the degrees of the polynomials
in $\mathcal{P}$ and those in $\mathcal{Q}$. The technique used in our exposition avoids using the  Ole{\u\i}nik-Petrovski{\u\i}-Thom-Milnor bounds, but uses directly
classically known formulas for the Betti numbers of smooth, complete
intersections in complex projective space. The bounds obtained from these
formulas depend more delicately on the individual degrees of the polynomials
involved (see Corollary
\ref{cor:bijkm}),
and this allows us to separate the roles of $d$ and $d_0$ in our
proof.
\end{remark}

\subsection{Outline of the proof of Theorem \ref{thm:main}}
The main idea behind our improved bound is to reduce the problem of bounding
the number of semi-algebraically connected components of all sign conditions
on a variety to the problem of bounding the sum of the
$\mathbb{Z}_2$-Betti
numbers of certain smooth complete intersections in complex projective
space.
This is done as follows. First assume that
$\ZZ(\mathcal{Q},\R^k)$ is bounded. The general case is reduced to
this case by an initial step using the conical triviality of semi-algebraic
sets at infinity.

Assuming that $\ZZ(\mathcal{Q},\R^k)$ is bounded, and
letting $Q = \sum_{F \in \mathcal{Q}} F^2$,
we consider another polynomial $\Def(Q,H,\zeta)$ which is an
infinitesimal perturbation of $Q$.
The basic semi-algebraic set, $T$, defined by $\Def(Q,H,\zeta) \leq 0$
is a semi-algebraic subset of $\R\langle\zeta\rangle^k$ (where
$\R\langle\zeta\rangle$ is the field of algebraic Puiseux series
with coefficients in $\R$, see Section \ref{subsec:puiseux} below for properties of
the field of Puiseux series that we need in this chapter). The
semi-algebraic set $T$ has the property
that for each semi-algebraically
connected component $C$ of $\ZZ(Q,\R^k)$ there exists a semi-algebraically
connected component $D$ of $T$, which is bounded over $\R$ and such that
$\lim_\zeta D = C$ (see Section \ref{subsec:puiseux}
for definition of $\lim_\zeta$).
The semi-algebraic set $T$ should be thought of as an
infinitesimal ``tube'' around $\ZZ(Q,\R^k)$, which is bounded by a smooth
hypersurface (namely, $\ZZ(\Def(Q,H,\zeta),\R\langle\zeta\rangle^k)$).
We then show it is possible to cut out a $k'$-dimensional
subvariety, $W$ in $\ZZ(\Def(Q,H,\zeta),\R\langle\zeta\rangle^k)$,
such that (for generic choice of co-ordinates) in fact
$\lim_\zeta W = \ZZ(Q,\R^k)$ (Proposition
\ref{prop:criticallocusontube}),
and moreover the homogenizations of
the polynomials defining $W$
define a non-singular complete intersection in $\PP_{\C\langle\zeta\rangle}^k$
(Proposition \ref{prop:general_pos_for_cr}).
$W$ is defined by
$k-k'$ forms of
degree at most $2 d_0$.
In order to bound the number of semi-algebraically connected components
of realizations of sign conditions of the family $\mathcal{P}$ on
$\ZZ(Q,\R^k)$,
we need to bound the number of semi-algebraically
connected components of the intersection
of $W$ with the zeros of certain infinitesimal perturbations of
polynomials in $\mathcal{P}$ (see Proposition \ref{prop:main} below).
The number of cases that we need to consider
is bounded by ${O(s) \choose k'}$, and
again each such set of polynomials define
a non-singular
complete intersection of
$p,\; k-k' \leq p \leq k$
hypersurfaces in $k$-dimensional
projective space  over an appropriate algebraically closed field,
$k-k'$ of which are defined
by forms having degree at most $2 d_0$ and the remaining
$m = p - k +k'$
of degree bounded by $d$.
In this situation, there are classical formulas known for the Betti numbers
of such varieties, and they imply a bound of
$\binom{k+1}{m+1}(2d_0)^{k-k'} d^{k'}+O(k)$
on the sum of the Betti numbers of such
varieties (see Corollary \ref{cor:bijkm} below).
The bounds on the sum of the Betti numbers of these projective complete
intersections in the algebraic closure imply
using the well-known Smith inequality (see
Theorem \ref{thm:smith}) a bound on the number of semi-algebraically
connected components of the real parts of these varieties, and in particular
the number of bounded components.
The product of the two bounds,
namely the combinatorial bound on the number of different
cases and the algebraic part depending on the degrees, summed
appropriately lead to the claimed bound.

\subsection{Connection to prior work}
The idea of approximating an arbitrary real variety of
dimension $k'$ by a complete intersection was used in \cite{BPR95b} to give
an efficient algorithm for computing sample points in each semi-algebraically
connected component of all realizable
sign conditions of a family of polynomials restricted to the variety.
Because of complexity issues related to algorithmically choosing a generic
system of co-ordinates however,
instead of choosing a single generic system of co-ordinates,
a finite universal family of different co-ordinate systems
was used to approximate the variety. Since in this exposition we are
not dealing with algorithmic complexity issues, we are free to choose
generic co-ordinates.
Note also that the idea of bounding the number of semi-algebraically
connected components of realizable sign conditions or of real
algebraic varieties, using known formulas for Betti numbers of non-singular,
complete intersections in complex projective spaces, and then
using Smith inequality,
have been used before
in several different
settings (see \cite{Bas05-first-Kettner} in the case of semi-algebraic
sets defined by quadrics and \cite{Benedetti-Loeser} for arbitrary real
algebraic varieties).

The rest of the chapter is organized as follows. In Section \ref{sec:prelim},
we state some known results that we will need to prove the main theorem. These
include explicit recursive
formulas for the sum of Betti numbers of non-singular, complete
intersections of complex projective varieties
(Section \ref{subsec:Betti}), the Smith inequality
relating the Betti numbers of complex varieties defined over $\R$ with those
of their real parts (Section \ref{subsec:Smith}),
some results about generic choice of co-ordinates (Sections \ref{subsec:polar},
\ref{subsec:generic}), and finally a few facts about non-archimedean
extensions and Puiseux series that we need for making perturbations
(Section \ref{subsec:puiseux}).
We prove the main theorem in Section \ref{sec:main}.

\section{Certain Preliminaries}
\label{sec:prelim}
\subsection{The Betti numbers of a non-singular complete intersection
in complex projective space}
\label{subsec:Betti}

If $\mathcal{P}$ is a finite subset of $\R[X_1,\ldots,X_k]$ consisting of homogeneous polynomials we denote
the set of zeros
of ${\mathcal P}$ in $\PP_\R^k$ (resp. in $\PP_\C^k$) by
$$
\displaylines{
\ZZ({\mathcal P},\PP_\R^k)=\{x\in \PP_\R^k\;\mid\;\bigwedge_{P\in{\mathcal P}}P(x)= 0\}
}
$$

$$
\displaylines{
(\mbox{resp. }
\ZZ({\mathcal P},\PP_\C^k)=\{x\in \PP_\C^k\;\mid\;\bigwedge_{P\in{\mathcal P}}P(x)= 0\}
).
}
$$
For $P \in \R[X_1,\ldots,X_k]$ we will denote by $\ZZ(P,\R^k)$ (
resp. $\ZZ(P,\C^k)$, $\ZZ(P,\PP_\R^k)$, $\ZZ(P,\PP_\C^k)$) the variety
$\ZZ(\{P\},\R^k)$
(
resp. $\ZZ(\{P\},\C^k)$, $\ZZ(\{P\},\PP_\R^k)$, $\ZZ(\{P\},\PP_\C^k)$).

For any locally closed semi-algebraic set $X$,
we denote by $b_i(X)$ the dimension
of
\[
\HH_i(X,\Z_2),
\]
the $i$-th homology group of $X$
with coefficients in $\Z_2$. We refer to \cite[Chapter 6]{BPRbook2}
for the definition of homology groups in case the field $\R$ is not
the field of real numbers.
Note that $b_0(X)$ equals the number of semi-algebraically connected
components of the semi-algebraic set $X$.

For $\sigma\in \{0,1,-1\}^{\mathcal P}$ and $V \subset \R^k$ a closed
semi-algebraic set, we will denote by
$b_i(\sigma,V)$ the dimension of
\[
\HH_i(\RR(\sigma,V),\Z_2).
\]

We will denote by
\[
b(\sigma,V) = \sum_{i \geq 0} b_i(\sigma,V).
\]
\begin{definition}
A projective variety $X\subset\PP^k_{\C}$ of codimension~$n$ is a
\textit{non-singular complete intersection}
if it is the intersection of $n$ non-singular
hypersurfaces
in $\mathbb{P}^k_{\C}$
that meet transversally at each point of the
intersection.
\end{definition}

Fix  an $m$-tuple of natural numbers $\bar{d} = (d_1,\ldots,d_m)$. Let
$X_{\C} = \Zer(\{Q_1,\ldots,Q_m\},\mathbb{P}_{\C}^{k})$,
such that the degree of $Q_i$ is $d_i$,
denote a complex projective variety of
codimension~$m$ which is a non-singular complete intersection.
It is a classical fact that the Betti numbers of $X_{\C}$
depend only on the degree sequence and not on the specific $X_{\C}$.
In fact, it follows from Lefshetz theorem on hyperplane sections
(see, for example, \cite[Section 1.2.2]{Voisin2})
that
\[
b_i(X_\C) = b_i(\mathbb{P}_{\C}^{k}),
\; 0 \leq i < k - m.
\]
Also, by Poinca\'re duality we have that,
\[
b_{i} (X_\C) =
b_{2(k-m) - i}(X_\C),
\; 0 \leq i \leq k - m.
\]
Thus, all the Betti numbers of $X_\C$
are determined once we know
$b_{k-m}(X_\C)$ or equivalently
the Euler-Poinca\'re characteristic
\[
\chi(X_\C)
= \sum_{i \geq 0}
(-1)^i b_i(X_\C).
\]

Denoting $\chi(X_\C)$
by $\chi^k_m(d_1,\dots,d_m)$ (since it only depends on the degree
sequence) we have the following recurrence relation
(see for example \cite{Benedetti-Loeser}).

\begin{equation}{\label{eqn:chi}}
\chi^k_m(d_1,\dots,d_m)=
\begin{cases}
k+1 &\text{ if } m=0\\
d_1 \ldots d_m &\text{ if } m=k\\
d_m\chi^{k-1}_{m-1}(d_1,\dots,d_{m-1})-(d_m-1)\chi^{k-1}_m(d_1,\dots,d_m)
&\text{ if } 0<m<k
\end{cases}
\end{equation}

We have the following inequality.

\begin{proposition}
\label{prop:ci_bound_NEW}
Suppose $1\leq d_1\leq d_2\leq \dots \leq d_m$.  The function $\chi^k_m(d_1,\ldots,d_m)$ satisfies
$$|\chi^k_m(d_1,\ldots,d_m)|\leq
\textstyle\binom{k+1}{m+1}
d_1\dots d_{m-1} d_m^{k-m+1}.$$
\end{proposition}

\begin{proof}
The proof is by induction in each of the three cases of Equation \ref{eqn:chi}.

\case{$m=0$}
$$\chi^k_0=k+1\leq \textstyle\binom{k+1}{1}=k+1  $$

\case{$m=k$}
$$\chi^k_m(d_1,\ldots,d_m)=d_1\ldots d_{m-1}d_m \leq
\textstyle\binom{k+1}{k+1}
d_1\ldots d_{m-1}d_m = d_1 \ldots d_{m-1} d_m$$

\case{$0<m<k$}
$$\begin{aligned}
  |\chi^k_m(d_1,\ldots,d_m)|=&\
  |d_m\chi^{k-1}_{m-1}(d_1,\dots,d_{m-1})-(d_m-1)\chi^{k-1}_{m}(d_1,\ldots,d_m)| \\ \leq
  &\ d_m|\chi^{k-1}_{m-1}(d_1,\dots,d_{m-1})|+d_m|\chi^{k-1}_m(d_1,\dots,d_m)| \\ \leq
  &\ d_m
  \textstyle\binom{k}{m}d_1\ldots d_{m-2}d_{m-1}^{(k-1)-(m-1)+1}+d_m
  \textstyle\binom{k}{m+1}d_1\ldots d_{m-1}d_{m}^{(k-1)-m+1}\\
  \overset{\ast}{\leq}
  &\
  \textstyle\binom{k}{m}d_1\dots d_{m-1}d_{m}^{k-m+1} +
  \textstyle\binom{k}{m+1}d_1\ldots d_{m-1}d_m^{k-m+1} \\ =
  &\ \textstyle\binom{k+1}{m+1} d_1\ldots d_{m-1}d_m^{k-m+1}
\end{aligned}$$
where the inequality $\overset{\ast}{\leq}$ follows from the
observation
$$d_{m-1}^{(k-1)-(m-1)+1} \leq d_{m-1}d_m^{k-m},$$
since $d_{m-1}\leq d_m$ by assumption, and the last equality is from the identity $\binom{k+1}{m+1}=\binom{k}{m}+\binom{k}{m+1}$.

\end{proof}

Now let $\beta^k_m(d_1,\dots,d_m)$ denote
$\sum_{i \geq 0} b_i(X_\C)$.

The following corollary is an immediate consequence of
Proposition \ref{prop:ci_bound_NEW} and the remarks preceding it.

\begin{corollary}
\label{cor:bijkm}
\begin{eqnarray*}
\beta^k_m(d_1,\dots,d_m) &\leq&
\textstyle\binom{k+1}{m+1}d_1\ldots d_{m-1}d_m^{k-m+1}+2(k-m+1)
\end{eqnarray*}
\end{corollary}

\subsection{Smith inequality}
\label{subsec:Smith}
We state a version of the
Smith inequality which plays a crucial role in the proof of the main
theorem.
Recall that for any compact topological
space equipped with an involution, inequalities derived from the {\em Smith
exact sequence} allows one to bound the {\em sum} of the Betti numbers
(with $\Z_2$ coefficients)
of the fixed point set of the involution by the sum of the Betti numbers
(again with $\Z_2$ coefficients)
of the space itself (see for instance, \cite{Viro},
p.~131).
In particular, we have for a complex projective
variety defined by real forms, with the involution taken to be
complex conjugation, the following theorem.

\begin{theorem}[Smith inequality]\label{thm:smith}
Let ${\mathcal Q} \subset \R[X_1,\ldots,X_{k+1}]$ be a family of
homogeneous polynomials.
Then,
\[
b(\Zer({\mathcal Q},\PP^k_{\R}))\le b(\Zer({\mathcal Q},\PP^k_{\textnormal{C}})).
\]
\end{theorem}

\begin{remark}
Note that we are going to use Theorem \ref{thm:smith} only for bounding the
number of semi-algebraically connected components (that is
the zero-th Betti number) of certain real varieties. Nevertheless, to
apply the inequality we need a bound on the sum of all the Betti numbers
(not just $b_0$) on the right hand side.
\end{remark}

The following theorem
used in the proof of Theorem \ref{thm:main} is a direct consequence
of Theorem \ref{thm:smith} and the bound in Corollary
\ref{cor:bijkm}.

\begin{theorem}
\label{thm:main2}
Let $\R$ be a real closed field and
$\mathcal{P} = \{P_1,\ldots,P_m\}
\subset \R[X_1,\ldots,X_k]$ with $\deg(P_i) = d_i, i=1,\ldots,m$,
and
$1\leq
d_1 \leq d_2 \leq \cdots \leq d_m$.
Let $\mathcal{P}^h = \{P_1^h,\ldots,P_m^h\}$, and
suppose
that
$P_1^h,\ldots,P_m^h$ define a non-singular complete intersection in
$\PP_\C^k$. Then,
$$
\displaylines{
b_0(\ZZ(\mathcal{P}^h,\PP_\R^k)) \leq
 \textstyle\binom{k+1}{m+1}
d_1 \cdots d_{m-1} d_m^{k-m+1}+2(k-m+1).
}
$$
In case $\ZZ(\mathcal{P},\R^k)$ is bounded,
$$
\displaylines{
b_0(\ZZ(\mathcal{P},\R^k)) \leq
\textstyle\binom{k+1}{m+1}d_1 \cdots d_{m-1} d_m^{k-m+1}+2(k-m+1).
}
$$
\end{theorem}

\begin{proof}
Proof is immediate from Theorem \ref{thm:smith} and Corollary \ref{cor:bijkm}.
\end{proof}

\begin{remark}
Note that the bound in Theorem \ref{thm:main2} is not true
if we omit the
assumption of being a non-singular complete intersection. A
counter-example is provided by Example \ref{eg:counterexample}.
\end{remark}

\begin{remark}
Another possible approach to the proof of Theorem \ref{thm:main2}
is to use the  critical point method and bound directly the number of
critical points of a generic projection
using the multi-homogeneous B\'ezout theorem (see for example
\cite{Safey}).
\end{remark}

\subsection{Generic coordinates}
\label{subsec:generic}
Unless otherwise stated, for any real closed field $\R$,
we are going to use the Euclidean topology
(see, for example, \cite[page 26] {BCR}) on $\R^k$.
Sometimes we will need to use (the coarser) Zariski topology, and
we explicitly state this whenever it is the case.

\begin{notation}
For a real algebraic set $V=\Zer(Q,\R^k)$
we let $\textnormal{reg} V$ denote the non-singular points
in dimension $\dim V$ of $V$
(\cite[Definition 3.3.9]{BCR}).
\end{notation}

\begin{definition}
Let $V=\Zer(Q,\R^k)$ be a real algebraic set.
Define $V^{(k)}=V$, and for $0\leq i \leq k-1$ define
$$V^{(i)}=V^{(i+1)}\setminus \textnormal{reg } V^{(i+1)}.$$
Set $\dim V^{(i)}=d(i)$.
\end{definition}

\begin{remark}{\label{rem:Vi}}
Note that $V^{(i)}$ is
Zariski closed for each $0\leq i \leq k$.
\end{remark}

\begin{notation}
We denote by $\textnormal{Gr}_\R(k,j)$ the real Grassmannian of $j$-dimensional
linear subspaces of  $\R^k$.
\end{notation}

\begin{notation}
For a real algebraic variety $V \subset \R^k$, and
$x \in \textnormal{reg } V $ where $\dim \textnormal{reg } V = p$, we denote by $T_x V$
the tangent space at $x$ to $V$ (translated to the origin). Note
that $T_x V$ is a $p$-dimensional subspace of $\R^k$, and hence
an element
of $\textnormal{Gr}_\R(k,p)$.
\end{notation}

\begin{definition}{\label{def:good}}
Let $V=\Zer(Q,\R^k)$ be a real algebraic set,
$1\leq j \leq k$, and
$\ell\in \textnormal{Gr}(k,k-j)$.  We say the linear space $\ell$ is
\textit{$j$-good} with respect to $V$ if either:
\blist
\item $j\notin d([0,k])$, or
\item $d(i)=j$,
and
$$
A_\ell:=\{x\in \textnormal{reg } V^{(i)}|\
\dim(\T_x V^{(i)} \cap \ell) =0\}
$$
is a non-empty dense Zariski open subset of $\textnormal{reg } V^{(i)}$.
\elist
\end{definition}

\begin{remark}{\label{rem:Aell}}
Note that the semi-algebraic subset $A_\ell$ is always a
(possibly empty) Zariski open subset of $\textnormal{reg } V^{(i)}$,
hence of $V^{(i)}$.  In the case where $V^{(i)}$ is an
irreducible Zariski closed subset (see Remark \ref{rem:Vi}),
the set $A_\ell$ is either empty or a non-empty dense
Zariski open subset of $\textnormal{reg } V^{(i)}$.
\end{remark}

\begin{definition}
Let $V=\Zer(Q,\R^k)$ be a real algebraic set and
$\mathcal{B}=\{v_1,\dots,v_k\}\subset \R^k$
a basis of $\R^k$.  We say that the basis
$\mathcal{B}$ is \textit{good} with respect to $V$
if for each $j$, $1\leq j \leq k$,
the linear space
$\textnormal{span}\{v_1,\dots,v_{k-j}\}$ is $j$-good.
\end{definition}

\begin{proposition}
{\label{prop:generic}}
Let $V=\Zer(Q,\R^k)$ be a real algebraic set and
$\{v_1,\dots,v_k\}\subset \R^k$ a basis of $\R^k$.
Then, there exists a non-empty open
semi-algebraic subset
of linear transformations
$\mathcal{O}\subset \textnormal{GL}(k,\R)$ such that for every $T\in \mathcal{O}$ the basis
$\{T(v_1),\dots,T(v_k)\}$ is good with respect to $V$.
\end{proposition}

The proof of Proposition \ref{prop:generic} uses the following
notation and lemma.

\begin{notation}
For any $\ell\in \textnormal{Gr}_\R(k,k-j)$, $1\leq j \leq k$, we denote
by $\Omega(\ell)$
the real algebraic subvariety of $\textnormal{Gr}_\R(k,j)$ defined by
$$\Omega(\ell)=\{\ell'\in \textnormal{Gr}_\R(k,j)| \ \ell \cap \ell'\neq 0\}.$$
\end{notation}

\begin{lemma}
\label{lem:empty}
For any non-empty open
semi-algebraic subset
$U\subset \textnormal{Gr}_\R(k,k-j)$, $1\leq j \leq k$,
we have
$$\bigcap_{\ell\in U} \Omega (\ell) = \emptyset.$$
\end{lemma}

\begin{proof}
See Lemma 2.19 of \cite{Barone-Basu11a}.
\hide{
We use a technique due to Chistov et al. (originally appearing in \cite{K})
who explicitly constructed a finite family of elements in
$\textnormal{Gr}_\R(k,k-j)$ such that every $\ell' \in \textnormal{Gr}_\R(k,j)$ is transversal
to at least one member of this family.
More precisely, let
$e_0,\ldots,e_{k-1}$ be the standard basis vectors in $\R^k$, and
let for any $x\in \R$,
$$
\displaylines{
v_k(x) = \sum_{i=0}^{k-1}{x^i} e_i.
}
$$

Then the set of vectors $v_{k}(x),v_k(x+1),\ldots,v_k(x+k-j-1)$ are
linearly independent and span a $(k-j)$-dimensional subspace of $\R^k$.
Denote by $\ell_x$ the corresponding element in $\textnormal{Gr}_\R(k,k-j)$. An easy
adaptation of the proof of Proposition 13.27 \cite{BPRbook2}
now shows that, for $\eps>0$, the set
\[
L_{\eps,k,j} := \{\ell_{m\eps} | 0 \leq m \leq k(k-j) \} \subset \textnormal{Gr}(k,k-j),
\]
has the property that for any $\ell' \in \textnormal{Gr}_\R(k,j)$, there exists
some $m$, $0 \leq m\leq k(k-j)$, such that $\ell' \cap \ell_{m \eps} = 0$.
In other words, for every $\eps > 0$,
$$
\displaylines{
\bigcap_{0 \leq m \leq k(k-j)} \Omega(\ell_{m\eps}) = \emptyset.
}
$$
By rotating co-ordinates we can assume that $\ell_0 \in U$, and
then by choosing $\eps$ small enough we can assume that
$L_{\eps,k,j} \subset U$. This finishes the proof.
}
\end{proof}

\begin{proof}[Proof of Proposition \ref{prop:generic}]
We prove that for each $j, 0 \leq j \leq k$,
the set of $\ell \in \textnormal{Gr}_\R(k,k-j)$ such that $\ell$ is not
$j$-good for $V$ is a semi-algebraic
subset of $\textnormal{Gr}_\R(k,k-j)$ without interior.
It then follows that its complement contains a
non-empty
open
dense
semi-algebraic
subset of $\textnormal{Gr}_\R(k,k-j)$, and hence there is a
non-empty
 open
semi-algebraic
subset
$\mathcal{O}_j \subset \textnormal{GL}(k,\R)$ such that for each $T \in \mathcal{O}_j$,
the linear space
$\textnormal{span}\{T(v_1),\dots,T(v_{k-j})\}$ is $j$-good with respect to $V$.

Let $j=d(i)$, $0\leq i \leq k$. Seeking a contradiction,
suppose that there is an open
semi-algebraic
subset
$U\subset \textnormal{Gr}_\R(k,k-j)$ such that every $\ell\in U$ is not $j$-good
with respect to $V$.
Let
$V^{(i)}_1,\dots,V^{(i)}_n$
be the
distinct
irreducible components of the Zariski closed set $V^{(i)}$.
For each $\ell \in U$,
$\ell$ is not $j$-good for some
$V^{(i)}_r$,
$1\leq r\leq n$ (otherwise $\ell$ would be $j$-good
for $V
$).
Let $U_1,\dots,U_n$ denote
the semi-algebraic sets defined by
$$U_r:=\{\ell\in U| \ \ell \text{ is not } j\text{-good for } V^{(i)}_r\}.$$
We have $U=U_1\cup \dots \cup U_n$, and $U$ is open in $\textnormal{Gr}_\R(k,k-j)$.
Hence,
for some $r$, $1\leq r \leq n$, we have $U_r$
contains an non-empty open
semi-algebraic
subset.
Replacing $U$ by this (possibly smaller) subset we have that
the set $A_\ell\cap \textnormal{reg } V^{(i)}_r$ is empty
for each $\ell \in U$ (cf. Definition \ref{def:good}, Remark \ref{rem:Aell}).
So,
$\textnormal{reg }V^{(i)}_r\subset \textnormal{reg } V^{(i)}\setminus A_\ell$
for every
$\ell \in U$,
and $$
\emptyset \neq \textnormal{reg } V^{(i)}_r\subset \bigcap_{\ell\in U} \textnormal{reg } V^{(i)}\setminus A_\ell.$$
Let
$\displaystyle{z\in \bigcap_{\ell\in U} \textnormal{reg } V^{(i)}\setminus A_\ell}$, but then
the linear space $\ell' = T_z (\textnormal{reg } V^{(i)})$
is in
$\displaystyle{\bigcap_{\ell \in U} \Omega(\ell)}$,
contradicting Lemma \ref{lem:empty}.
\end{proof}

\subsection{Non-singularity of the set critical of points of
hypersurfaces for generic projections}
\label{subsec:polar}
\begin{notation}
Let $H \in \R[X_1,\ldots,X_k]$. For $0 \leq p \leq k$, we will denote by
$\Cr_p(H)$ the set of polynomials
\[
\{
H, \frac{\partial H}{\partial X_1},\ldots, \frac{\partial H}{\partial X_p}
\}.
\]
We will denote by $\Cr_p^h(H)$ the corresponding set
\[
\{
H^h, \frac{\partial H^h}{\partial X_1},\ldots,
\frac{\partial H^h}{\partial X_p}
\}
\]
of homogenized polynomials.
\end{notation}

\begin{notation}
Let $d$ be even. We will denote by
$\Pos_{\R,d,k} \subset \R[X_1,\ldots,X_k]$
the set of non-negative polynomials in $\R[X_1,\ldots,X_k]$
of degree at most $d$.
Denoting by $\R[X_1,\ldots,X_k]_{\leq d}$
the finite dimensional
vector subspace of $\R[X_1,\ldots,X_k]$ consisting of polynomials of degree
at most $d$, we have that $\Pos_{\R,d,k}$ is a (semi-algebraic) cone in
$\R[X_1,\ldots,X_k]_{\leq d}$ with non-empty interior.
\end{notation}

\begin{proposition}
\label{prop:generic_polar}
Let $\R$ be a real closed field and $\C$ the algebraic closure of $\R$.
Let $d>0$ be even. Then there exists $H \in \Pos_{\R,d,k}$, such that
for each $p,\; 0 \leq p \leq k$, $\Cr_p^h(H)$ defines a non-singular complete
intersection in $\PP^k_\C$.
\end{proposition}

\begin{proof}[Proof]
The proposition follows from the fact that the generic polar varieties of
non-singular complex hypersurfaces are non-singular complete intersections
\cite[Proposition 3]{Bank97}, and since $\Pos_{\R,d,k}$
has non-empty interior, we can choose a generic polynomial in
$\Pos_{\R,d,k}$ having this property.
\end{proof}

\begin{remark}

The fact that generic
polar varieties of a non-singular complex variety
are non-singular complete intersections
is not true in general for higher codimension varieties,
see \cite{Bank97,Bank10}, in particular \cite[Section 3]{Bank10}.
\end{remark}

\subsection{Infinitesimals and Puiseux series}{\label{sec:puiseux}}
\label{subsec:puiseux}
In our arguments we are going to use infinitesimals and
non-archimedean extensions of a given real closed field $\R$.  A
typical non-archimedean extension of $\R$ is the field $\R\la\eps\ra$
of algebraic Puiseux series with coefficients in $\R$ , which coincide
with the
germs of semi-algebraic continuous functions (see \cite{BPRbook2},
Chapter 2, Section 6 and Chapter 3, Section 3).  An element $x\in \R\la
\eps\ra$ is bounded over $\R$ if $\vert x \vert \le r$ for some $0\le
r \in \R$.  The subring $\R\la\eps\ra_b$ of elements of $\R\la\eps\ra$
bounded over $\R$ consists of the Puiseux series with non-negative
exponents.  We denote by $\lim_{\varepsilon}$ the ring homomorphism
from~$\R \langle \varepsilon \rangle_b$ to $\R$ which maps $\sum_{i
  \in \mathbb{N}} a_i \varepsilon^{i / q}$ to $a_0$. So, the mapping
$\lim_{\varepsilon}$ simply replaces $\varepsilon$ by $0$ in a bounded
Puiseux series.  Given $S\subset \R\la \eps \ra^k$, we denote by
$\lim_\eps(S)\subset \R^k$ the image by $\lim_\eps$ of the elements of
$S$ whose coordinates are bounded over $\R$.
We denote by $\R\langle \ep_1,\ep_2,\ldots,\ep_\ell\rangle$ the real closed field $\R\langle \ep_1 \rangle \langle \ep_2 \rangle \cdots \langle \ep_\ell \rangle$, and we let $\lim_{\ep_i}$ denote the ring homomorphism $\lim_{\ep_i}\lim_{\ep_{i+1}}\cdots \lim_{\ep_{\ell}}$.

More generally, let $\R'$ be a real closed field extension of $\R$.
If $S\subset \R^ k$ is a semi-algebraic set, defined by a boolean
formula $\Phi$ with coefficients in $\R$, we denote by $\Ext(S,\R')$
the extension of $S$ to $\R'$, i.e.\ the semi-algebraic subset of
$\R'^k$ defined by $\Phi$.  The first property of $\Ext(S,\R')$ is
that it is well defined, i.e.\ independent on the formula $\Phi$
describing $S$ (\cite{BPRbook2} Proposition 2.87).  Many properties of
$S$ can be transferred to $\Ext(S,\R')$: for example $S$ is non-empty
if and only if $\Ext(S,\R')$ is non-empty, $S$ is semi-algebraically
connected if and only if $\Ext(S,\R')$ is semi-algebraically connected
(\cite{BPRbook2} Proposition 5.24).

\section{Proof of the Main Theorem}
\label{sec:main}

\begin{remark}
Most of the techniques employed in the proof of the main theorem are similar to those found in \cite{BPRbook2}, see \cite[Section 13.1 and Section 13.3]{BPRbook2}.
\end{remark}

Throughout this section, $\R$ is a real closed field,
$\mathcal{Q},\mathcal{P}$ are finite subsets
of $\R[X_1,\ldots,X_k]$, with
$\deg P= d_P$
for all $P \in \mathcal{P}$, and
$\deg(Q) \leq d_0$ for all $Q \in \mathcal{Q}$.
We
denote by
$k'$ the real dimension of $\ZZ(\mathcal{Q},\R^k)$.
Let $Q = \sum_{F \in \mathcal{Q}} F^2$.

For $x \in \R^k$ and $r > 0$, we will denote by $B_k(0,r)$ the open ball
centred at $x$ of radius $r$. For any semi-algebraic subset $X \subset \R^k$,
we denote by $\overline{X}$ the closure of $X$ in $\R^k$. It follows from
the Tarski-Seidenberg transfer principle (see for example
\cite[Ch 2, Section 5]{BPRbook2}) that the closure
of a semi-algebraic set is again semi-algebraic.

We suppose
using Proposition \ref{prop:generic} that
after making a
linear change in co-ordinates if necessary
the given system of co-ordinates
is good
with respect to $\ZZ(Q,\R^k)$.

Using Proposition \ref{prop:generic_polar},
suppose that $H \in \Pos_{\R,2d_0,k}$ satisfies
\begin{property}
\label{prop:H}
for any $p, 0 \leq p \leq k$, $\Cr_p^h(H)$ defines a non-singular complete
intersection in $\PP^k_\C$.
\end{property}

Let $\Def(Q,H,\zeta)$ be defined by
\[
\Def(Q,H,\zeta) = (1 -\zeta)Q - \zeta H.
\]

We first prove several properties of the polynomial $\Def(Q,H,\zeta)$.

\begin{proposition}
\label{prop:limitofdef}
Let
$\tilde{\R}$ be any real closed field containing $\R\langle 1/\Omega\rangle$, and let
$C$ be a semi-algebraically connected component of
$$
\ZZ(Q,\tilde{\R}^k)  \cap B_k(0,\Omega).
$$
Then, there exists a semi-algebraically connected component,
$D \subset \tilde{\R}\langle \zeta\rangle^k$  of the semi-algebraic set
$$W = \{ x \in B_k(0,\Omega) \;\mid\; \Def(Q,H,\zeta)(x) \leq 0 \}$$
such that $\overline{C} = \lim_\zeta D$.
\end{proposition}

\begin{proof}
It is clear that $\ZZ(Q,\tilde{\R}\langle \zeta\rangle^k) \cap B_k(0,\Omega)
\subset W$,
since
$H(x) \geq 0$ for all $x\in \tilde{\R}\langle \zeta\rangle^k$. Now let $C$ be
a semi-algebraically connected component of
$\ZZ(Q,\tilde{\R}^k) \cap B_k(0,\Omega)$
and let $D$ be the
semi-algebraically connected component of $W$ containing
$\Ext(C,\tilde{\R}\langle\zeta\rangle)$.
Since $D$ is bounded over $\tilde{\R}$
and semi-algebraically connected
we have $\lim_\zeta D$ is semi-algebraically
connected (using for example Proposition 12.43 in \cite{BPRbook2}),
and contained in $\ZZ(Q,\tilde{\R}^k)
\cap \overline{B_k(0,\Omega)}$.
Moreover, $\lim_\zeta D$ contains $\overline{C}$.
But $\overline{C}$ is a semi-algebraically
connected component of $\ZZ(Q,\tilde{\R}^k) \cap
\overline{B_k(0,\Omega)}$
(using the
conical structure at infinity of $\ZZ(Q,\tilde{\R}^k)$),
and hence $\lim_\zeta D = \overline{C}$.
\end{proof}

\begin{proposition}
\label{prop:criticallocusontube}
Let
$\tilde{\R}$ be any real closed field containing $\R\langle 1/\Omega\rangle$, and
$$W = \ZZ(\Cr_{k-k'-1}(\Def(Q,H,\zeta)), \tilde{\R}\langle \zeta\rangle^k)
\cap B_k(0,\Omega).
$$
Then, $\lim_\zeta W = \ZZ(Q,\tilde{\R}^k) \cap \overline{B_k(0,\Omega)}$.
\end{proposition}

We will use the following notation.
\begin{notation}
For $ 1\leq p \leq q < k$,
we denote by  $\pi_{[p,q]}: \R^k=\R^{[1,k]}\rightarrow \R^{[p,q]}$ the projection
$$(x_1, \ldots, x_k)\mapsto (x_p, \ldots, x_q).$$
\end{notation}

\begin{proof}[Proof of Proposition \ref{prop:criticallocusontube}]
By Proposition \ref{prop:limitofdef} it is clear that
$\lim_\zeta W \subset \ZZ(Q,\tilde{\R}^k) \cap \overline{B_k(0,\Omega)}$.
We prove the other inclusion.

Let $V = \ZZ(Q,\tilde{\R}^k)$,
and suppose that $x \in \textnormal{reg } V^{(i)} \cap B_k(0,\Omega)$ for
some $i, k-k' \leq i \leq k$.
Every open
semi-algebraic
neighbourhood $U$ of $x$ in
$V \cap B_k(0,\Omega)$
contains a point $y \in \textnormal{reg } V^{(i')} \cap B_k(0,\Omega)$
for some $i' \geq i$, such that
that the local dimension of $V$ at
$y$ is equal to $d(i')$.
Moreover, since the given system of co-ordinates is assumed to be good for
$V$,
we can also assume that
the tangent space
$T_{y}(\textnormal{reg }\; V^{(i')})$
is transverse to the span of the first $k-
d(i')$ co-ordinate
vectors.

It suffices to prove that there exists $z \in W$ such that $\lim_\zeta z = y$.
If this is  true for every neighbourhood $U$ of $x$ in $V$, this
would imply that $x \in \lim_\zeta W$.

Let $p = d(i')$.
The property that
$T_{y}(\textnormal{reg }\; V^{(i')})$
is transverse to the span of the first $k-p$ co-ordinate
vectors implies that $y$ is an isolated point of
$V \cap \pi_{[k-p+1,k]}^{-1}(y)$.
Let $T \subset \tilde{\R}\langle \zeta\rangle^k$
denote the semi-algebraic subset of $B_k(0,\Omega)$
defined by
$$
T = \{ x \in B_k(0,\Omega)\; |\; \Def(Q,H,\zeta)(x) \leq 0 \},
$$
and
$D_y$ denote the semi-algebraically connected component of
$T \cap  \pi_{[k-p+1,k]}^{-1}(y)$ containing $y$. Then,
$D_y$ is a closed and bounded semi-algebraic set, with
$\lim_\zeta D_y = y$.
The boundary of $D_y$ is contained in
\[\ZZ(\Def(Q,H,\zeta),\tilde{\R}\langle \zeta\rangle^k) \cap \pi_{[k-p+1,k]}^{-1}(y).
\]
Let $z \in D_y$ be a point in $D_y$ for which the $(k-p)$-th
co-ordinate achieves
its maximum. Then,
$z \in \ZZ(\Cr_{k-p-1}(\Def(Q,H,\zeta)),\tilde{\R}\langle \zeta\rangle^k)$,
and since $p \leq k'$,
$$\ZZ(\Cr_{k-p-1}(\Def(Q,H,\zeta)),\tilde{\R}\langle \zeta\rangle^k)
\subset
\ZZ(\Cr_{k-k'-1}(\Def(Q,H,\zeta)),\tilde{\R}\langle \zeta\rangle^k)$$
and
hence, $z \in W$. Moreover,
$\lim_\zeta z = y$.
\end{proof}

\begin{proposition}
\label{prop:general_pos_for_cr}
Let $\tilde{\R}$ be any real closed field containing $\R$ and $\tilde{\C}$ the algebraic closure of $\tilde{\R}$.
For for every $p, 0 \leq p \leq k$, $\Cr_p^h(\Def(Q,H,\zeta))$
defines a non-singular complete intersection in
$\PP^k_{\tilde{\C}\langle \zeta\rangle}$.
\end{proposition}

\begin{proof}
By Property \ref{prop:H} of $H$,
we have that
for each $p,\; 0 \leq p \leq k$, $\Cr_p^h(H)$
defines a non-singular complete intersection in
$\PP^k_{\tilde{\C}}$.
Thus, for each $p,\; 0 \leq p \leq k$, $\Cr_p^h(\Def(Q,H,1))$
defines a non-singular complete intersection in
$\PP^k_{\tilde{\C}}$.
Since the property of being non-singular complete intersection
is first order expressible, the set of
$t \in \tilde{\C}$
for which this holds
is constructible, and since the property is also stable there is an open
subset containing $1$ for which it holds. But since a constructible subset of
$\tilde{\C}$
is either finite or co-finite, there exists
an open interval to the right of $0$ in
$\tilde{\R}$ for which the property holds,
and in particular it holds for infinitesimal $\zeta$.
\end{proof}

\begin{proposition}
\label{prop:nonstrict-to-strict}
Let $\sigma \in \{0,+1,-1\}^{\mathcal{P}}$, and let
let $C$ be a semi-algebraically connected component of
$\RR(\sigma,\ZZ(Q,\R\langle 1/\Omega \rangle^k) \cap B_k(0,\Omega))$.
Then, there exists a unique semi-algebraically connected component,
$D \subset \R\langle 1/\Omega,\eps,\delta\rangle^k$,
of the semi-algebraic set defined by
$$
\displaylines{
(Q = 0)
\wedge
\bigwedge_{P \in \mathcal{P}\atop \sigma(P)=0}
(-\delta  < P < \delta)
\wedge
\bigwedge_{P \in \mathcal{P}\atop \sigma(P)=1} (P  > \eps) \wedge
\bigwedge_{P \in \mathcal{P}\atop \sigma(P)=-1} (P <  -\eps)
\wedge
(|X|^2 < \Omega^2)
}
$$
such that
$C = D \cap \R
\langle 1/\Omega \rangle
^k$.
Moreover, if $C$ is a semi-algebraically connected component of
$\RR(\sigma,\ZZ(Q,\R\langle 1/\Omega \rangle^k)\cap B_k(0,\Omega))$,
$C'$ is a semi-algebraically connected component of
$\RR(\sigma',\ZZ(Q,\R\langle 1/\Omega \rangle^k) \cap B_k(0,\Omega))$, and
$D,D'$ are the unique semi-algebraically connected components as above satisfying
$C=D\cap \R\langle 1/\Omega \rangle^k,\; C'=D'\cap \R\langle 1/\Omega \rangle^k$,
then we have $\overline{D}\cap \overline{D'}=\emptyset$ if $C\neq C'$.

\end{proposition}
\begin{proof}
The first part is clear.  To prove the second part suppose, seeking a contradiction, that $x\in \overline{D}\cap \overline{D'}$.  Notice that $\lim_{\delta} x \in \Ext(C,\R\langle 1/\Omega ,\ep \rangle)\cap \Ext(C',\R\langle 1/\Omega,\ep)$, but $\Ext(C,\R\langle 1/\Omega,\ep \rangle)\cap \Ext(C',\R\langle 1/\Omega,\ep\rangle)=\emptyset$ since $C\cap C'=\emptyset$.
\end{proof}

We denote by $\R'$ the real closed field
$\R\langle 1/\Omega,\eps,\delta\rangle$
and by $\C'$ the algebraic closure of $\R'$.

We also denote
\begin{equation}
\label{eqn:G}
G = \sum_{i=1}^{k} X_i^2 - \Omega^2.
\end{equation}

Let
$\mathcal{P}' \subset \R'[X_1,\ldots,X_k]$
be defined by
$$
\displaylines{
\mathcal{P}' =
\bigcup_{P \in \mathcal{P}} \{P\pm \eps, P\pm \delta\} \cup \{G\}.
}
$$

By Proposition \ref{prop:nonstrict-to-strict} we will henceforth restrict
attention to strict sign conditions on the family $\mathcal{P}'$.

Let
$\mathcal{H} =
(H_F \in \Pos_{\R',\deg(F),k}
)_{F \in \mathcal{P}'} $
be a family of polynomials with generic coefficients.
More precisely, this means that that $\mathcal{H}$ is chosen so that it
avoids a certain Zariski
closed subset of the product
$\displaystyle{
\times_{F\in \mathcal{P}'} \Pos_{\R'\langle\zeta\rangle,\deg(F),k}
}
$
of codimension at least one, defined by the condition that
$$
\Cr_{k-k'-1}^h(\Def(Q,H,\zeta)) \cup
\bigcup_{F \in \mathcal{P}''}\{H_F^h\}
$$
is not a non-singular, complete intersection in
$\PP_{\C'\langle\zeta\rangle}^k$ for some
${\mathcal P}'' \subset \mathcal{P}'$.

\begin{proposition}
\label{prop:general_pos}
For each $j, \; 0 \leq  j \leq k'$,
and subset ${\mathcal P}'' \subset \mathcal{P}'$ with
$\card \;\mathcal{P}'' = j$,
and $\tau \in \{-1,+1\}^{\mathcal{P}''}$
the set of homogeneous polynomials
$$
\Cr_{k-k'-1}^h(\Def(Q,H,\zeta)) \cup
\bigcup_{F \in \mathcal{P}''}\{(1 -\eps') F^h  - \tau(F)\; \eps'\; H_F^h\}
$$
defines a non-singular, complete intersection in
$\PP_{\C'\langle\zeta,\eps'\rangle}^k$.
\end{proposition}

\begin{proof}
Consider the family of polynomials,
$$
\displaylines{
\Cr_{k-k'-1}^h(\Def(Q,H,\zeta)) \cup
\bigcup_{F \in \mathcal{P}''}\{(1 - t) F^h - \tau(F)\; t\; H_F^h\}
}
$$
obtained by substituting $t$ for $\eps'$ in the given system.
Since, by Proposition \ref{prop:general_pos_for_cr}
the set $\Cr_{k-k'-1}^h(\Def(Q,H,\zeta))$ defines
a non-singular complete intersection in $\PP_{\C'\langle \zeta \rangle}$, and the $H_F$'s are
chosen generically,
the above system defines a non-singular complete intersection in
$\PP_{\C'\langle \zeta \rangle}$ when $t=1$. The set of $t \in \C'\langle \zeta \rangle$  for which the above system
defines a non-singular, complete intersection is constructible, contains $1$,
and since being a non-singular, complete intersection is a stable condition,
it is co-finite. Hence, it must contain an
open interval to the right of 0 in $\R'\langle \zeta \rangle$, and hence in particular if we substitute the infinitesimal $\eps'$ for
$t$ we obtain that the system defines a non-singular, complete intersection
in $\PP_{\C'\langle\zeta, \eps'\rangle}^k$.
\end{proof}

\begin{proposition}
\label{prop:main}
Let $\tau \in \{+1,-1\}^{\mathcal{P}'}$ with
$\tau(F) = -1$, if $F = G$
(Eqn. \ref{eqn:G}),
and
let $C$ be a semi-algebraically
connected component of
$\RR(\tau,\ZZ(Q,\R'^k))$.

Then, there exists a
a subset
$\mathcal{P}'' \subset \mathcal{P'}$ with
$\card \;\mathcal{P}''  \leq k'$,
and
a
bounded
semi-algebraically connected component $D$ of
the algebraic set
$$
\displaystyle{
\ZZ(\Cr_{k-k'-1}(\Def(Q,H,\zeta)) \cup
\bigcup_{F \in \mathcal{P}''}\{(1 -\eps') F - \tau(F)\; \eps'\; H_F\},
\R'\langle\zeta,\eps'\rangle^k)
}
$$
such that
$\lim_\zeta D \subset \overline{C}$,
and $\lim_\zeta D \cap C\neq \emptyset$.
\end{proposition}

\begin{proof}
Let $V = \ZZ(Q,\R'^k)$.
By Proposition \ref{prop:criticallocusontube}, with $\tilde{\R}$ substituted by $\R'$, we have that,
for each $x \in C$ there exists
$y \in \ZZ(\Cr_{k-k'-1}(\Def(Q,H,\zeta), \R'\langle \zeta\rangle^k)$ such that
$\lim_\zeta y = x$.

Moreover, using the fact that
$\RR(\tau,\R'^k)$ is open we have that
\[
y \in \ZZ(\Cr_{k-k'-1}(\Def(Q,H,\zeta)), \R'\langle \zeta\rangle^k)
\cap \RR(\tau,\R'^k).
\]

Thus, there exists a semi-algebraically
connected component $C'$
of
$$\ZZ(\Cr_{k-k'-1}(\Def(Q,H,\zeta)), \R'\langle\zeta\rangle^k)
\cap \RR(\tau,\R'^k),
$$
such that
$\lim_\zeta C' \subset \overline{C}$,
and $\lim_\zeta C' \cap C \neq \emptyset$.

Note that the closure $\overline{C}$ is a semi-algebraically
connected component of
\[
\RR(\overline{\tau},\ZZ(Q,\R'^k)),
\]
where $\overline{\tau}$ is the formula
$\wedge_{F \in \mathcal{P}'} (\tau(F) F \geq 0)$.

The proof of the proposition now follows the proof of
Proposition 13.2 in \cite{BPRbook2}, and uses
the fact that the set of  polynomials
$\bigcup_{F \in \mathcal{P}'}\{(1 -\eps') F - \tau(F)\eps' H_P\}$
has the property that, no $k'+1$ of distinct elements of this set
can have a common zero in
$\ZZ(\Cr_{k-k'-1}(\Def(Q,H,\zeta)), \R'\la\zeta,\eps'\ra^k)$
by Proposition \ref{prop:general_pos}.
\end{proof}

\begin{proof}[ Proof of Theorem \ref{thm:main}]
Using the conical structure at infinity of semi-algebraic sets, we have the following equality,
$$
\begin{aligned}
&\sum_{\sigma \in \{-1,1,0\}^\mathcal{P}} b_0 (\RR(\sigma,\Zer(\mathcal{Q},\R^k)))\\
=&\sum_{\sigma \in \{-1,1,0\}^\mathcal{P}} b_0 (\RR(\sigma,\Zer(\mathcal{Q},\R\langle 1/\Omega \rangle^k)\cap B_k(0,\Omega))).
\end{aligned}
$$

By Proposition \ref{prop:nonstrict-to-strict}
it suffices to bound the number of semi-algebraically
connected components of the realizations
$\RR(\tau, \ZZ(Q,\R'^k)\cap B_k(0,\Omega))$,
where $\tau \in \{-1,1\}^{\mathcal{P}'}$ satisfying
\begin{eqnarray}
\label{eqn:tau}
\tau(F) &=& -1, \mbox{ if } F = G, \nonumber \\
\tau(F) &=&  1   \mbox{ if } F = P - \eps \mbox{ or } F = P + \delta \mbox{ for some } P \in \mathcal{P},\\
\tau(F) &=& -1, \mbox { otherwise}. \nonumber
\end{eqnarray}

Using Proposition \ref{prop:main} it suffices to bound the number
of semi-algebraically connected components which are bounded over
$\R'$ of the real algebraic sets
$$
\displaystyle{
\ZZ(\Cr_{k-k'-1}(\Def(Q,H,\zeta)) \cup
\bigcup_{F \in \mathcal{P}''}\{(1 -\eps') F + \tau(F) \eps' H_F\},
\R'\langle\zeta,\eps'\rangle^k)
}
$$
for all $\mathcal{P}'' \subset \mathcal{P}'$ with
$\card \;\mathcal{P''} =j \leq k'$ and all
$\tau \in \{-1,1\}^{\mathcal{P}''}$
satisfying Eqn. \ref{eqn:tau}.

Using Theorem \ref{thm:main2}
we get that the number of such
components is
bounded by
\[
\textstyle\binom{k+1}{k-k'+\#\mathcal{P}''+1}\; (2 d_0)^{k-k'}\;
{d}_{\mathcal{P}''}\;\max_{F\in \mathcal{P}''}\{ 2d_0,\deg(F) \}^{k'-\#\mathcal{P}''}+2(k-\#\mathcal{P}''+1),
\]
where ${d}_{\mathcal{P}''} = \prod_{F \in \mathcal{P}''} \deg(F)$.

Each $F\in \mathcal{P}'\setminus \{G\}$ is of the form $F\in \{P\pm \ep,P\pm \delta\}$ for some $P \in \mathcal{P}$, and the algebraic sets defined by each of these four polynomials are disjoint.
Thus, we have that
$$\displaylines{
\sum_{\sigma \in \{0,1,-1\}^{\mathcal{P}}} b_0(\RR(\sigma,\ZZ(Q,\R^k)))
}
$$
is bounded by
$$
\displaylines{
\sum_{{\mathcal{I}\subset \mathcal{P}}\atop \#\mathcal{I}\leq k'}
4^{\#\mathcal{I}}\left(\textstyle\binom{k+1}{k-k'+\#\mathcal{I}+1}  \; (2 d_0)^{k-k'}\; {d}_{\mathcal{I}}\; \max_{P\in \mathcal{I}}\{2d_0,d_P \}^{k'-\#\mathcal{I}}  +2(k-\#\mathcal{I}+1)\right).
}$$
\end{proof}

\section{Applications}
\label{sec:applications}
There are several applications of the bound on the number of semi-algebraically
connected components of sign conditions of a family of real polynomials in
discrete geometry.
We discuss below
an application for bounding the number of geometric
permutations of $n$ well separated convex bodies in $\re^d$ induced
by $k$-transversals.

In \cite{GPW96} the authors reduce the problem
of bounding
the number of geometric
permutations of $n$ well separated convex bodies in $\re^d$ induced
by $k$-transversals
to bounding the number
of semi-algebraically connected components realizable sign conditions
of
\[
{2^{k+1} -2 \choose k}{n \choose k+1}
\]
polynomials in $d^2$ variables,
where each polynomial has degree at most $2k$,
on an algebraic variety (the real Grassmannian of $k$-planes in
$\re^d$) in
$\re^{d^2}$ defined by polynomials of degree $2$. The real Grassmanian
has dimension $k(d-k)$. Applying Theorem \ref{thm:main} we obtain that the
number of semi-algebraically connected components of all realizable
sign conditions in this case is bounded by
$$
\displaylines{
\left (k {2^{k+1} -2 \choose k}{n \choose k+1}\right)^{k(d-k)}\; (O(1))^{d^2},
}
$$
which is a strict improvement of the bound,
$$
\displaylines{
\left ({2^{k+1} -2 \choose k}{n \choose k+1}\right)^{k(d-k)}\; (O(k))^{d^2},
}
$$
in \cite[Theorem 2]{GPW96} (especially in the case
when $k$ is close to $d
)$.

As mentioned in the introduction our bound
might also have some relevance in a new method which has been
developed for  bounding the number of incidences between points and
algebraic varieties of constant degree, using a decomposition
technique
based on the polynomial ham-sandwich cut theorem
\cite{Guth-Katz,Matousek11b,Solymosi-Tao,Zahl}.



\chapter{Homotopy types of limits and additive complexity}{\label{ch:homotopy}}

\newcommand{\constM}{(p+1)(k+a+2)+2k\binom{p+1}{2}}
\newcommand{\constMprime}{(p+1)(s+2)+3\binom{p+1}{2}+3}


In tame geometries, sets of bounded description complexity have several finiteness properties including having finitely many connected components, finitely many homeomorphism types, etc.  While the fact that these numbers are finite follows from Hardt's triviality theorem for some notions of description complexity (e.g., dense format), bounds on the number of connected components or the number of homeomorphism types that can occur do not follow from Hardt's theorem for some notions of description complexity.

In this chapter\footnote{The results of this chapter have already been described in a joint work with S. Basu \cite{Barone-Basu11b}, to which we refer for some of the proofs.}
we prove that the number of
distinct
homotopy types of limits of one-parameter
semi-algebraic families of closed and bounded semi-algebraic sets is
bounded singly exponentially in the additive complexity of any
quantifier-free first order formula defining the family.  As an
important consequence, we derive that the number of
distinct
homotopy types of
semi-algebraic subsets of $\mathbb{R}^k$ defined by a quantifier-free first
order formula $\Phi$, where the sum of the additive complexities of
the polynomials appearing in $\Phi$ is at most $a$, is bounded by
$2^{(k+a)^{O(1)}}$.
This proves a conjecture made in \cite{BV06}.

\hide{
Roughly speaking the additive complexity of a
polynomial (see Definition \ref{def:rational_additive} below for a
precise definition) is bounded from above by the number of additions
in any straight line program (allowing divisions) that computes the
values of the polynomial at generic points of $\mathbb{R}^n$.
This measure of complexity strictly generalizes the more familiar
measure of complexity of real polynomials
based on counting the number of monomials in the support
(as in Khovanski\u{\i}'s  theory of ``Fewnomials'' \cite{Kho}),
and is thus of considerable interest in quantitative real
algebraic geometry.
Additive complexity of real univariate polynomials was first considered
in the context of computational complexity theory by Borodin and Cook
\cite{Borodin-Cook76}, who proved an effective bound on the number
of real zeros of an univariate
polynomial in terms of its additive
complexity. This result was further improved upon by
Grigoriev \cite{Grigoriev82}
and Risler \cite{Risler85} who applied
Khovanski\u{\i}'s  results on fewnomials \cite{Kho}.
A surprising
fact conjectured in \cite{BRbook}, and proved by Coste \cite{CosteFew}
and van den Dries \cite{Dries}, is that the number of topological
types of real algebraic varieties defined by polynomials of bounded
additive complexity is finite. \footnote{Note that what we call
  ``additive complexity'' is called ``rational additive complexity''
  in \cite{BV06}, and what we call ``division-free additive
  complexity'' is called ``additive complexity'' there.}
}

In order to state our result
precisely, we need a few preliminary definitions.

\begin{definition}
\label{def:additive}
The \emph{division-free additive complexity} of a polynomial is a
non-negative integer,
and we say that
a polynomial $P \in \mathbb{R}[X_1, \ldots ,X_k]$ has \emph{division-free additive
complexity
at most $a$},
$a\geq 0$, if there are polynomials $Q_1, \ldots , Q_a \in \mathbb{R}[X_1,
  \ldots ,X_k]$ such that
\begin{itemize}
\item[(i)]
$Q_1=u_1X_{1}^{\alpha_{11}} \cdots X_{k}^{\alpha_{1k}} +
v_1X_{1}^{\beta_{11}} \cdots X_{k}^{\beta_{1
k}}$,\\
where $u_1, v_1 \in \mathbb{R}$, and
$\alpha_{11}, \ldots ,\alpha_{1k}, \beta_{11}, \ldots , \beta_{1k} \in \N$;

\item[(ii)]
$Q_j=u_jX_{1}^{\alpha_{j1}} \cdots X_{k}^{\alpha_{j k}}
\prod_{1 \le i \le j-1}Q_{i}^{\gamma_{j i}} +
v_jX_{1}^{\beta_{j1}} \cdots X_{k}^{\beta_{j k}}\prod_{1 \le i \le j-1}Q_{i}^{\delta_{ji}}$,\\
where $1 < j \le a$, $u_j, v_j \in \mathbb{R}$, and
$\alpha_{j1}, \ldots ,\alpha_{j k}, \beta_{j1}, \ldots , \beta_{j k},
\gamma_{ji}, \delta_{ji} \in \N$ for $1 \le i <j$;

\item[(iii)]
$P= cX_{1}^{\zeta_{1}} \cdots X_{k}^{\zeta_{k}}\prod_{1 \le j \le a}Q_{j}^{\eta_{j}}$,\\
where $c \in \mathbb{R}$, and $\zeta_1, \ldots , \zeta_k, \eta_1, \ldots ,\eta_a \in \N$.
\end{itemize}
In this case, we say that the above sequence of equations
is a \emph{division-free additive representation} of $P$ of length $a$.
\end{definition}

In other words, $P$ has division-free additive complexity at most $a$ if
there exists a straight line program which,
starting with variables
$X_1, \ldots ,X_m$ and constants in $\mathbb{R}$ and
applying additions and multiplications,
computes $P$ and
which uses at most $a$ additions
(there is no bound on the number of multiplications).
Note that
the additive complexity of a polynomial (cf. Definition \ref{def:rational_additive})
is clearly at most its
division-free additive complexity, but can be much smaller (see
Example \ref{ex:rational_additive} below).

\begin{example}
\label{ex:additive}
The polynomial $P:=(X+1)^d \in \mathbb{R}[X]$ with $0<d \in \Z$, has $d+1$ monomials when expanded but division-free additive complexity at most 1.
\end{example}

\begin{notation}
We denote by
$\mathcal{A}^{\mathrm{div-free}}_{k,a}$ the family of ordered (finite) lists
${\mathcal P}=(P_1, \ldots , P_s)$ of polynomials $P_i \in \mathbb{R}[X_1,\ldots,X_k]$,
with the division-free additive complexity of every $P_i$ not exceeding $a_i$, with
$a=\sum_{1 \le i \le s}a_i$.
Note that $\mathcal{A}^{\mathrm{div-free}}_{k,a}$ is allowed to contain lists
of different sizes.
\end{notation}

Suppose that $\phi$ is a Boolean formula with atoms
$\{p_i,q_i,r_i \mid 1 \leq i \leq s\}$.
For an ordered list ${\mathcal P} = (P_1,\ldots,P_s)$ of polynomials
$P_i \in {\mathbb{R}[X_1,\ldots,X_k]}$, we denote by $\phi_{\mathcal P}$
the formula obtained from $\phi$ by replacing
for each $i,\ 1\leq i \leq s $, the atom $p_i$
(respectively, $q_i$ and $r_i$) by
$P_i= 0$ (respectively, by $P_i > 0$ and by $P_i < 0$).

\begin{definition}\label{def:hom_lists}
We say that two ordered lists ${\mathcal P} = (P_1,\ldots,P_s)$,
${\mathcal Q} = (Q_1,\ldots,Q_s)$
of polynomials $P_i, Q_i \in {\mathbb{R}[X_1,\ldots,X_k]}$ have the same \emph{homotopy
type} if for any Boolean formula $\phi$, the semi-algebraic sets defined by
$\phi_{\mathcal P}$ and $\phi_{\mathcal Q}$ are
homotopy equivalent.
Clearly, in order to be homotopy equivalent two lists should have equal size.
\end{definition}

\begin{example}
\label{eg:homotopy_types}
Consider the lists
$\mathcal{P} = (X_1,X_2^2,X_1^2+X_2^2+1)$
and $\mathcal{Q}= (X_1^3, X_2^4,1)$. It is easy to see that they
have the same homotopy type, since in this case
for
each Boolean formula
$\phi$ with $9$ atoms, the semi-algebraic sets
defined by $\phi_{\mathcal{P}}$ and $\phi_{\mathcal{Q}}$ are
identical.
A slightly more non-trivial example is provided by
$\mathcal{P} = (X_2 - X_1^2, X_2)$ and $\mathcal{Q} = (X_2, X_2 + X_1^2)$.
In this case,
for
each Boolean formula
$\phi$ with $6$ atoms, the semi-algebraic sets
defined by $\phi_{\mathcal{P}}$ and $\phi_{\mathcal{Q}}$ are
not identical but
homeomorphic. Finally, the singleton sequences
$\mathcal{P} = (X_2X_1(X_1-1))$ and $\mathcal{Q} = (X_2(X_1^2 - X_2^4))$
are homotopy equivalent. In this case the semi-algebraic sets
sets
defined by $\phi_{\mathcal{P}}$ and $\phi_{\mathcal{Q}}$ are
homotopy equivalent, but not necessarily
homeomorphic. For instance, the algebraic set
defined by $X_2X_1(X_1-1)=0$ is
homotopy equivalent to
the algebraic set defined by $X_2(X_1^2 - X_2^4)=0$, but they are not
homeomorphic to each other.
\end{example}

The following theorem is proved in \cite{BV06}.

\begin{theorem}\cite{BV06}
\label{thm:additive}
The number of
distinct
homotopy types of ordered lists in
$\mathcal{A}^{\mathrm{div-free}}_{k,a}$
does not exceed
\begin{equation}\label{eq:additive}
2^{O(k+a)^8}.
\end{equation}
In particular, if $\phi$ is any Boolean formula with $3s$ atoms,
the number of
distinct
homotopy types of the  semi-algebraic sets defined
by $\phi_{\mathcal P}$, where
${\mathcal P} = (P_1,\ldots,P_s)\in \mathcal{A}^{\mathrm{div-free}}_{k,a}$,
does not exceed (\ref{eq:additive}).
\end{theorem}

\begin{remark}
The bound 
in Theorem \ref{thm:additive}
is stated in a slightly different form than in the original paper, to take
into account the fact that by our definition the division-free
additive complexity of a polynomial (for example, that of a monomial)
is allowed to be $0$. This is not an important issue (see Remark
\ref{rem:zero} below).
\end{remark}

The additive complexity of a polynomial is defined
as follows \cite{Borodin-Cook76, Grigoriev82, Risler85, BRbook}.

\begin{definition}
\label{def:rational_additive}
A polynomial $P \in \mathbb{R}[X_1, \ldots ,X_k]$ is said to have
\emph{additive complexity}
at most $a$ if there are \emph{rational functions}
$Q_1, \ldots , Q_a \in \mathbb{R}(X_1, \ldots ,X_k)$ satisfying conditions
(i), (ii), and (iii)
in Definition \ref{def:additive}
with $\N$ replaced by $\Z$,
and we say that the above sequence of equations
is an \emph{additive representation} of $P$ of length $a$.
\end{definition}

\begin{example}
\label{ex:rational_additive}
The polynomial $X^d+ \cdots + X+1 =(X^{d+1}-1)/(X-1)\in \mathbb{R}[X]$
with $0<d \in \Z$, has  additive complexity
(but not division-free additive complexity) at most $2$ (independent
of $d$).
\end{example}

\begin{notation}
We denote by
$\mathcal{A}_{k,a}$ the family of ordered (finite) lists
${\mathcal P}=(P_1, \ldots , P_s)$ of polynomials $P_i \in \re[X_1,\ldots,X_k]$,
with the  additive complexity of every $P_i$ not exceeding $a_i$,
with
$a=\sum_{1 \le i \le s}a_i$.
\end{notation}

It was conjectured in \cite{BV06} that
Theorem \ref{thm:additive} could be strengthened by replacing
$\mathcal{A}^{\mathrm{div-free}}_{k,a}$ by $\mathcal{A}_{k,a}$.
In this chapter we prove this conjecture. More formally, we prove

\begin{theorem}
\label{thm:main1}
The number of
distinct
homotopy types of ordered lists in
$\mathcal{A}_{k,a}$
does not exceed
$2^{
(k+a)^{O(1)}}$.
\end{theorem}

\subsection{Additive complexity and limits of semi-algebraic sets}
The proof of Theorem \ref{thm:additive} in \cite{BV06} proceeds by reducing the
problem to the case of bounding the number of
distinct
homotopy types of
semi-algebraic sets defined by polynomials having a bounded number
of monomials.
The reduction
which was already used by Grigoriev \cite{Grigoriev82} and Risler \cite{Risler85}
is as follows.
Let  ${\mathcal P} \in \mathcal{A}^{\mathrm{div-free}}_{k,a}$ be an ordered list.
For each polynomial $P_i \in {\mathcal P}$, $1 \le i \le s$,
consider the sequence of polynomials
$Q_{i 1}, \ldots, Q_{i a_i}$ as in Definition~\ref{def:additive}, so that
$$
P_i:=c_i X_{1}^{\zeta_{i 1}} \cdots X_{k}^{\zeta_{i k}}\prod_{1 \le j \le a_i}Q_{i j}^{\eta_{i j}}.
$$
Introduce $a_i$ new variables
$Y_{i1}, \ldots ,Y_{i a_i}$.
Fix a semi-algebraic set $S \subset \mathbb{R}^m$,
defined by a formula $\phi_{\mathcal P}$.
Consider
the semi-algebraic set $\widehat S$, defined by the conjunction of $a$
3-nomial equations obtained from equalities in (i), (ii) of
Definition~\ref{def:additive} by replacing
$Q_{i j}$ by $Y_{i j}$
for all $1 \le i \le s$, $1 \le j \le a_k$, and the formula $\phi_{\mathcal P}$
in which every occurrence of an atomic formula
of the kind $P_k \ast 0$, where $\ast \in \{ =, >, < \}$,
is replaced by the formula
$$
c_i X_{1}^{\zeta_{i 1}} \cdots X_{k}^{\zeta_{i k}}\prod_{1 \le j \le a_i}Y_{i j}^{\eta_{i j}}
\ast 0.
$$
Note that $\widehat S$ is a
semi-algebraic set in $\mathbb{R}^{k+a}$.

Let $\rho:\> \mathbb{R}^{k+a} \to \mathbb{R}^k$
be the projection map on the subspace
spanned by
$X_1, \ldots ,X_k$.
It is clear that the restriction $\rho_{\widehat S}:\> \widehat S \to S$ is a
homeomorphism, and moreover $\widehat S$ is defined by polynomials having
at most $k+a$ monomials.
Thus, in order to bound the number of
distinct
homotopy types
for $S$, it suffices to bound the same number for $\widehat{S}$, but
since   $\widehat S$ is defined by at most $2 a$ polynomials in
$k+a$ variables having at most $k+a$ monomials in total,
we have reduced the problem of bounding the number
of distinct homotopy types occurring in
$\mathcal{A}^{\mathrm{div-free}}_{k,a}$,
to that of bounding the the number
of distinct homotopy types of semi-algebraic sets
defined by at most $2 a$ polynomials
in $k+a$ variables,
with the total number of monomials appearing bounded by $k+a$.
This allows us to apply a bound proved in the fewnomial case in \cite{BV06},
to obtain
a singly exponential bound on the number of distinct homotopy types
occurring in  $\mathcal{A}^{\mathrm{div-free}}_{k,a}$.

Notice that for the map $\rho_{\widehat S}$ to be a homeomorphism it is
crucial that the exponents $\eta_{i j},\gamma_{i j},\delta_{i j}$
be non-negative, and this restricts
the proof to the case of division-free additive complexity.
We overcome this difficulty as follows.

Given a polynomial $F \in \mathbb{R}[X_1,\ldots,X_k]$ with
additive complexity bounded by $a$, we
prove that $F$ can be expressed as a quotient $\frac{P}{Q}$ with
$P,Q \in \mathbb{R}[X_1,\ldots,X_k]$ with
the sum of the \emph{division-free} additive complexities of $P$ and $Q$ bounded by $a$
(see Lemma \ref{lem:equivalence} below).
We then express the set of real zeros of $F$ in $\mathbb{R}^k$ inside any
fixed closed ball
as the Hausdorff limit of a
one-parameter semi-algebraic family defined using the polynomials
$P$ and $Q$ (see
Proposition \ref{prop:sectionfivemain} and the accompanying
Example \ref{eg:main} below).

While the limits of one-parameter semi-algebraic families defined by
polynomials with bounded division-free additive complexities
themselves can have complicated descriptions which cannot be described
by polynomials of bounded division-free additive complexity, the
topological complexity (for example, measured by their Betti numbers)
of such limit sets are well controlled.  Indeed, the problem of
bounding the Betti numbers of Hausdorff limits of one-parameter
families of semi-algebraic sets was considered by Zell in
\cite{Hausdorff}, who proved a singly exponential bound on the Betti
numbers of such sets.  We prove in this chapter (see
Theorems \ref{thm:main_weak} and  \ref{thm:main} below)
that the number of
distinct
homotopy types of such limits
can indeed be bounded singly exponentially in terms of the format of
the formulas defining the one-parameter family.  The techniques
introduced by Zell in \cite{Hausdorff} (as well certain semi-algebraic
constructions described in \cite{BZ09}) play a crucial role in the
proof of our bound.  These intermediate results may be of independent
interest.

Finally, applying
Theorem \ref{thm:main_weak}
to the one-parameter
family referred to in the above paragraph, we obtain a bound
on the number of
distinct
homotopy types of real algebraic varieties defined by
polynomials having bounded additive complexity. The semi-algebraic
case requires certain additional techniques and is dealt with in
Section \ref{subsec:semi-algebraic}.

\subsection{Homotopy types of limits of semi-algebraic sets}
In order to state our results on bounding the number of
distinct
homotopy types of limits
of one-parameter families of semi-algebraic sets we need to introduce
some notation.

\begin{notation}
For any first order formula $\Phi$
with $k$ free variables, if $\mathcal{P}\subset \mathbb{R}[X_1,\dots,X_k]$
consists of the polynomials appearing in $\Phi$, then we call $\Phi$ a $\mathcal{P}$-formula.
\end{notation}

\begin{remark}{\label{rem:zero}}
A monomial has additive complexity
0, but every $\mathcal{P}$-formula with $\mathcal{P}\subset \mathbb{R}[X_1,\ldots,X_k]$ containing only monomials is equivalent to a $\mathcal{P}'$-formula, where $\mathcal{P}'=\{X_1,\ldots,X_k\}$.
In particular, if $\phi$ is a $\mathcal{P}$-formula with
(division-free) additive format bounded by $(a,k)$, then $\phi$ is equivalent to a $\mathcal{P}'$-formula having
(division-free)
additive format bounded by $(a,k)$ and such that the cardinality of $\mathcal{P}'$ is
at most $a+k$.
\end{remark}

\section{Proof of a Weak Version of Theorem \ref{thm:main}}
\label{sec:main}

In this section we prove the following weak version of Theorem \ref{thm:main}
(using \emph{division-free} additive format rather than additive format) which
is needed in the proof of Theorem \ref{thm:main1}.

\begin{theorem}
\label{thm:main_weak}
For each $a,k\in \mathbb{N}$,
there exists a finite collection
$\mathcal{S}_{k,a}$
of semi-algebraic subsets of
$\mathbb{R}^N$, $N=(k+2)(k+1)+\binom{k+2}{2}$, with
$\card\; \mathcal{S}_{k,a}= 2^{O(k(k^2+a))^{8}} =
2^{(k+a)^{O(1)}}
$,
which satisfies
the following property.
If
$\TT\subset \mathbb{R}^k\times \mathbb{R}_+$
is a bounded semi-algebraic set described by a formula having
\emph{division-free}
additive format bounded by $(a,k+1)$
such that $\TT_t$ is closed for each $t>0$,
then
$\TT_{\limit}$
is
homotopy equivalent to some $S\in \mathcal{S}_{k,a}$
(cf. Notation \ref{not:limit}).
\end{theorem}

\subsection{Outline of the proof}
The main steps in the
proof of Theorem \ref{thm:main_weak}
are as follows.  Let $\TT \subset
\mathbb{R}^{k}\times
\mathbb{R}_{+}$
be a bounded semi-algebraic set,
such that
$\TT_t$ is closed for each $t \in \mathbb{R}$,
and let
$\TT_{\limit}$
be as in Notation \ref{not:limit}.

We first prove that for all  small enough $\lambda>0$,
there exists a semi-algebraic surjection
$f_\lambda:\TT_\lambda \rightarrow \TT_{\limit}$
which is metrically close to the identity map $1_{\TT_\lambda}$
(see Proposition \ref{prop:12} below).
Using a semi-algebraic realization of the fibered join described in
\cite{BZ09}
(see also \cite{GVZ04}),
we then
consider, for any fixed $p \geq 0$, a semi-algebraic set
$\mathcal{J}^p_{f_\lambda}(\TT_\lambda)$ which is $p$-equivalent to
$\TT_\limit$
(see Proposition \ref{prop:5}). The
definition of $\mathcal{J}^p_{f_\lambda}(\TT_\lambda)$ still involves the map
$f_\lambda$, whose definition is not simple, and hence we cannot
control the topological type of $\mathcal{J}^p_{f_\lambda}(\TT_\lambda)$
directly.
However, the fact that $f_\lambda$ is
metrically close to the identity map
enables us to adapt the main technique in \cite{Hausdorff} due to Zell. We
replace $\mathcal{J}^p_{f_\lambda}(\TT_\lambda)$ by another semi-algebraic set,
which we denote by
$\mathcal{D}^p_\eps(\TT)$ (for $\eps>0$ small enough),
which is  homotopy equivalent to
$\mathcal{J}^p_{f_\lambda}(\TT_\lambda)$, but
whose definition no longer involves the
map $f_\lambda$ (Definition \ref{defncalD}).
We can now bound the format of $\mathcal{D}^p_\eps(\TT)$
in terms of the format of the formula defining $\TT$.
This key result is summarized in Proposition  \ref{prop:main}.

We first recall the definition of $p$-equivalence
(see, for example, \cite[page 144]{tomDieck08}).

\begin{definition}[$p$-equivalence]
\label{def:p-equivalence}
A map $f: A \rightarrow B$ between two topological spaces is called a
\emph{$p$-equivalence} if the induced
map
\[
f_*: \boldsymbol{\pi}_i(A,a) \rightarrow \boldsymbol{\pi}_i(B,f(a))
\]
is,
 for each $a \in A$,
bijective
for  $0 \leq i < p$, and
surjective
for $i=p$,
and we say that $A$ is \em{$p$-equivalent} to $B$.
\end{definition}

\begin{proposition}
\label{prop:main}
Let
$\TT\subset \mathbb{R}^k\times \mathbb{R}_+$
be a bounded
semi-algebraic set such that $\TT_t$ is closed
for each  $t>0$, and let
$p \geq 0$.
Suppose also that $\TT$ is described by a formula
having
(division-free)
additive
format
bounded by
$(a,k+1)$
and dense format
$(s,d,k+1)$.
Then, there exists a semi-algebraic set
$\mathcal{D}^p\subset \mathbb{R}^N$,
$N=(p+1)(k+1)+\binom{p+1}{2}$,
 such  that $\mathcal{D}^p$ is
$p$-equivalent to
$\TT_{\limit}$ (cf. Notation \ref{not:limit})
and such that $\mathcal{D}^p$ is described by a
formula having
(division-free)
additive
 format
bounded by $(M,N)$
and dense format
$(M',d+1,N)$,
where
$M=\constM$
and $M'=\constMprime$.
\hide{
$$\begin{aligned}
M&=\textstyle{\constM}, \\
M'&=\textstyle{\constMprime}.
\end{aligned}$$
}

\end{proposition}

Finally, Theorem \ref{thm:main_weak} is an easy consequence of
Proposition \ref{prop:main}.

\begin{notation}{\label{not:limit}}
\hide{
For
any $k \geq 1$, and
 $ 1\leq p \leq q \leq k$,
we denote by  $\pi_{[p,q]}: \mathbb{R}^k=\mathbb{R}^{[1,k]}\rightarrow \mathbb{R}^{[p,q]}$ the projection
$$(x_1, \ldots, x_k)\mapsto (x_p, \ldots, x_q)$$
(omitting the dependence on $k$ which should be clear from context).
In case $p=q$ we
will denote by $\pi_p$ the projection $\pi_{[p,p]}$. %
For any semi-algebraic subset
$X\subset \mathbb{R}^{k+1}$,
 and $\lambda\in \mathbb{R}$,
we denote by $X_\lambda$ the following semi-algebraic subset of $\mathbb{R}^k$:
$$X_\lambda = \pi_{[1,k]}(X\cap \pi_{k+1}^{-1}(\lambda)).$$
}
We denote by $\mathbb{R}_+$ the set of strictly positive elements of $\mathbb{R}$.
If additionally $X\subset \mathbb{R}^k\times \mathbb{R}_+$, then we denote by $X_{\limit}$
the following semi-algebraic subset of $\mathbb{R}^k$:
$$ X_{\limit} := \pi_{[1,k]}(\overline{X} \cap \pi_{k+1}^{-1}(0)),$$
where $\overline{X}$ denotes the topological closure of $X$ in
$\mathbb{R}^{k+1}$.
\zz
\end{notation}

We have the following theorem
which establishes a singly exponential bound on
the number of
distinct
homotopy types of the Hausdorff limit
of a one-parameter family of compact semi-algebraic sets defined by a
first-order formula of bounded additive format. This result complements
the result in \cite{BV06} giving singly exponential bounds on the
homotopy types of semi-algebraic sets defined by first-order
formulas having bounded
division-free
additive format on one hand, and the
result of Zell \cite{Hausdorff} bounding the Betti numbers
of the Hausdorff
limits of one-parameter families of semi-algebraic sets on the other,
and could be of independent interest.

\begin{theorem}
\label{thm:main}
For each $a,k\in \mathbb{N}$,
there exists a finite collection
$\mathcal{S}_{k,a}$
of semi-algebraic subsets of
$\mathbb{R}^N$, $N=(k+2)(k+1)+\binom{k+2}{2}$, with
$\card\; \mathcal{S}_{k,a}
=2^{(k+a)^{O(1)}}$,
which satisfies
the following property.
If $\TT\subset \mathbb{R}^k\times \mathbb{R}_+$
is a bounded semi-algebraic set described by a formula having
additive format bounded by $(a,k+1)$
such that $\TT_t$ is closed for each $t>0$,
then
$\TT_{\limit}$
is
homotopy equivalent to some
$S\in \mathcal{S}_{k,a}$
(cf. Notation \ref{not:limit}).
\end{theorem}

The rest of the chapter is devoted to the proofs of Theorems \ref{thm:main}
and \ref{thm:main1} and is organized as follows. We first prove
a weak version (Theorem \ref{thm:main_weak}) of
Theorem \ref{thm:main} in Section \ref{sec:main}, in which the term
``additive complexity'' in the statement of Theorem \ref{thm:main}
is replaced by the term
``division-free additive complexity''.
Theorem \ref{thm:main_weak} is
then used  in Section \ref{sec:main1}
to prove Theorem \ref{thm:main1} after introducing some additional
techniques, which in turn is
used to prove Theorem \ref{thm:main}.

\subsection{Topological definitions}
We first recall the basic definition of the
the iterated join of a topological space.

\begin{notation}
\label{not:simplex}
For each
$p \geq 0$,
we denote
\[
\Delta_{[0,p]} = \{\textbf{t} = (t_0,\ldots,t_p)\mid t_i \geq  0, 0 \leq i \leq p, \sum_{i=0}^p
t_i = 1,
\}
\]
the standard $p$-simplex.
For each subset $I = \{i_0,\ldots,i_m\}, 0 \leq i_0 < \cdots < i_m \leq p$,
let $\Delta_I \subset \Delta_{[0,p]}$ denote the face
$$\Delta_I = \{ \textbf{t} = (t_0,\ldots,t_p)
\in \Delta_{[0,p]}\; \mid \; t_i = 0 \mbox{ for all } i \not\in I \}
$$
of $\Delta_{[0,p]}$.
\end{notation}

\hide{
\begin{definition}
\label{def:twofoldjoin}
The join $J(X,Y)$  of two topological spaces $X$ and $Y$ is defined by
\begin{equation}
\label{eqn:definitionoftwofoldjoin}
J(X,Y) \defeq X\times Y
\times \Delta_{[0,1]}/\sim,
\end{equation}
where
\[
(x,y,t_0,t_1) \sim (x',y',t_0,t_1)
\]
if
$t_0 = 1,x = x'$ or  $t_1=1, y= y'$.
\end{definition}

 \begin{figure}[hbt]
         \centerline{
           \scalebox{0.5}{
             \input{join.pstex_t}
             }
           }
         \caption{Join of two segments}
         \label{fig:join}
 \end{figure}

Intuitively, $J(X,Y)$ is obtained by joining each point of $X$ with
each point of $Y$ by a unit interval (see Figure \ref{fig:join}).

By iterating the above definition with the same space $X$ we obtain
}

\begin{definition}
\label{def:pfoldjoin}
For
$p \geq 0$,
the $(p+1)$-fold join
$J^p(X)$ of a topological space
$X$ is
\begin{equation}
\label{eqn:definitionofjoin}
J^p(X) \defeq \underbrace{X\times\cdots\times X}_{(p+1)\mbox{ times }}
\times \Delta_{[0,p]}/\sim,
\end{equation}
where
\[
(x_0,\ldots,x_p,t_0,\ldots,t_p) \sim (x_0',\ldots,x_p',t_0,\ldots,t_p)
\]
if for each $i$ with $t_i \neq 0$, $x_i = x_i'$.
\end{definition}

In the special situation
when $X$ is a semi-algebraic set,
the space $J^p(X)$ defined above is not immediately
a semi-algebraic set, because of taking quotients.
We now define a semi-algebraic set,
$\mathcal{J}^p(X)$, that is
homotopy equivalent
to $J^p(X)$.

\hide{
We use the lower case bold-face notation $\mathcal{\textbf{x}}$ to denote a
point
$\mathcal{\textbf{x}}=(x_1,\dots,x_k)$ of $\mathbb{R}^k$, and upper-case $\X=(X_1,\dots,X_k)$
to denote a \emph{block of variables}.
}

Let $\Delta'_{[0,p]} \subset \mathbb{R}^{p+1}$ denote the set defined by
$$
\Delta'_{[0,p]} = \{ \textbf{t} = (t_0,\ldots,t_p) \in \mathbb{R}^{p+1} \;\mid\; \sum_{0 \leq i \leq p} t_i = 1 , |\textbf{t}|^2 \leq 
1\}.
$$

For each subset $I = \{i_0,\ldots,i_m\}, 0 \leq i_0 < \cdots < i_m \leq p$,
let $\Delta'_I \subset \Delta'_{[0,p]}$ denote
$$
\Delta'_I = \{ \textbf{t} = (t_0,\ldots,t_p)
\in \Delta'_{[0,p]}\; \mid \; t_i = 0 \mbox{ for all } i \not\in I \}.
$$

It is clear that the standard simplex
$\Delta_{[0,p]}$
is a deformation retract of $\Delta'_{[0,p]}$
via a deformation retraction,
$\rho_p: \Delta'_{[0,p]} \rightarrow \Delta_{[0,p]}$,
that restricts to a deformation retraction
$\rho_p|_{\Delta'_I}:\Delta'_I\rightarrow \Delta_I$ for each $I \subset [0,p]$.

We use the lower case bold-face notation $\mathcal{\textbf{x}}$ to denote a
point
$\mathcal{\textbf{x}}=(x_1,\dots,x_k)$ of $\mathbb{R}^k$
and upper-case $\X=(X_1,\dots,X_k)$
to denote a \emph{block of variables}.
In the following definition the role of the
$\binom{p+1}{2}$
variables
$(A_{ij})_{0 \leq i < j \leq p}$ can be safely ignored, since they are
all set to $0$. Their significance will be clear later.

\begin{definition}[The semi-algebraic join \cite{BZ09}%
    \label{def:semi-algebraic-join}]
For a semi-algebraic subset $X \subset \mathbb{R}^k$
contained in $B_k(0,R)$,
defined by a $\mathcal{P}$-formula
$\Phi$,
we define
$$
\begin{aligned}
\mathcal{J}^p(X)=&\{(\mathcal{\textbf{x}}^0,\dots,\mathcal{\textbf{x}}^p,
  \textbf{t},\textbf{a})\in  \mathbb{R}^{(p+1)(k+1)+\binom{p+1}{2}}|\\ &\quad
  \Omega^R(\mathcal{\textbf{x}}^0,\dots,\mathcal{\textbf{x}}^p,\textbf{t}) \wedge
  \Theta_1(\textbf{t},\textbf{a}) \wedge
  \Theta_2^\Phi (\mathcal{\textbf{x}}^0,\dots,\mathcal{\textbf{x}}^p,\textbf{t})
\},
\end{aligned}
$$
where
\begin{equation}
{\label{eqn:semi-algebraic-join}}
\begin{aligned}
\Omega^R \ :=\ & \quad \bigwedge_{i=0}^p (|\X^i|^2\leq R^2) \wedge |\T|^2\leq
1, \\
\Theta_1 \ :=\ & \quad \sum_{i=0}^p T_i=1 \wedge
\sum_{0 \leq i < j \leq p} A_{i j}^2=0, \\
\Theta_2^\Phi\ :=\ & \quad \bigwedge_{i=0}^p (T_i=0 \vee  \Phi(\X^i)), \\
\end{aligned}\end{equation}
We denote the formula $\Omega^R\wedge \Theta_1\wedge \Theta_2^\Phi$
by $\mathcal{J}^p(\Phi)$.
\end{definition}

\begin{notation}
For any
$\R\in \mathbb{R}_+$,
we denote by $B_k(0,R) \subset \mathbb{R}^k$,
the open ball of radius $R$ centered at the origin.
\end{notation}

\hide{

We introduce some more notation.

\begin{notation}
For any
$\R\in \mathbb{R}_+$,
we denote by $B_k(0,R) \subset \mathbb{R}^k$,
the open ball of radius $R$ centered at the origin.
\end{notation}

\begin{notation}
For
$P \in \mathbb{R}[X_1,\ldots,X_k]$,
we denote by $\Zer(P,\mathbb{R}^k)$ the real
algebraic set defined by $P=0$.
\end{notation}

\begin{notation}{\label{not:reali}}
For any first order formula $\Phi$
with $k$ free variables, we denote by $\Reali(\Phi)$ the semi-algebraic
subset of $\mathbb{R}^k$ defined by $\Phi$.
\hide{
Additionally, if $\mathcal{P}\subset \mathbb{R}[X_1,\dots,X_k]$
consists of the polynomials appearing in $\Phi$, then we call $\Phi$ a $\mathcal{P}$-formula.
}
\end{notation}

}

It is checked easily from Definition \ref{def:semi-algebraic-join} that
$$
\displaylines{
\mathcal{J}^p(X) \subset \left(\overline{B_{k}(0,R)}\right)^{p+1} \times \Delta'_{[0,p]} \times
\{\textbf{0}\},
}
$$
and that the deformation retraction
$\rho_p: \Delta'_{[0,p]} \rightarrow \Delta_{[0,p]}$
extends to a deformation retraction,
$\tilde{\rho}_p: \mathcal{J}^p(X) \rightarrow \tilde{\mathcal{J}}^p(X)$, where
$\tilde{\mathcal{J}}^p(X)$ is defined by
$$
\displaylines{
\tilde{\mathcal{J}}^p(X)=
\{
(\mathcal{\textbf{x}}^0,\dots,\mathcal{\textbf{x}}^p,
  \textbf{t},\textbf{a})\in \left(\overline{B_{k}(0,R)}\right)^{p+1} \times \Delta_{[0,p]} \times
\{\textbf{0}\} \;\mid\;
   \Theta_2^\Phi (\mathcal{\textbf{x}}^0,\dots,\mathcal{\textbf{x}}^p,\textbf{t})
\}.
}
$$
Finally, it is a consequence of the Vietoris-Beagle theorem
(see \cite[Theorem 2]{BWW06})
that
$\tilde{\mathcal{J}}^p(X)$ and $J^p(X)$ are homotopy equivalent.
We thus have, using notation introduced
above, that
\begin{proposition}
$\mathcal{J}^p(X)$ is homotopy equivalent to $J^p(X)$.
\end{proposition}

\begin{remark}
The necessity of defining $\mathcal{J}^p(X)$ instead of just
$\tilde{\mathcal{J}}^p(X)$ has to do with removing the inequalities
defining the standard simplex from the defining formula
$\mathcal{J}^p(\Phi)$, and this will
simplify certain arguments later in our exposition.
\end{remark}

We now generalize the above constructions and define joins over maps (the
topological and semi-algebraic joins defined above are special cases when the
map is a constant map to a point).

\begin{notationdefinition}
\label{notdef:fiberproduct}
\emph{
Let $f:A \rightarrow B$ be a map
between topological spaces $A$ and $B$.
For each $p \geq 0$,
we denote by $W_f^p(A)$ the
\emph{$(p+1)$-fold fiber product} of $A$ over $f$.
In other words
\[
W_f^p(A) = \{(x_0,\ldots,x_p) \in A^{p+1}
\mid
f(x_0) = \cdots = f(x_p)
\}.
\]}
\end{notationdefinition}

\begin{definition}[Topological join over a map]
\label{def:joinoveramap1}
Let $f:X \rightarrow Y$ be a map
between topological spaces $X$ and $Y$.
For
$p \geq 0$,
the
\emph{$(p+1)$-fold join}
$J^p_f(X)$  of $X$ over $f$
is
\begin{equation}
\label{eqn:definitionofjoin1}
J^p_f(X) \defeq W_f^p(X) \times \Delta^p/\sim,
\end{equation}
where
\[
(x_0,\ldots,x_p,t_0,\ldots,t_p) \sim (x_0',\ldots,x_p',t_0,\ldots,t_p)
\]
if for each $i$ with $t_i \neq 0$, $x_i = x_i'$.
\end{definition}

In the special situation
when $f$ is a semi-algebraic continuous map,
the space $J^p_f(X)$ defined above is (as before) not immediately
a semi-algebraic set, because of taking quotients.
Our next goal is to obtain a semi-algebraic set,
$\mathcal{J}^p_{f}(X)$ which is homotopy equivalent
to $J^p_f(X)$ similar to the case of the ordinary join.

\newcommand{\thetafour}{\bigwedge_{0\leq i<j \leq p } (T_i= 0 \vee T_j = 0 \vee
|f(\X^i)-f(\X^j)|^2=A_{ij})}

\begin{definition}[The semi-algebraic fibered join \cite{BZ09}%
    \label{defnJP}]
For a semi-algebraic subset $ X\subset \mathbb{R}^k$
contained in $B_k(0,R)$,
defined by a $\mathcal{P}$-formula
$\Phi$ and $f:X\to Y$ a semi-algebraic map, we define
$$
  \begin{aligned}\mathcal{J}^p_{f}(X)=&\{
(\textbf{x}^0,\dots,\textbf{x}^p,
  \textbf{t},\textbf{a})\in  \mathbb{R}^{(p+1)(k+1)+\binom{p+1}{2}}|\\ &\quad
  \Omega^R(\mathcal{\textbf{x}}^0,\dots,\mathcal{\textbf{x}}^p,\textbf{t}) \wedge
  \Theta_1(\textbf{t},\textbf{a}) \wedge
  \Theta_2^\Phi (\mathcal{\textbf{x}}^0,\dots,\mathcal{\textbf{x}}^p,\textbf{t})\wedge
  \Theta_3^f
  (\mathcal{\textbf{x}}^0,\dots,\mathcal{\textbf{x}}^p,\textbf{t},\textbf{a})
\},
\end{aligned}$$ where

$\Omega^R, \Theta_1,
\Theta_2^\Phi$ have been defined previously, and
\begin{equation}{\label{eqn:fibjoin}}\begin{aligned}
\Theta_3^f \ := \  & \thetafour.
\end{aligned}\end{equation}
We denote the formula $\Omega^R\wedge \Theta_1\wedge \Theta_2^\Phi \wedge
\Theta_3^f$ by $\mathcal{J}^p_f(\Phi)$.
\end{definition}

Observe that there exists a natural map,
$J^p(f): \mathcal{J}^p_f(X) \rightarrow Y$, which maps a point
$(\textbf{x}^0,\dots,\textbf{x}^p, \textbf{t},\textbf{0}) \in
\mathcal{J}^p_f(X)$ to $f(\textbf{x}^i)$ (where $i$ is such that $t_i \neq 0$).
It is easy to see that for each $\textbf{y} \in Y$,
$J^p(f)^{-1}(\textbf{y}) = \mathcal{J}^p(f^{-1}(\textbf{y}))$.

The following proposition
follows from the above observation and the generalized
Vietoris-Begle theorem
(see \cite[Theorem 2]{BWW06})
and is important in the proof of Proposition \ref{prop:main};
it relates up to $p$-equivalence
the semi-algebraic set $\mathcal{J}^p_f(X)$ to the
image of a closed, continuous semi-algebraic surjection $f:X\to Y$.
\hide{
It is similar to
Theorem  2.12 proved in \cite{BZ09}.
}
Its proof
is similar to the proof of
Theorem  2.12 proved in \cite{BZ09} and is omitted.

\begin{proposition}
{\label{prop:5}}\cite{BZ09}
Let $f:X \to Y$ a closed, continuous
semi-algebraic
surjection
with
$X \subset B_k(0,R)$
a closed
semi-algebraic set.  Then,
for  every $p\geq 0$,
the map $J^p(f):  \mathcal{J}^p_f(X) \to Y$ is a $p$-equivalence.
\end{proposition}

We now define a thickened version of the semi-algebraic set
$\mathcal{J}^p_f(X)$ defined above and prove that it is homotopy equivalent to
$\mathcal{J}^p_f(X)$.
The variables $A_{ij}, 0 \leq i < j \leq p$, play an important role in
the thickening process.

\begin{definition}[The thickened semi-algebraic fibered join]
\label{defncalJ}
For $X\subset \mathbb{R}^k$ a semi-algebraic
  set contained in $B_k(0,R)$
  defined by a $\mathcal{P}$-formula
$\Phi$,
  $p\geq 1$, and $\ep>0$
  define
  $$
  \begin{aligned}\mathcal{J}^p_{f,\ep}(X)=&\{(\mathcal{\textbf{x}}^0,\dots,\mathcal{\textbf{x}}^p,
  \textbf{t},\textbf{a})\in  \mathbb{R}^{(p+1)(k+1)+\binom{p+1}{2}}|\\ &\quad
\Omega^R(\mathcal{\textbf{x}}^0,\dots,\mathcal{\textbf{x}}^p,\textbf{t}) \wedge
  \Theta_1^\ep(\textbf{t},\textbf{a}) \wedge
  \Theta_2^\Phi (\mathcal{\textbf{x}}^0,\dots,\mathcal{\textbf{x}}^p,\textbf{t})\wedge
  \Theta_3(\mathcal{\textbf{x}}^0,\dots,\mathcal{\textbf{x}}^p,\textbf{t},\textbf{a})
  \},
  \end{aligned}$$ where
\begin{equation}{\label{eqn:fibjoin2}}\begin{aligned}
\Omega^R\ :=\ &\quad \bigwedge_{i=0}^p (|\X^i|^2\leq R^2) \wedge |\T|^2\leq
1, \\
\Theta_1^\ep\ :=\ &\quad \sum_{i=0}^p T_i=1 \wedge \sum_{1\leq i<j\leq p} A_{i j}^2\leq \ep, \\
\Theta_2^\Phi\ :=\ & \quad \bigwedge_{i=0}^p (T_i=0 \vee  \Phi(\X^i)), \\
\Theta_3^f \ := \  & \thetafour.
\end{aligned}\end{equation}
\end{definition}

Note that if $X$ is closed
(and bounded), then
$\mathcal{J}^p_{f,\ep}(X)$ is again closed (and bounded).

The relation between $\mathcal{J}^p_{f}(X)$ and $\mathcal{J}^{p}_{f,\ep}(X)$ is
described in the following proposition.

\begin{proposition}
\label{prop:7}
For $p\in \mathbb{N}$,  $f:X\to Y$
  semi-algebraic  there exists $\ep_0>0$ such that $\mathcal{J}^p_{f}(X)$ is
homotopy equivalent to $\mathcal{J}^{p}_{f,\ep}(X)$
  for all $0<\ep\leq \ep_0$.
\end{proposition}

Proposition \ref{prop:7} follows from the following two lemmas.

\begin{lemma}{\label{lem:prop7b}} For $p\in \mathbb{N}$, $f:X\to Y$
  semi-algebraic we have
  $$\mathcal{J}^p_f(X)=\bigcap_{t>0} \mathcal{J}^p_{f,t}(X).$$
\end{lemma}

\begin{proof} Obvious from Definitions \ref{defnJP} and \ref{defncalJ}. \end{proof}


\begin{lemma}
\label{lem:prop7a}  Let
$\TT \subset \mathbb{R}^k\times \mathbb{R}_+$
  such that each $\TT_t$ is closed and $\TT_t\subseteq B_k(0,R)$
  for $t>0$.   Suppose further that for all $0<t\leq t'$ we have $\TT_t
  \subseteq \TT_{t'} $.
Then, $$\bigcap_{t>0}\TT_t =
  \pi_{[1,k]}\left(\overline{\TT}\cap \pi_{k+1}^{-1}(0)\right).$$
  Furthermore,
  there exists  $\ep_0>0$ such that for all $\ep$ satisfying
  $0<\ep\leq \ep_0$  we have that
$\TT_\ep$
  is
semi-algebraically homotopy
  equivalent to
$\TT_\limit$
(cf. Notation \ref{not:limit}).
\end{lemma}

\hide{
\begin{lemma}
\label{lem:prop7a}  Let
$X \subset \mathbb{R}^k\times \mathbb{R}_+$
  such that each $X_t$ is closed and $X_t\subseteq B_k(0,R)$
  for $t>0$.   Suppose further that for all $0<t\leq t'$ we have $X_t
  \subseteq X_{t'} $.
Then, $$\bigcap_{t>0}X_t =
  \pi_{[1,k]}\left(\overline{X}\cap \pi_{k+1}^{-1}(0)\right).$$
  Furthermore,
  there exists  $\ep_0>0$ such that for all $\ep$ satisfying
  $0<\ep\leq \ep_0$  we have that $X_{\limit}$ is
semi-algebraically homotopy
  equivalent to $X_\ep$
(cf. Notation \ref{not:limit}).
\end{lemma}
}

\hide{, but we could not
locate a precise statement to this effect in the
literature. We include a proof in the appendix.
}

\begin{proof} 
The first part of the proposition is straightforward.  The second part follows easily from Lemma 16.16 in \cite{BPRbook2}.

\end{proof}

\begin{proof}[Proof of Proposition \ref{prop:7}]
The set $\TT=\{(\mathcal{\textbf{x}},t)\in
\re^{k+1}| \ t>0 \wedge \mathcal{\textbf{x}}\in \mathcal{J}^p_{f,t}(X)\}$ satisfies the conditions
of Lemma  \ref{lem:prop7a}.  The proposition now follows from Lemma
\ref{lem:prop7a}  and Lemma \ref{lem:prop7b}.
\end{proof}

\begin{proposition}
\label{prop:8}
For $p\in\mathbb{N}$, $f:X\to Y$ semi-algebraic, and
$0<t\leq t'$,
  $$\mathcal{J}^p_{f, t}(X)\subseteq \mathcal{J}^p_{f,t'}(X).$$
  Moreover,  there exists $\ep_0>0$ such that for
$0<\ep\leq \ep'<\ep_0$
the above inclusion induces a semi-algebraic homotopy equivalence.
\end{proposition}

The first part of Proposition \ref{prop:8} is obvious
from the definition of $\mathcal{J}^p_{f, \ep}(X)$.
The second part follows from Lemma \ref{lem:2} below.

The following lemma is probably well known and easy.
However, since we were unable to locate
an exact statement to this effect in the literature,
we include a proof.

\begin{lemma}
\label{lem:2}
Let
$\TT\subset \mathbb{R}^k \times \mathbb{R}_+$
be a
semi-algebraic set, and
suppose that
$\TT_t \subset \TT_{t'}$
for all $0<t<t'$. Then, there
exists $\ep_0$ such that for each $0<\ep<\ep'\leq \ep_0$ the inclusion
map
$\TT_\ep \overset{i_{\ep'}}{\hookrightarrow} \TT_{\ep'}$
induces a
semi-algebraic homotopy equivalence.
\end{lemma}

\begin{proof}
We prove that
there exists $\phi_{\ep'}:\TT_{\ep'}\to
\TT_{\ep}$ such that $$\begin{aligned}\phi_{\ep'}\circ i_{\ep'}&:\TT_\ep
\to \TT_{\ep}\halfspace, \quad \phi_{\ep'}\circ i_{\ep'} \simeq \Id_{\TT_{\ep}}, \\
i_{\ep'} \circ \phi_{\ep'}&: \TT_{\ep'} \to \TT_{\ep'}, \quad i_{\ep'}\circ
\phi_{\ep'}\simeq \Id_{\TT_{\ep'}}.
\end{aligned}
$$

We first define $i_t:\TT_\ep \hookrightarrow \TT_t$ and $\widehat{i}_t:\TT_t
\hookrightarrow \TT_{\ep'}$, and note that trivially $i_\ep=
\Id_{\TT_\ep}$, $\widehat{i}_{\ep'}= \Id_{\TT_{\ep'}}$, and
$i_{\ep'}=\widehat{i}_\ep$.  Now, by Hardt triviality there exists
$\ep_0>0$, such that there is a definably trivial homeomorphism $h$ which
commutes with the projection
$\pi_{k+1}$,
i.e., the following diagram commutes.
$$
\xymatrix{\TT_{\ep_0} \times (0,\ep_0] \ar[r]^{h
\phantom{stuffhere}} \ar[d]^{\pi_{k+1}} & \TT \cap \{(\mathcal{\textbf{x}},t)| \ 0<t\leq
\ep_0\} \ar[dl]_{\pi_{k+1}} \\ (0,\ep_0] }
$$

Define
$F(\mathcal{\textbf{x}},t,s)=h(\pi_{[1,k]}\circ h^{-1}(\mathcal{\textbf{x}},t),s)$.  Note that $F(\mathcal{\textbf{x}},t,t)=h
(\pi_{[1,k]} \circ h^{-1}(\mathcal{\textbf{x}},t),t)=h(h^{-1}(\mathcal{\textbf{x}},t))=(\mathcal{\textbf{x}},t)$.
We define
$$\begin{aligned}\phi_t:\TT_t&\to \TT_\ep, \\ \phi_t(\mathcal{\textbf{x}})&=\pi_{[1,k]}\circ
F(\mathcal{\textbf{x}},t,\ep), \\ \widehat{\phi}_t: \TT_{\ep'}& \to \TT_t, \\
\widehat{\phi}_t(\mathcal{\textbf{x}})&=\pi_{[1,k]}\circ F(\mathcal{\textbf{x}},\ep',t)\end{aligned}$$ and
note that
$\phi_{\ep'}=\widehat{\phi}_\ep$.

Finally,
define $$\begin{array}{ll}
H_1(\cdot,t)&= \phi_t\circ i_t :\TT_\ep \to \TT_\ep, \\
H_1(\cdot,\ep)&=\phi_\ep\circ i_\ep = \Id_{\TT_\ep}, \\
H_1(\cdot,\ep')&=\phi_{\ep'}\circ i_{\ep'}, \\
\ & \ \\
H_2(\cdot,t)&= \widehat{i}_t \circ \widehat{\phi}_t: \TT_{\ep'} \to  \TT_{
  \ep'}, \\
H_2(\cdot,\ep)&=\widehat{i}_\ep \circ \widehat{\phi}_\ep = i_{\ep'}
\circ \phi_{ \ep'}, \\
H_2(\cdot,\ep')&=\widehat{i}_{\ep'} \circ \widehat{\phi}_{\ep'} =
\Id_{\TT_{ \ep'}}. \end{array}$$

The semi-algebraic continuous maps $H_1$ and $H_2$ defined above
give a semi-algebraic homotopy between the maps
$\phi_{\ep'} \circ i_{\ep'} \simeq  \Id_{\TT_{\ep}}$ and
$i_{\ep'} \circ \phi_{ \ep'} \simeq \Id_{\TT_{\ep'}}$
proving the required semi-algebraic homotopy equivalence.
\end{proof}

\newcommand{\Upsilonthree}
{\bigwedge_{0\leq i<j \leq p }
     (T_i= 0 \vee T_j = 0
\vee |\X^i-\X^j|^2= A_{i j})}

As mentioned before, we would like to
replace $\mathcal{J}^p_{f,\eps}(X)$
by another semi-algebraic set,
which we denote by
$\mathcal{D}^p_\eps(X)$,
which is  homotopy equivalent to
$\mathcal{J}^p_{f,\eps}(X)$,
under certain assumptions on $f$ and $\eps$,
whose definition no longer involves the
map $f$. This is what we do next.

\begin{definition}
[The thickened diagonal]
\label{defncalD}
For a semi-algebraic set $X\subset \mathbb{R}^k$ contained in $B_k(0,R)$
defined by a $\mathcal{P}$-formula
$\Phi$,
$p\geq 1$, and $\ep>0$,
define
$$
  \begin{aligned}\mathcal{D}^p_{\ep}(X)=&\{(\mathcal{\textbf{x}}^0,\dots,\mathcal{\textbf{x}}^p,
  \textbf{t},\textbf{a})\in  \mathbb{R}^{(p+1)(k+1)+\binom{p+1}{2}}|\\ &\quad
\Omega^R(\mathcal{\textbf{x}}^0,\dots,\mathcal{\textbf{x}}^p,\textbf{t}) \wedge
  \Theta_1(\textbf{t},\textbf{a}) \wedge
  \Theta_2^\Phi (\mathcal{\textbf{x}}^0,\dots,\mathcal{\textbf{x}}^p,\textbf{t})\wedge
  \Upsilon(\mathcal{\textbf{x}}^0,\dots,\mathcal{\textbf{x}}^p,\textbf{t},\textbf{a})
  \},
    \end{aligned}$$
    where $\Omega^R,\Theta_1^\ep,\Theta_2^\Phi$ are defined as in Equation \ref{eqn:fibjoin2}, and
\begin{equation*}{
\label{eqn:calD}
}\begin{aligned}
\Upsilon \ := \  & \Upsilonthree.
\end{aligned}\end{equation*}
\end{definition}

Notice that the formula defining the thickened diagonal,
$\mathcal{D}^p_{\ep}(X)$
in Definition
\ref{defncalD},
is identical
to that defining the thickened semi-algebraic fibered join,
$\mathcal{J}^p_{f,\ep}(X)$
in Definition \ref{defncalJ},
except that $\Theta_3^f$ is replaced by $\Upsilon$, and
$\Upsilon$ does not depend on the map $f$ or on the set $X$.

\begin{proposition}
{\label{prop:calDp} \label{prop:9}}
Let $X \subset \mathbb{R}^k$ be a semi-algebraic set defined by a
quantifier free
formula
$\Phi$ having
\hide{
dense format bounded by  $(s,d,k)$,
and division-free additive format bounded by $(a,k)$.
}
(division-free) additive format bounded by $(a,k)$
and dense format bounded by $(s,d,k)$.
 Then,
$\mathcal{D}^p_\eps(X)$
is a semi-algebraic subset set of
$\mathbb{R}^N$, defined by a formula
with
(division-free)
 additive format bounded by $(M,N)$ and dense format
bounded by   $(M',d+1,N)$,  where
$M=\constM$, $M'=\constMprime$, and
  $N=(p+1)(k+1)+\binom{p+1}{2}$.
\end{proposition}

\begin{proof}
It is a straightforward computation to bound the division-free additive
format and give the dense format of the formulas
$\Omega^R,\Theta_1^\eps,
 \Upsilon$ as well as the (division-free) additive format and dense format of the formula
 $\Theta_2^\Phi$.
More precisely, let
\hide{
$$\begin{array}{ll}
M_{\Omega^R} = (p+1)k+(p+1), & M'_{\Omega^R} = (p+1)+1, \\
M_{\Theta_1^\ep}=(p+1)+\binom{p+1}{2}, & M'_{\Theta_1^\ep}=2, \\
M_{\Theta_2^\Phi}=(p+1)(a+1), & M'_{\Theta_2^\Phi}=(p+1)(s+1), \\
M_{\Upsilon}=\binom{p+1}{2}(2k+2), &M'_{\Upsilon}=3\binom{p+1}{2}.
\end{array}$$
}
$$\begin{array}{ll}
M_{\Omega^R} = (p+1)k+(p+1), & M'_{\Omega^R} = (p+1)+1, \\
M_{\Theta_1^\ep}=(p+1)+\binom{p+1}{2}, & M'_{\Theta_1^\ep}=2, \\
M_{\Theta_2^\Phi}=(p+1)a, & M'_{\Theta_2^\Phi}=(p+1)(s+1), \\
M_{\Upsilon}=2k\binom{p+1}{2}, &M'_{\Upsilon}=3\binom{p+1}{2}.
\end{array}$$
It is clear from
Definition \ref{defncalD}
that the
division-free additive format (resp. dense format) of $\Omega^R$ is
bounded by $(M_{\Omega^R},N)$, $N=(p+1)(k+1)+\binom{p+1}{2}$
(resp. $(M'_{\Omega^R},2,N)$).  Similarly, the division-free
additive format (resp.
 dense format) of $\Theta_1^\ep,
 \Upsilon$
is bounded by $(M_{\Theta_1^\ep},N),
(M_{\Upsilon},N)$
 (resp. $(M'_{\Theta_1^\ep},2,N)$,
 $(M'_{\Upsilon},2,N)$).
 Finally, the (division-free) additive format of
 $\Theta_2^\Phi$
 is bounded by $(M_{\Theta_2^\Phi},N)$
 and dense format is $(M'_{\Theta_2^\Phi},d+1,N)$.
The (division-free) additive
 format (resp. dense format) of the formula defining
$\mathcal{D}^p_\ep(X)$
 is thus bounded by
$$\left(M_{\Omega^R}+M_{\Theta_1^\ep}+M_{\Theta_2^\Phi}
+M_{\Upsilon},N\right) \qquad (\text{resp.} (M'_{\Omega^R}+M'_{\Theta_1^\ep}+M'_{\Theta_2^\Phi}
+M'_{\Upsilon},d+1,N)).$$
\end{proof}

We now relate the thickened semi-algebraic fibered-join and
the thickened diagonal
using a sandwiching argument similar in spirit to that used in
\cite{Hausdorff}.

\subsubsection{Limits of one-parameter families}
In this section, we fix a bounded semi-algebraic set
$\TT\subset \mathbb{R}^k\times \mathbb{R}_+$
such that
$\TT_t$ is closed and  $\TT_t
\subseteq B_{k}(0,R)$ for some
$R\in \mathbb{R}_+$ and all $t>0$.
Let $\TT_\limit$ be as in Notation \ref{not:limit}.

We need the following proposition proved in \cite{Hausdorff}.

\begin{proposition}[\cite{Hausdorff} Proposition 8]
\label{prop:12}
There
exists $\lambda_0>0$ such that for every $\lambda \in (0,\lambda_0]$
there exists a continuous semi-algebraic surjection $f_\lambda: \TT_\lambda
\to
\TT_\limit
$ such that the family of maps
 $\{f_\lambda\}_{0<\lambda\leq \lambda_0}$
satisfies
\begin{enumerate}
\item
$$
\lim_{\lambda\to 0} \max_{\emph{\textbf{x}}\in
  \TT_\lambda}  |\emph{\textbf{x}}-f_\lambda(\emph{\textbf{x}})|=0,
$$
and
\item
for each $\lambda,\lambda'\in (0,\lambda_0)$, $f_\lambda=f_{\lambda'}\circ g$
for some semi-algebraic homeomorphism
$g :\TT_\lambda \to \TT_{\lambda'}$.
\end{enumerate}
\end{proposition}

\hide{
Proposition \ref{prop:12} is proved in \cite{Hausdorff}. However, we
need another result which could be deduced from the proof of Proposition
\ref{prop:12} in \cite{Hausdorff} but not explicitly stated there. In the
appendix, we include a different proof of Proposition \ref{prop:12}, along
with the proof of the intermediate result that we need.
}

\begin{proposition}
\label{prop:13}
There exist
  $\lambda_1$  satisfying $0<\lambda_1\leq \lambda_0$ and semi-algebraic
  functions  $\delta_0,\delta_1: (0,\lambda_1) \rightarrow \mathbb{R}$,
  such that
\begin{enumerate}
\item $0 < \delta_0(\lambda) < \delta_1(\lambda)$, for $\lambda \in (0,\lambda_1)$,
\item  $\lim_{\lambda \to 0}
 \delta_0 (\lambda)=0$,
  $\lim_{\lambda \to 0} \delta_1(\lambda)\neq
  0$,
\item
for  each $\lambda\in (0,\lambda_1)$,
and $\delta, \delta'$ satisfying
$0<\delta_0(\lambda)<\delta<\delta'<  \delta_1(\lambda)$,
the inclusion  $\mathcal{D}^p_{\delta'}(\TT_\lambda)
  \hookrightarrow \mathcal{D}^p_{\delta}(\TT_\lambda)$
induces a semi-algebraic homotopy equivalence.
\end{enumerate}
\end{proposition}

Proposition \ref{prop:13} is adapted from Proposition 20 in
\cite{Hausdorff}
and the proof is identical after replacing
$D^p_\lambda(\delta)$ (defined in \cite{Hausdorff}) with
the semi-algebraic set $\mathcal{D}^p_{\delta}(\TT_\lambda)$ defined above
(Definition \ref{defncalD}).


Let $f_\lambda$, $\lambda\in (0,\lambda_0]$, satisfy the conclusion of Proposition \ref{prop:12}.
As in \cite{Hausdorff}, define for $p\in \mathbb{N}$
\begin{equation}
\eta_p(\lambda)=p(p+1)\left(4R \max_{\mathcal{\textbf{x}} \in T_\lambda} |
\mathcal{\textbf{x}}-f_\lambda(\mathcal{\textbf{x}})| + 2\left(\max_{\mathcal{\textbf{x}} \in T_\lambda } |
\mathcal{\textbf{x}}-f_\lambda(\mathcal{\textbf{x}})|\right)^2\right).
\end{equation}

\hide{
Note that, for $q\leq p$, we have $\eta_q(\lambda)\leq \eta_p(\lambda)$
for each $\lambda\in (0,\lambda_0]$,
and that $\lim_{\lambda \to 0} \eta_p(\lambda)=0$ for each $p\in \mathbb{N}$
by Proposition \ref{prop:12} A.
}
Note that, for every $\lambda\in (0,\lambda_0]$ and every $q\leq p$,
we have $\eta_q(\lambda) \leq \eta_p(\lambda)$.  Additionally, for each $p\in \mathbb{N}$,
$\lim_{\lambda \to 0 } \eta_p(\lambda)=0$ by Proposition \ref{prop:12} A.

Define for $\overline{\textbf{x}}=(\mathcal{\textbf{x}}^0,\dots,\mathcal{\textbf{x}}^p)\in \mathbb{R}^{(p+1)k}$ the sum
$\rho_p(\overline{\textbf{x}})$ as $$\rho_p(\mathcal{\textbf{x}}^0,\dots,\mathcal{\textbf{x}}^p)=\sum_{1\leq i < j
\leq p} |\mathcal{\textbf{x}}^i - \mathcal{\textbf{x}}^j|^2.$$
A special case of this sum corresponding to all $t_i\neq 0$
appears in the formula $\Upsilon^\ep_1$ of Definition \ref{defncalD}
after making the replacement $a_{ij}=|\mathcal{\textbf{x}}^i-\mathcal{\textbf{x}}^j|$.
The next lemma is taken from \cite{Hausdorff} to which we refer
the reader for the proof.

\begin{lemma}[\cite{Hausdorff} Lemma 21]\label{lem:7}  Given $\eta_p(\lambda)$
  and  $f_\lambda:\TT_\lambda \to \TT
  _\limit$ as above, we have
$$|\halfhalfspace \sum_{i<j} |f_\lambda(\emph{\textbf{x}}^i)-f_\lambda(\emph{\textbf{x}}^j)|^2-
  \sum_{i<j} |\emph{\textbf{x}}^i-\emph{\textbf{x}}^j|^2\halfhalfspace| \halfspace \leq \halfspace
  \eta_p(\lambda),
$$
and in particular
$$\rho_p(\emph{\textbf{x}}^0,\dots,\emph{\textbf{x}}^p)\halfhalfspace \leq \halfhalfspace
\rho_p(f_\lambda(\emph{\textbf{x}}^0),\dots,f_\lambda(\emph{\textbf{x}}^p))+\eta_p(\lambda)
\halfhalfspace \leq \halfhalfspace \rho_p(\emph{\textbf{x}}^0,\dots,\emph{\textbf{x}}^p) + 2\eta_p(\lambda).
$$
\end{lemma}

The next proposition follows immediately from Lemma \ref{lem:7},
Definition \ref{defncalJ}, and Definition \ref{defncalD}.

\begin{proposition}
\label{propINC}
For every $\lambda\in (0,\lambda_0)$ and $\ep>0$, we have
  $$ \mathcal{J}^p_{f_\lambda,\ep}(\TT_\lambda) \subseteq
  \mathcal{D}^p_{\ep+\eta_p(\lambda)} (\TT_\lambda) \subseteq
  \mathcal{J}^p_{f_\lambda,\ep+2 \eta_p (\lambda)}(\TT_\lambda).$$
\end{proposition}

Let $\ep_1,\ep_2\in \mathbb{R}_+$ satisfy the conclusions of Proposition \ref{prop:7},
 Proposition \ref{prop:8}, respectively.  Set $\ep_0=\min\{\ep_1,\ep_2\}$.

\begin{proposition}
\label{prop:15}
For any $p\in
  \mathbb{N}$, there exist
$\lambda,\ep,\delta\in \mathbb{R}_+$ such that
$\ep\in (0,\ep_0)$, $\lambda\in (0,\lambda_0)$, and
\hide{
such that $\mathcal{D}^p_\delta(\TT_\lambda)
  \overset{i}{\hookrightarrow} \mathcal{J}^p_{f_\lambda,\ep}(\TT_\lambda)$ and
such that
these inclusions induce
 a homotopy  equivalence
}
 $$\mathcal{D}^p_\delta(\TT_\lambda) \simeq
  \mathcal{J}^p_{f_\lambda,\ep}(\TT_\lambda).$$
\end{proposition}

\begin{proof} We first describe how to choose
$\ep,\ep'\in (0,\ep_0)$, $\lambda\in (0,\lambda_0)$ and $\delta,\delta' \in
(\delta_0(\lambda),\delta_1(\lambda))$
(cf. Proposition \ref{prop:13}) so that $$
\mathcal{D}^p_{\delta'}(\TT_\lambda) \subseteq \mathcal{J}^p_{f_\lambda,\ep}(\TT_\lambda)
\overset{\ast}{\subseteq} \mathcal{D}^p_{\delta}(\TT_\lambda)\subseteq
\mathcal{J}^p_{f_\lambda,\ep'}(\TT_\lambda),$$ and secondly we show
that,
with these choices,
the inclusion $(\ast)$ induces a homotopy equivalence.

Since the limit of $\delta_1(\lambda)-\delta_0(\lambda)$ is not zero
for $0<\lambda<
\lambda_1\leq \lambda_0
$ and $\lambda$ tending to zero, while the
limits of $\eta_p(\lambda)$ and $\delta_0(\lambda)$ are zero
(by  Proposition \ref{prop:13}, Proposition \ref{prop:12} A), we can
choose $0<\lambda<\lambda_0$ which simultaneously satisfies
$$2\eta_p(\lambda)<\fraction{\delta_1(\lambda)-\delta_0(\lambda)}{2} \qquad
\text{and} \qquad \delta_0(\lambda)+4\eta_p(\lambda)<\ep_0.$$
\hide{
 Then, since
$\delta_0(\lambda)>0$ we have the following two inequalities
$$2\eta_p(\lambda)<\fraction{\delta_1(\lambda) + \delta_0(\lambda)}{2}
\qquad \text{and} \qquad
2\eta_p(\lambda)<\delta_1(\lambda)-\delta_0(\lambda).$$
}

Set
$\delta'=\delta_0+\eta_p(\lambda)$,
\hide{
$\ep=\delta+2\eta_p(\lambda)$,
$\delta'=\delta+3\eta_p(\lambda)$, and $\ep'=\delta+4\eta_p(\lambda)$.
}
$\ep=\delta_0+2\eta_p(\lambda)$,
$\delta=\delta_0+3\eta_p(\lambda)$, and $\ep'=\delta_0+4\eta_p(\lambda)$.
From Proposition \ref{propINC} we have the following inclusions,
$$
\mathcal{D}^p_{\delta'}(\TT_\lambda) \overset{i}{\hookrightarrow}
\mathcal{J}^p_{f_\lambda,\ep}(\TT_\lambda) \overset{j}{\hookrightarrow}
\mathcal{D}^p_{\delta}(\TT_\lambda)\overset{k}{\hookrightarrow}
\mathcal{J}^p_{f_\lambda,\ep'}(\TT_\lambda).$$

Furthermore,
it is easy to see that $\delta,
\delta'\in (\delta_0(\lambda),\delta_1(\lambda))$ and that
$\ep,\ep'\in (0,\ep_0)$, and so we have that both $j\circ i$ and
$k\circ j$ induce semi-algebraic homotopy equivalences (Proposition \ref{prop:13},
Proposition \ref{prop:8} resp.).

For each $\textbf{z} \in \mathcal{D}^p_{\delta'}(\TT_\lambda)$ we have the following
diagram between the homotopy groups.

$$
\xymatrix{ \boldsymbol{\pi}_\ast(\mathcal{D}^p_{\delta'}(\TT_\lambda),\textbf{z})
\ar[rr]^{(j\circ i )_\ast}_{\cong} \ar[rd]^{i_\ast} & &
 \boldsymbol{\pi}_\ast(\mathcal{D}^p_{\delta}(\TT_\lambda), \textbf{z}) \ar[rd]^{k_\ast} & \\ &
 \boldsymbol{\pi}_\ast(\mathcal{J}^p_{f_\lambda,\ep}(\TT_\lambda),\textbf{z}) \ar[rr]^{(k\circ j)_\ast}_{\cong}
 \ar[ru]^{j_\ast} & & \boldsymbol{\pi}_\ast(\mathcal{J}^p_{f_\lambda,\ep'}(\TT_\lambda),
\textbf{z})}
$$
where we have identified $\textbf{z}$ with its images under the various
inclusion maps.

Since $(j\circ i)_\ast=j_\ast \circ i_\ast$, the surjectivity of
$(j\circ i)_\ast$ implies that $j_\ast$ is surjective, and similarly
$(k\circ j)_\ast$ is injective
ensures that $j_\ast$ is injective.
Hence, $j_\ast$ is an isomorphism as required.

This implies that
the inclusion map
\hide{
between $\mathcal{D}^p_\delta (
\TT_\lambda)$ and $\mathcal{J}^p_{f_\lambda,\ep}(\TT_\lambda)$
}
$\mathcal{J}^p_{f_\lambda,\ep}(\TT_\lambda)\overset{j}{\hookrightarrow} \mathcal{D}^p_\delta(\TT_\lambda)$
 is a
weak homotopy equivalence
(see \cite[page 181]{Whitehead}). Since
both spaces have the structure of a finite
CW-complex, every weak equivalence is in fact a homotopy equivalence
(\cite[Theorem 3.5, p. 220]{Whitehead}).
\end{proof}


We now prove Proposition \ref{prop:main}.
\begin{proof}[Proof of Proposition \ref{prop:main}]
Let
$\TT\subset \mathbb{R}^k \times \mathbb{R}_+$
such that $\TT_\lambda$ is closed and $\TT_\lambda\subset B_k(0,R)$ for some $R\in \mathbb{R}$
and all
$\lambda\in \mathbb{R}_+$.
Applying Proposition \ref{prop:15},
we have that there exist $\lambda\in
(0,\lambda_0)$ and $\ep\in (0,\ep_0)$
such that
 the sets $\mathcal{D}^p_\delta(\TT_\lambda)\simeq
\mathcal{J}^p_{f_\lambda,\ep}(\TT_\lambda)$ are
semi-algebraically
homotopy equivalent.  Also, by
Proposition \ref{prop:7} the sets
$\mathcal{J}^p_{f_\lambda,\ep}(\TT_\lambda)\simeq \mathcal{J}^p_{f_\lambda}(\TT_\lambda)$ are
semi-algebraically homotopy equivalent.
By Proposition \ref{prop:5} and Proposition \ref{prop:12} the map
$J(f_\lambda):\mathcal{J}^p_{f_\lambda}(\TT_\lambda)\twoheadrightarrow \TT_0$ induces a
$p$-equivalence.

Thus we have the following sequence of homotopy equivalences and
$p$-equivalence.
\begin{equation}{\label{eqn:mainresult}}
\mathcal{D}^p_{\delta}(\TT_\lambda)\simeq
\mathcal{J}^p_{f_\lambda,\ep}(\TT_\lambda)
\simeq\mathcal{J}^p_{f_\lambda}(\TT_\lambda)
\halfspace {{\widetilde{\to\halfbackspace \succ}_p}}\halfspace \TT_{\limit}\end{equation}

The first homotopy equivalence follows from Proposition \ref{prop:15},
the second from Proposition \ref{prop:7}, and the last $p$-equivalence
is a consequence of Propositions \ref{prop:5} and \ref{prop:12}.
The bound on the format
of the
formula defining $\mathcal{D}^p:=\mathcal{D}^p_\delta(\TT_\lambda)$ follows from
Proposition \ref{prop:9}.  This finishes the proof.
\end{proof}

\begin{proof}[Proof of Theorem \ref{thm:main_weak}]
The theorem follows directly from Proposition \ref{prop:main},
Theorem \ref{thm:additive} and
Proposition \ref{prop:top_basic} after choosing $p=k+1$.
\end{proof}


\section{Proofs of Theorem \ref{thm:main1} and Theorem \ref{thm:main}
\label{sec:main1}}

\subsection{Algebraic preliminaries}
We start
with a lemma
that provides a slightly different characterization of
additive complexity from that given in Definition
\ref{def:rational_additive}.
Roughly speaking the lemma states that any given
 additive representation of a given polynomial $P$ can be
modified without changing its length to another  additive
representation of $P$ in which any negative exponents occur only
in the very last step. This simplification will be very useful in
what follows.

\begin{lemma}\cite[page 152]{Dries}
\label{lem:equivalence}
For any $P\in \mathbb{R}[X_1,\ldots,X_k]$ and $a\in \mathbb{N}$ we have
  $P$ has  additive complexity at most $a$ if and only if
there exists a sequence of equations (*)
\begin{itemize}
\item[(i)] $Q_1=u_1X_{1}^{\alpha_{11}} \cdots X_{k}^{\alpha_{1k}} +
  v_1X_{1}^{\beta_{11}} \cdots X_{k}^{\beta_{1k}}$,\\ where $u_1, v_1
  \in \mathbb{R}$, and $\alpha_{11}, \ldots ,\alpha_{1k}, \beta_{11}, \ldots ,
  \beta_{1k} \in \N$;

\item[(ii)] $Q_j=u_jX_{1}^{\alpha_{j1}} \cdots X_{k}^{\alpha_{j k}}
  \prod_{1 \le i \le j-1}Q_{i}^{\gamma_{j i}} + v_jX_{1}^{\beta_{j1}}
  \cdots X_{k}^{\beta_{j k}}\prod_{1 \le i \le
    j-1}Q_{i}^{\delta_{ji}}$,\\ where $1 < j \le a$, $u_j, v_j \in
  \mathbb{R}$, and $\alpha_{j1}, \ldots ,\alpha_{j k}, \beta_{j1}, \ldots ,
  \beta_{j k}, \gamma_{ji}, \delta_{ji} \in \N$ for $1 \le i <j$;

\item[(iii)] $P= cX_{1}^{\zeta_{1}} \cdots X_{k}^{\zeta_{k}}\prod_{1
  \le j \le a}Q_{j}^{\eta_{j}}$,\\ where $c \in \mathbb{R}$, and $\zeta_1,
  \ldots , \zeta_k, \eta_1, \ldots ,\eta_a \in \mathbb{Z}$.
\end{itemize}
\end{lemma}


\begin{remark}
Observe that in Lemma \ref{lem:equivalence}
all exponents other than those in line (iii) are in $\mathbb{N}$ rather
than in $\mathbb{Z}$ (cf. Definition~\ref{def:rational_additive}).
Observe also that if
a polynomial $P$ satisfies the conditions of the lemma, then
it has  additive complexity at most $a$.
\end{remark}



\newcommand{\hatP}{\widehat{P}}
\newcommand{\hatQ}{\widehat{Q}}

\subsection{The algebraic case}
Before proving Theorem \ref{thm:main1} it is useful to first consider the
algebraic case separately, since the main technical ingredients
used in the proof of Theorem \ref{thm:main1} are more clearly
visible in this case. With this in mind, in this section we consider the
algebraic case and prove the following theorem, deferring the proof in the
general semi-algebraic case till the next section.

\begin{theorem}{
\label{thm:algebraic}}
 The number of
distinct
homotopy types of
  $\Zer(F,\mathbb{R}^k)$ amongst all polynomials $F\in \mathbb{R}[X_1,\dots,X_k]$ having
    additive complexity at most $a$ does not exceed
$$ 2^{O(k(k^2+a))^8} =
2^{(k+a)^{O(1)}}.$$
\end{theorem}

Before proving Theorem \ref{thm:algebraic} we need a few preliminary results.

\begin{proposition}{\label{prop:sectionfivemain}}
Let $F,P,Q\in \mathbb{R}[\X]$
such that $FQ=P$,
$R \in \mathbb{R}_+$, and define
\begin{equation}{\label{eqn:tee}}
\TT:=\{(\emph{\textbf{x}},t)\in
\mathbb{R}^k \times \mathbb{R}_+ | \; 
P^2(\emph{\textbf{x}}) \leq t(Q^2(\emph{\textbf{x}})-t^N) \wedge 
|\emph{\textbf{x}}|^2\leq R^2
\},
\end{equation}
where $N = 2 \deg(Q)+1$.
Then, using Notation \ref{not:limit}
\[
\TT_{\limit}=\Zer(F,\mathbb{R}^k) \cap \overline{B_k(0,R)}.
\]
\end{proposition}

Before proving Proposition \ref{prop:sectionfivemain}
we first discuss an illustrative example.

\begin{example}
\label{eg:main}
Let
\[
F_1=X(X^2+Y^2-1),
\]
\[F_2=X^2+Y^2-1.
\]
Also, let
\[
P_1=X^2(X^2+Y^2-1),
\]
\[
P_2= X(X^2+Y^2-1),
\]
and
\[
Q_1=Q_2=X.
\]

For $i=1,2$, and $R >0$, let
\[
\TT^i =\{(\mathcal{\textbf{x}},t)\in
\mathbb{R}^k\times \mathbb{R}_+
| \; 
P_i^2(\mathcal{\textbf{x}}) \leq t(Q_i^2(\mathcal{\textbf{x}})-t^N)
\wedge 
|\mathcal{\textbf{x}}|^2
\leq R^2
\}
\]
as in Proposition \ref{prop:sectionfivemain}.

In Figure \ref{fig:example}, we display
from left to right,
$\Zer(F_1,\mathbb{R}^2)$, $\TT^1_\ep$, $\Zer(F_2,\mathbb{R}^2)$ and
and $\TT^2_\ep$, respectively (where $\ep=.005$ and $N=3$).
Notice that, for $i=1,2$ and any fixed $R>0$, the semi-algebraic
set $T^i_\ep$ approaches (in the
sense of Hausdorff
distance)
the set $\ZZ(F_i,\mathbb{R}^2) \cap \overline{B_2(0,R)}$
as $\ep \rightarrow 0$.

\begin{figure}[h!]
\vspace{-150 pt}
\centering
  \subfloat{\label{fig:a} \includegraphics[bb=50 0 500 700,
      clip,width=.18\textwidth]{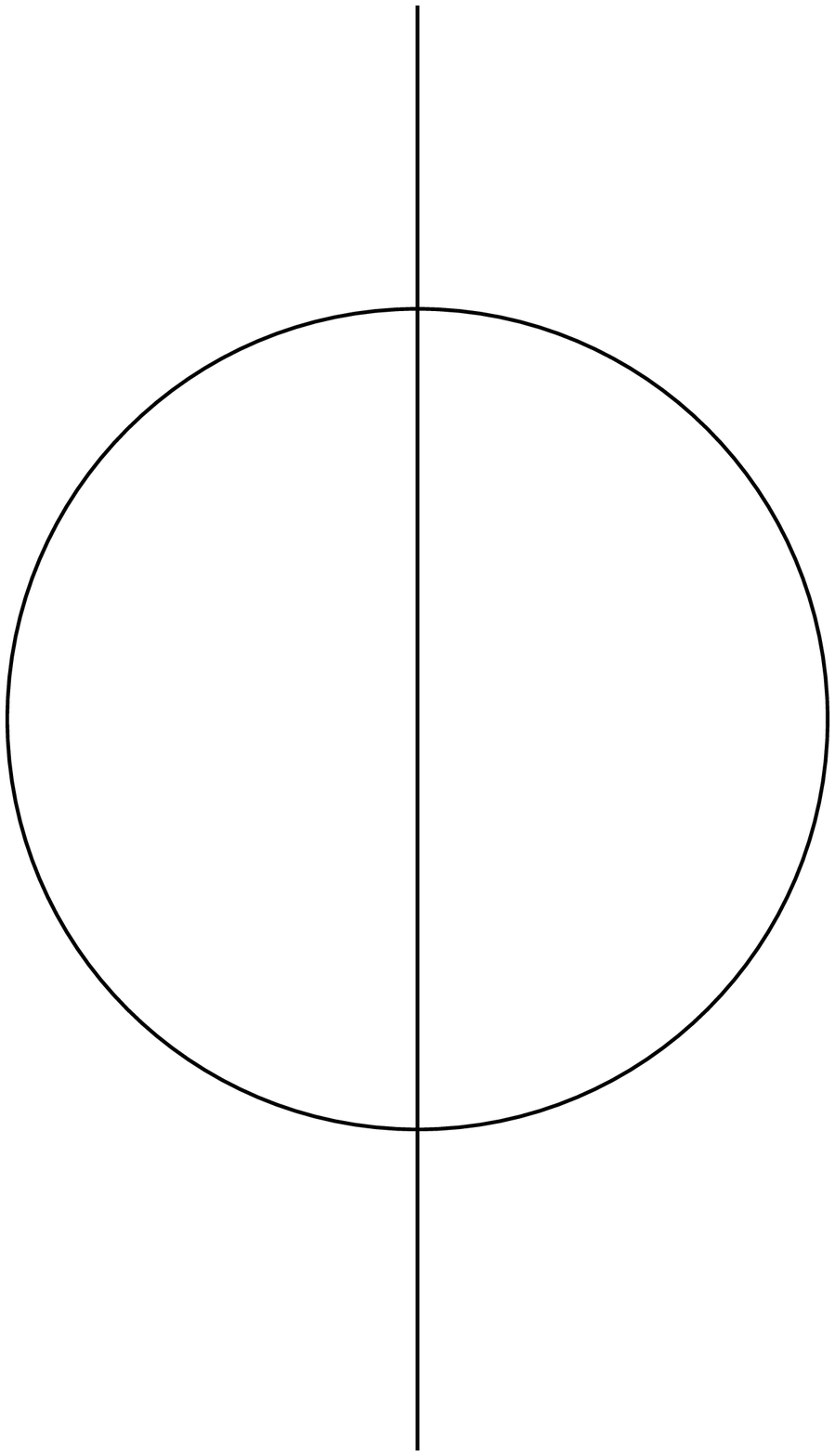}}
  \subfloat{\label{fig:b} \includegraphics[bb=20 0 230 500,
      clip,width=.27\textwidth]{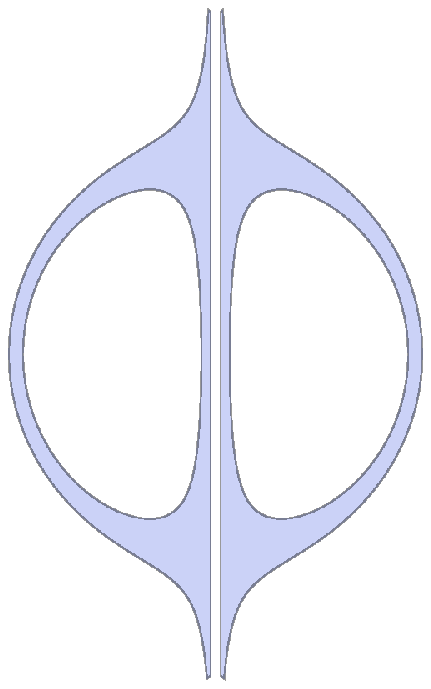}}
  \subfloat{\label{fig:c} \includegraphics[bb=80 40 500 700,clip,
      width=.18\textwidth]{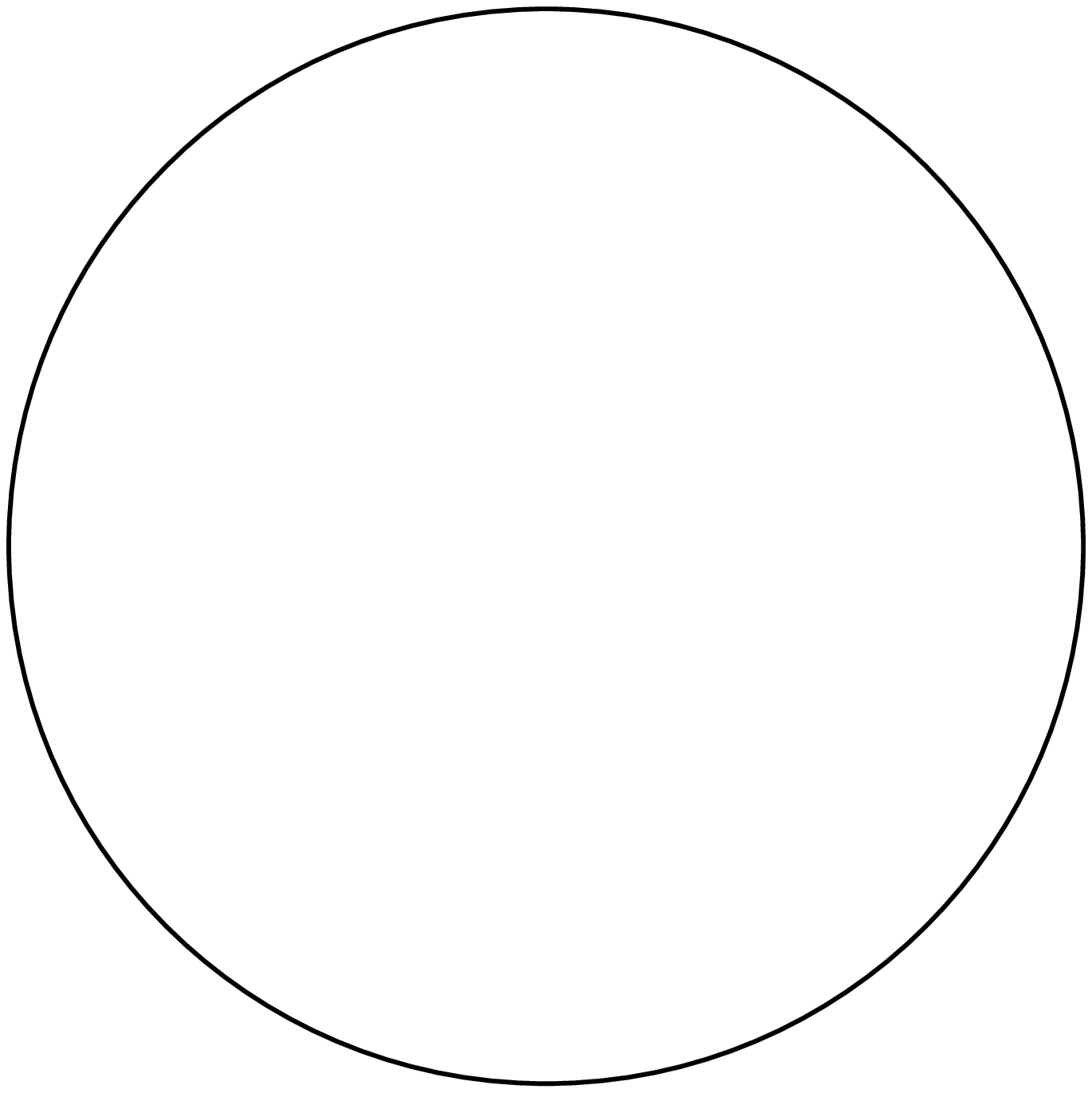}}
  \subfloat{\label{fig:d} \includegraphics[bb=40 0 200 500 ,clip,
      width=.22\textwidth]{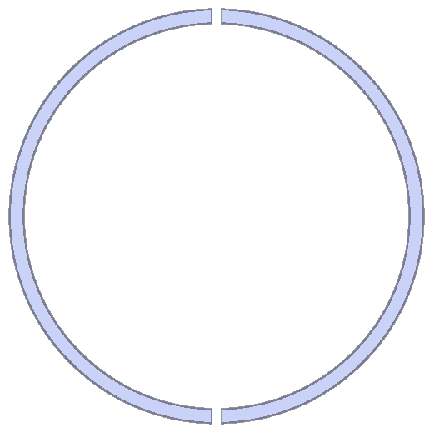}}
  \caption{Two examples.}
  \label{fig:example}
\end{figure}

\end{example}

We now prove Proposition \ref{prop:sectionfivemain}.

\begin{proof}[Proof of Proposition \ref{prop:sectionfivemain}]
We show both inclusions.  First let $\mathcal{\textbf{x}}\in \TT_{\limit}$, and we
show that $F(\mathcal{\textbf{x}})=0$.
In particular, we prove that $0\leq F^2(\mathcal{\textbf{x}})<\ep$ for every $\ep>0$.

Let $\ep>0$.
Since $F^2$ is
continuous,
there exists $\delta>0$ such that
\begin{equation}{\label{eqn:sectfiveCONTF}}|\mathcal{\textbf{x}}-\textbf{y}|^2<\delta \quad
\implies \quad |F^2(\mathcal{\textbf{x}})-F^2(\textbf{y})|<\fraction{\ep}{2}.\end{equation} After
possibly making $\delta$ smaller we can suppose that
$\delta<\fraction{\ep^2}{4}$.

From the definition of $\TT_{\limit}$
(cf. Notation \ref{not:limit}),
we have that
\begin{equation}{\label{eqn:Tzero}} \TT_{\limit} = \{ \mathcal{\textbf{x}} \mid \quad (\forall
\delta)(\delta>0\implies (\exists t)(\exists \textbf{y})(\textbf{y}\in \TT_t \wedge
|\mathcal{\textbf{x}}-\textbf{y}|^2+t^2<\delta))\}.\end{equation}

Since
$\mathcal{\textbf{x}}\in \TT_{\limit}$,
there exists
$t\in \mathbb{R}_+$
and $\textbf{y}\in \TT_t$ such that $|\mathcal{\textbf{x}}-\textbf{y}|^2+t^2<\delta$, and in
particular both $|\mathcal{\textbf{x}}-\textbf{y}|^2<\delta$ and
$t^2<\delta<\fraction{\ep^2}{4}$.  The former inequality implies that
$|F^2(\mathcal{\textbf{x}})-F^2(\textbf{y})|<\fraction{\ep}{2}$.  The latter inequality implies
$t<\fraction{\ep}{2}$, and this together with $\textbf{y}\in \TT_t$
implies
$$\begin{aligned} &P^2(\textbf{y})\leq t (Q^2(\textbf{y})-t^N)\\ \implies &
F^2(\textbf{y})Q^2(\textbf{y})\leq t (Q^2(\textbf{y})-t^N) \\ \implies & 0\leq
F^2(\textbf{y})\leq t-\frac{t^{N+1}}{Q^2(\textbf{y})}<t \\ \implies & 0\leq
F^2(\textbf{y})<\fraction{\ep}{2}.
\end{aligned}$$
So, $F^2(\textbf{y})<\fraction{\ep}{2}$.  Finally, note that $|F^2(\mathcal{\textbf{x}})|\leq
|F^2(\mathcal{\textbf{x}}) - F^2(\textbf{y})| + |F^2(\textbf{y})| < \fraction{\ep}{2} +
\fraction{\ep}{2} = \ep$.

We next prove the other inclusion, namely we show
$\Zer(F,\mathbb{R}^k)\cap \overline{B_k(0,R)}\subseteq \TT_{\limit}$.
Let $\mathcal{\textbf{x}}\in \Zer(F,\mathbb{R}^k)\cap \overline{B_k(0,R)}$.
We fix $\delta>0$ and  show that there exists
$t\in \mathbb{R}_+$
and $\textbf{y}\in \TT_t$
such that $|\mathcal{\textbf{x}}-\textbf{y}|^2+t^2<\delta$ (cf. Equation \ref{eqn:Tzero}).

There are two cases to consider.
\begin{itemize}
\item[$Q(\mathcal{\textbf{x}})\neq 0$:]
Since $Q(\mathcal{\textbf{x}})\neq 0$,
there exists $t>0$ such that $Q^2(\mathcal{\textbf{x}})\geq t^N$ and $t^2<\delta$.
Now, $\mathcal{\textbf{x}}\in \TT_t$ and $$|\mathcal{\textbf{x}}-\mathcal{\textbf{x}}|^2+t^2 \ = \ t^2 \ <\delta,$$ so
setting $\textbf{y}=\mathcal{\textbf{x}}$ we see that $\textbf{y}\in \TT_t$ and $|\mathcal{\textbf{x}}-\textbf{y}|+t^2<\delta$.
Thus, $\mathcal{\textbf{x}}\in \TT_{\limit}$ as desired.

\item[$Q(\mathcal{\textbf{x}})=0$:]

Let $\textbf{v}\in \mathbb{R}^k$ be generic, and
denote
$\widehat{P}(U)=P(\mathcal{\textbf{x}}+U\textbf{v})$,
$\widehat{Q}(U)=Q(\mathcal{\textbf{x}}+U\textbf{v})$, and
$\widehat{F}(U)=F(\mathcal{\textbf{x}}+U\textbf{v})$.
Note that
\begin{eqnarray}
\label{eqn:multiplicity}
\widehat{P} = \widehat{F}\widehat{Q}, \nonumber \\
\widehat{P}(0) = \widehat{Q}(0) = \widehat{F}(0) = 0.
\end{eqnarray}

If $F$ is not the zero polynomial, then neither is $\widehat{P}$, since $\textbf{v}$ is generic.
Indeed,
assume $F$ is not identically zero, and hence $P$
is not identically zero.
In order to prove that $\widehat{P}$ is not
identically zero for a generic choice of
$\textbf{v}$,
write $P = \sum_{0 \leq i \leq d} P_i$ where $P_i$ is the homogeneous
part of $P$ of degree $i$, and $P_d$ not identically zero. Then, it is
easy to see that $\widehat{P}(U) = P_d(\textbf{v}) U^d + \mathrm{\mbox{lower
    degree terms}}$.  Since $\mathbb{R}$ is an infinite field, a generic
choice of $\textbf{v}$ will avoid the set of zeros of $P_d$,
and thus,
$\widehat{P}$ is not identically zero.

We further require
that $\mathcal{\textbf{x}}+t\textbf{v}\in B_k(0,R)$ for $t>0$ sufficiently small.  For
generic $\textbf{v}$, this is true for either $\textbf{v}$ or $-\textbf{v}$, and so
after possibly replacing $\textbf{v}$ by $-\textbf{v}$
(and noticing that since $P_d$ is homogeneous we have $P_d(\textbf{v}) =
(-1)^d P_d(-\textbf{v})$) we may assume $\mathcal{\textbf{x}}+t\textbf{v}\in B_k(0,R)$ for $t>0$
sufficiently small.
Let $t_0>0$ be such that $\mathcal{\textbf{x}}+t\textbf{v}\in B_k(0,R)$ for $0<t< t_0$.

Denoting by $\nu = \mathrm{mult}_0(\hatP)$ and
$\mu = \mathrm{mult}_0(\hatQ)$, we have from (\ref{eqn:multiplicity})
that $\nu > \mu$.
Let
$$\begin{aligned} \hatP(U)&=\sum_{i=\nu}^{\deg_U \hatP} c_iU^i = U^\nu
  \cdot \sum_{i=0}^{\deg_U \hatP - \nu} c_{\nu+i}U^i = c_\nu U^\nu +
  \text{ (higher order terms)}, \\ \hatQ(U)&=\sum_{i=\mu}^{\deg_U \hatQ} d_iU^i = U^\mu \cdot
  \sum_{i=0}^{\deg_U \hatQ - \mu} d_{\mu+i}U^i = d_\mu U^\mu +
  \text{ (higher order terms)} \end{aligned}$$
where $c_\nu,d_\mu \neq 0$.

Then we have $$\begin{aligned} \hatP^2(U)&=c_\nu^2 U^{2\nu} +
  \text{ (higher order terms)}, \\
  \hatQ^2(U)&=d_\mu^2 U^{2\mu} +
  \text{ (higher order terms)}, \\
  D(U):=U(\hatQ^2(U)-U^N)&= U(d_\mu^2 U^{2\mu} +
  \text{ (higher order terms)}-U^N), \\
  D(U)-\hatP^2(U)&= d_\mu^2 U^{2\mu+1} + \text{ (higher order terms)} -
      U^{N+1}.
\end{aligned}$$

Since $\mu \leq \deg(Q)$ and $N = 2 \deg(Q) + 1$,
we have that $2\mu+1 < N+1$. Hence,
there exists $t_1 \in \mathbb{R}_+$ such that for
each $t$, $0 < t < t_1$,
we have $D(t)-\hatP^2(t)\geq 0$.
Thus,
$\mathcal{\textbf{x}}+t\textbf{v}\in \TT_t$ for each $t$, $0<t<\min\{t_0,t_1\}$.
Let $t_2 = (\frac{\delta}{|\textbf{v}|^2+1})^{1/2}$ and note that
for all $t$, $0<t< t_2$, we have $(|\textbf{v}|^2+1)t^2<\delta$.
Finally, if $t$ satisfies
$0<t< \min\{t_0,t_1,t_2\}$ then
$\mathcal{\textbf{x}}+t\textbf{v}\in \TT_t$,
and
$$|\mathcal{\textbf{x}}-(\mathcal{\textbf{x}}+t\textbf{v})|^2+t^2\ = \ (|\textbf{v}|^2+1)t^2 <
\delta.$$
Hence, setting $\mathcal{\textbf{y}}=\mathcal{\textbf{x}}+t\textbf{v}$ (cf. Equation \ref{eqn:Tzero}) we have shown that $\mathcal{\textbf{x}}\in \TT_\limit$ as desired.

The case where $F$ is the zero polynomial is straightforward.

\end{itemize}
\end{proof}

\begin{proof}[Proof of Theorem \ref{thm:algebraic}]
For each $F\in \mathbb{R}[X_1,\ldots,X_k]$, by the conical structure at infinity of semi-algebraic sets
(see for instance \cite[page 188]{BPRbook2}), we have
that
there exists $R_F\in \mathbb{R}_+$ such that, for every $R>R_F$, the
semi-algebraic sets $\Zer(F,\mathbb{R}^k)\cap \overline{B_k(0,R)}$ and $\Zer(F,\mathbb{R}^k)$
are semi-algebraically
homeomorphic.

Let $\ell\in \mathbb{N}$, $F_1,\ldots,F_\ell\in \mathbb{R}[X_1,\ldots,X_k]$ such that each $F_i$ has additive complexity at most $a$ and, for every $F$ having additive complexity at most $a$, the algebraic sets $\Zer(F,\mathbb{R}^k),\Zer(F_i,\mathbb{R}^k)$ are semi-algebraically homeomorphic for some $i$, $1\leq i \leq \ell$
(see, for example \cite[Theorem 3.5]{Dries}).  Let $R=\max_{1\leq i \leq \ell} \{R_{F_i}\}$.

Let $F\in \{F_i\}_{1\leq i \leq \ell}$.
By Lemma \ref{lem:equivalence} there exists polynomials $P,Q \in \mathbb{R}[X_1,\ldots,X_k]$ such that
$FQ=P$, and such that $P,Q$
satisfies $P^2-T(Q^2-T^N)\in \mathbb{R}[X_1,\ldots,X_k,T]$ has division-free additive complexity bounded by $a+2$.
Let
$$
\TT=\{(\mathcal{\textbf{x}},t)\in
\mathbb{R}^k \times \mathbb{R}_+| \;
P^2(\mathcal{\textbf{x}}) \leq t(Q^2(\mathcal{\textbf{x}})-t^N) \wedge 
|\mathcal{\textbf{x}}|^2
\leq R^2
\}.
$$

By Proposition \ref{prop:sectionfivemain}
 we have that
$
\TT_{\limit} =
\Zer(F,\mathbb{R}^k)
\cap \overline{B_k(0,R)}.
$
Note the one-parameter semi-algebraic
family $\TT$ (where the last co-ordinate is the parameter) is
described by a formula having division-free
additive format
$(a+k+2,k+1)$.

By Theorem \ref{thm:main_weak}
we obtain a collection of
semi-algebraic sets $\mathcal{S}_{k,a+k+2}$
such that $\TT_{\limit}$, and hence $\Zer(F,\mathbb{R}^k)$, is homotopy equivalent
to some $S\in \mathcal{S}_{k,a+k+2}$
and $\#\mathcal{S}_{k,a+k+2} = 2^{O(k(
k^2
+a))^8}$,
which proves the theorem.

\end{proof}

\subsection{The semi-algebraic case}
\label{subsec:semi-algebraic}
We first prove a generalization of Proposition \ref{prop:sectionfivemain}.

\begin{notation} Let $\X=(X_1,\dots,X_k)$ be a block of variables
and $\textbf{k}=(k_1,\dots,k_n)\in \mathbb{N}^n$ with
$\sum_{i=1}^n k_i=k$.  Let $\textbf{r}=(r_1,\dots,r_n)\in \mathbb{R}^n$ with $r_i>0$,
$i=1,\dots,n$.  Let $B_{\textbf{k}}(0,\textbf{r})$ denote the product
$$B_{\textbf{k}}(0,\textbf{r}):= B_{k_1}(0,r_1)\times \dots \times B_{k_n}(0,r_n).$$
\end{notation}

\begin{notation}{\label{not:reali}}
For any first order formula $\Phi$
with $k$ free variables, we denote by $\Reali(\Phi)$ the semi-algebraic
subset of $\mathbb{R}^k$ defined by $\Phi$.
\hide{
Additionally, if $\mathcal{P}\subset \mathbb{R}[X_1,\dots,X_k]$
consists of the polynomials appearing in $\Phi$, then we call $\Phi$ a $\mathcal{P}$-formula.
}
\end{notation}

\begin{proposition}
\label{prop:level1}
Let $F_1,\dots,F_s,P_1,\dots,P_s,Q_1,\dots,Q_s\in
\mathbb{R}[\X^1,\dots,\X^n]$, $\mathcal{P}=\{F_1,\dots,F_s\}$
such that $F_iQ_i=P_i$, for all $i=1,\ldots,s$.  Suppose
$\X^i=(X^i_1,\dots,X^i_{k_i})$ and let \emph{$\textbf{k}=(k_1,\dots,k_n)$}.
Suppose $\phi$ is a $\mathcal{P}$-formula containing no negations and no
inequalities.  Let $$\begin{aligned} \bar{P_i}:=&P_i\prod_{j\neq
    i}Q_j, \\ \bar{Q}:=&\prod_j Q_j, \end{aligned}$$ and let
$\bar{\phi}$ denote the formula
obtained
from $\phi$
by replacing each $F_i=0$ with
$$\bar{P_i}^2 - U(\bar{Q}^2-U^N) \leq 0,$$
where $U$ is the last variable of $\bar{\phi}$,
$N=2\deg (\bar{Q})+1$.  Then,
for every $\emph{\textbf{r}}=(r_1,\dots,r_n)\in \mathbb{R}_+^n$,
we have (cf. Notation \ref{not:reali} and Notation \ref{not:limit})
\begin{equation}
\label{eqn:level1}
\Reali\left(\bigwedge_{i=1}^n (|\X^i|^2\leq r_i^2) \wedge \bar{\phi}
\wedge U>0
\right)_{\limit}=\Reali(\phi) \cap \overline{B_{\emph{\textbf{k}}}(0,\emph{\textbf{r}})}.
\end{equation}
\end{proposition}

\begin{proof} We follow the proof of Proposition \ref{prop:sectionfivemain}.
The only case which is not immediate is the case $\mathcal{\textbf{x}}\in
\Reali(\phi)\cap \overline{B_{\textbf{k}}(0,\textbf{r})}$ and $\bar{Q}(\mathcal{\textbf{x}})=0$.

\newcommand{\barphi}{\bar{\phi}}
\newcommand{\barphialpha}{{\bar{\phi}_{\alpha}}}

Suppose $\mathcal{\textbf{x}}\in \Reali(\phi)\cap \overline{B_{\textbf{k}}(0,\textbf{r})}$ and
that $\bar{Q}(\mathcal{\textbf{x}})=0$.  Since $\phi$ is a formula containing no
negations and no inequalities, it consists of conjunctions and
disjunctions of equalities.  Without loss of generality we can assume
that $\phi$ is written as a disjunction of conjunctions, and still
without negations.  Let
$$\phi = \bigvee_{\alpha} \phi_\alpha$$ where $\phi_\alpha$ is a
conjunction of equations.  As above let $\barphialpha$ be the formula
obtained from $\phi_\alpha$ after replacing each $F_i=0$ in
$\phi_\alpha$ with
$$\bar{P_i}^2\leq U(\bar{Q}^2-U^N),$$ $N=2\deg (\bar{Q})+1$.

We have
$$\begin{aligned} \Reali\left(\bigwedge_{i=1}^p (|\X^i|^2\leq r_i^2
  )\wedge \barphi
\wedge U>0
 \right)_{\limit} & =\Reali\left( \bigwedge_{i=1}^p
  (|\X^i|^2\leq r_i^2 )\wedge \left(\bigvee_\alpha \barphialpha \right)
\wedge U>0
  \right)_{\limit} \\ & =\Reali\left( \bigvee_{\alpha} \bigwedge_{i=1}^p
  (|\X^i|^2\leq r_i^2 )\wedge \barphialpha
\wedge U>0
\right)_{\limit} \\ & =
  \bigcup_\alpha \Reali\left(\bigwedge_{i=1}^p (|\X^i|^2\leq r_i^2 )\wedge
  \barphialpha
\wedge U>0
\right)_{\limit}.
\end{aligned}$$

In order to show that $\mathcal{\textbf{x}}\in
\Reali\left(\bigwedge_{i=1}^p (|\X^i|^2\leq r_i^2)\wedge \barphi
\wedge U>0
\right)_{\limit}$ it now suffices to show that if $\mathcal{\textbf{x}} \in
\Reali(\phi_\alpha)\cap \overline{B_{\textbf{k}}(0,\textbf{r})}$ and
$\bar{Q}(\mathcal{\textbf{x}})=0$, then $\mathcal{\textbf{x}} $
belongs to $ \Reali\left(\bigwedge_{i=1}^p
(|\X^i|^2\leq r_i^2 )\wedge \barphialpha
\wedge U>0
\right)_{\limit}$.

Let $\mathcal{\textbf{x}}\in \Reali(\phi_\alpha)\cap \overline{B_{\textbf{k}}(0,\textbf{r})}$ and suppose
$\bar{Q}(\mathcal{\textbf{x}})=0$. Let $\mathcal{Q}\subseteq \mathcal{P}$
consist of the polynomials of $\mathcal{P}$ appearing
in $\phi_\alpha$.
Let $\textbf{v}\in \mathbb{R}^k$ be generic, and
denote
$\widehat{P_i}(U)=\bar{P_i}(\mathcal{\textbf{x}}+U\textbf{v})$,
$\widehat{Q}(U)=\bar{Q}(\mathcal{\textbf{x}}+U\textbf{v})$, and
$\widehat{F_i}(U)=\bar{F}(\mathcal{\textbf{x}}+U\textbf{v})$.
Note that
\begin{eqnarray}
\label{eqn:multiplicity-two}
\widehat{P_i} = \widehat{F_i}\widehat{Q}, \nonumber \\
\widehat{P_i}(0) = \widehat{Q}(0) = \widehat{F_i}(0) = 0.
\end{eqnarray}

\newcommand{\hatPi}{\widehat{P_i}} As in the proof of Proposition
\ref{prop:sectionfivemain}, if $F_i\in \mathcal{Q}$ is not the zero
polynomial then
$\widehat{P_i}$ is not
identically zero.  Since $\phi_\alpha$ consists of a conjunction of
equalities and
$$\bigwedge_{F\in \mathcal{Q} \atop F\not\equiv 0} F=0 \iff \bigwedge_{F\in
  \mathcal{Q}} F=0,$$ we may assume that $\mathcal{Q}$ does not contain the zero
polynomial.  Under this assumption,
we have that for every $F_i\in \mathcal{Q}$ the univariate polynomial
$\widehat{P_i}$ is not identically zero.  As in the proof of
Proposition \ref{prop:sectionfivemain},
there exists $t_0\in \mathbb{R}_+$ such
that for all $t$, $0<t<t_0$, we have
$\mathcal{\textbf{x}}+t\textbf{v}\in B_{\textbf{k}}(0,\textbf{r})$.
Denoting by $\nu_i = \mathrm{mult}_0(\hatPi)$ and $\mu =
\mathrm{mult}_0(\hatQ)$, we have from (\ref{eqn:multiplicity-two}) that
$\nu_i > \mu$ for all $i=1,\dots,s$.

Let
$$\begin{aligned} \hatPi(U)&=\sum_{j=\nu_i}^{\deg_U \hatPi} c_jU^j =
  U^{\nu_i} \cdot \sum_{j=0}^{\deg_U \hatPi - \nu_i} c_{\nu_i+j}U^j =
  c_{\nu_i} U^{\nu_i} + \text{ (higher order terms)},
  \\ \hatQ(U)&=\sum_{j=\mu}^{\deg_U \hatQ} d_j^j = U^\mu \cdot
  \sum_{j=0}^{\deg_U \hatQ - \mu} d_{\mu+j}U^j = d_\mu U^\mu + \text{
    (higher order terms)} \end{aligned}$$ where $d_\mu \neq 0$ and
$c_{\nu_i}\neq 0 $.

Then we have $$\begin{aligned} \hatPi^2(U)&=c_{\nu_i}^2 U^{2\nu_i} +
  \text{ (higher order terms)}, \\
  \hatQ^2(U)&=d_\mu^2 U^{2\mu} +
  \text{ (higher order terms)}, \\
  D(t):=U(\hatQ^2(U)-U^N)&= U(d_\mu^2 U^{2\mu} +
  \text{ (higher order terms)}-U^N), \\
  D(U)-\hatPi^2(U)&= d_\mu^2 U^{2\mu+1} + \text{ (higher order terms)} -
      U^{N+1}.
\end{aligned}$$

Since $\mu \leq \deg(\bar{Q})$ and $N = 2 \deg(\bar{Q}) + 1$,
we have that $2\mu+1 < N+1$. Hence, there exists $t_{1,i}\in \mathbb{R}_+$ such that for all $t$, $0 < t < t_{1,i}$,
we have that $D(t)-\hatPi^2(t)\geq 0$, and thus
$\mathcal{\textbf{x}}+t\textbf{v}$ satisfies
$$\bar{P}_i^2(\mathcal{\textbf{x}}+t\textbf{v})\leq t(\bar{Q}^2(\mathcal{\textbf{x}}+t\textbf{v})-t^N).$$
Let $t_1=\min\{t_{1,1},\dots,t_{1,s}\}$.
Let $t_2 = (\frac{\delta}{|\textbf{v}|^2+1})^{1/2}$ and note that
for all $t\in \mathbb{R}$, $0<t< t_2$, we have $(|\textbf{v}|^2+1)t^2<\delta$.
Finally, if $t$ satisfies $0<t< \min\{t_0,t_1,t_2\}$ then
$$(\mathcal{\textbf{x}}+t\textbf{v},t)\in \Reali\left(\bigwedge_{i=1}^p
(|\X^i|^2\leq r_i^2)\wedge \barphialpha
\wedge U>0
\right)$$
and
$$|\mathcal{\textbf{x}}-(\mathcal{\textbf{x}}+t\textbf{v})|^2+t^2\ = \ (|\textbf{v}|^2)t^2 <
\delta,$$
and so we have shown that
$$\mathcal{\textbf{x}}\in \Reali\left(\bigwedge_{i=1}^p (|\X^i|^2\leq r_i^2)\wedge
\barphialpha
\wedge U>0
\right)_{\limit} .$$
\end{proof}

Using the same notation as in Proposition \ref{prop:level1} above:
\begin{corollary}
\label{cor:level1}
Let $\phi$ be a $\mathcal{P}$-formula, containing no negations
and no inequalities, with
$\mathcal{P} \subset \R[X_1,\ldots,X_k]$
with
$\mathcal{P}\in \mathcal{A}_{k,a}$.
Then, there exists a family
of polynomials $\mathcal{P}' \subset \R[X_1,\ldots,X_k,U]$,
and a
$\mathcal{P}'$-formula $\bar{\phi}$
satisfying (\ref{eqn:level1}), and such that
$\mathcal{P}'\in \mathcal{A}^{\mathrm{div-free}}_{k+1,(k+a)(a+2)}$.
\end{corollary}

\begin{proof}
The proof is immediate from
Lemma \ref{lem:equivalence}, Remark \ref{rem:zero}, and
the definition of $\bar{\phi}$.
\end{proof}

\begin{definition}
Let $\Phi$ be a $\mathcal{P}$-formula, $\mathcal{P}\subseteq \mathbb{R}[\X_1,\dots,\X_k]$, and say
that $\Phi$ is a \emph{$\mathcal{P}$-closed formula} if the formula $\Phi$ contains
no negations and all the inequalities in atoms of $\Phi$ are weak inequalities.
\end{definition}

Let $\mathcal{P} = \{F_1,\ldots,F_s\} \subset \mathbb{R}[X_1,\ldots,X_k]$, and
$\Phi$ a $\mathcal{P}$-closed formula.

For $R \in \mathbb{R}_+$, let $\Phi_R$ denote the formula $\Phi \wedge (|\X|^2 -R^2
\leq 0)$.

Let $\Phi^\dagger$ be the formula obtained from $\Phi$
by replacing each
occurrence of the atom $F_i\ast 0$, $\ast\in \{=,\leq,\geq\}$,
$i=1,\dots,s$, with
$$
\begin{aligned}F_i-V_i^2=0 & \mbox{ if } \ast \in \{\leq\}, \\
-F_i-V_i^2= 0 &  \mbox{ if } \ast \in\{\geq\}, \\
F_i=0 & \mbox{ if }\ast \in \{=\},
\end{aligned}
$$
and for
$R, R' \in \mathbb{R}_+$, let $\Phi^\dagger_{R,R'}$ denote the
formula
\[
\Phi^\dagger \wedge (U_1^2 + |\X|^2 -R^2 = 0) \wedge (U_2^2 +
|\V|^2 - R'^2 = 0).
\]

We have
\begin{proposition}{\label{prop:sectionfivemain2}}
\[
\Reali(\Phi) = \pi_{[1,k]}(\Reali(\Phi^\dagger)),
\]
and for all $0 < R \ll R'$,
\[
\Reali(\Phi_R) = \pi_{[1,k]}(\Reali(\Phi^\dagger_{R,R'})),
\]
\end{proposition}

\begin{proof}
Obvious.
\end{proof}

Note that, for $0< R \ll R'$, $\pi_{[1,k]}|_{
\Reali(\Phi^\dagger_{R,R'})}$ is a continuous,
semi-algebraic surjection onto
$\Reali(\Phi_R)$.
Let $\pi_{R,R'}$ denote the map
$\pi_{[1,k]}|_{\Reali(\Phi^\dagger_{R,R'})}$.

\begin{proposition}
\label{prop:reduction2}
We have that  $\mathcal{J}^p_{\pi_{R,R'}}
(\Reali(\Phi^\dagger_{R,R'}))$
is $p$-equivalent to $\pi_{[1,k]}(\Reali(\Phi^\dagger_{R,R'}))$.
Moreover, for any two formulas $\Phi,\Psi$, the realizations
$\Reali(\Phi)$ and $\Reali(\Psi)$ are homotopy equivalent if,
for all $1 \ll R \ll R'$,
$$\Reali(\mathcal{J}^p_{\pi_{R,R'}}(\Phi^\dagger_{R,R'}))\simeq
\Reali(\mathcal{J}^p_{\pi_{R,R'}}(\Psi^\dagger_{R,R'}))$$
are homotopy equivalent for some $p > k$.
\end{proposition}
\begin{proof}
Immediate from Proposition \ref{prop:5} and
Propositions \ref{prop:top_basic} and \ref{prop:sectionfivemain2}.
\end{proof}

Suppose that
$\Phi$ has additive format
bounded by $(a,k)$, and suppose that the number of polynomials
appearing $\Phi$ is $s$,
and without loss of generality
we can assume that $s \leq k+a$ (see
Remark \ref{rem:zero}).
Then the sum of the
additive complexities of the polynomials appearing in
$\Phi^{\dagger}_{R,R'}$ is bounded by
$3a+3s+2\leq 3a+3(a+k)+2\leq 6(k+a)$,
and the formula $\Phi^{\dagger}_{R,R'}$ has additive
format bounded by
$(6(k+a),2k+a+2)$.
\hide{
Using the notation of
Proposition \ref{prop:level1},
let $\mathcal{J}^{p}_{\pi_{R,R'}}(\Phi^\dagger_{R,R'})^\star$ denote
the formula,
with last variable $U$,
\begin{equation}
{\label{eqn:star}}
\Omega^R \wedge \overline{\left(\Theta_1 \wedge \Theta_2^{\Phi^\dagger_{R,R'}}\wedge \Theta_3^{\pi_{R,R'}}\right)}
\wedge U>0.
\end{equation}
}
Consequently, the additive format of the formula
$$\Theta_1\wedge \Theta_2^{\Phi^\dagger_{R,R'}} \wedge \Theta_3^{\pi_{R,R'}}$$ is bounded by $(M,N)$,
$$
\begin{aligned}
M &= (p+1)(6k+6a+1)+\textstyle{\binom{p+1}{2}}(4k+2a+3) \\
N &= (p+1)(2k+a+3)+\textstyle{\binom{p+1}{2}}.
\end{aligned}
$$
In the above, the estimates of Proposition \ref{prop:calDp} suffice, with $(a,k)$ replaced by $(6(k+a), 2k+a+2)$.  Now, applying Corollary \ref{cor:level1} we have that there exists a $\mathcal{P}'$-formula
$$\overline{\left(\Theta_1\wedge \Theta_2^{\Phi^\dagger_{R,R'}} \wedge \Theta_3^{\pi_{R,R'}}\right)}$$
which satisfies Equation \ref{eqn:level1} and such that the \emph{division-free} additive format of this formula is bounded by
$((N+M)(M+2) ,N+1 )$.  Finally, let $\mathcal{J}^{p}_{\pi_{R,R'}}(\Phi^\dagger_{R,R'})^\star$ denote
the formula,
with last variable $U$,
\begin{equation}
{\label{eqn:star}}
\Omega^R \wedge \overline{\left(\Theta_1 \wedge \Theta_2^{\Phi^\dagger_{R,R'}}\wedge \Theta_3^{\pi_{R,R'}}\right)}
\wedge U>0,
\end{equation}
and we have that the \emph{division-free} additive format of $\mathcal{J}^p_{\pi_{R,R'}}(\Phi^\dagger_{R,R'})^\star$ is bounded by $(
M',N+1)$,
$$M'=(p+1)(2k+a+3)+(N+M)(M+2).$$
Note that $M'\leq 5M^2$.

We have shown the following,

\begin{proposition}
\label{prop:boundonT}
Suppose that the sum of the  additive complexities of
$F_i, 1\leq i \leq s$, is bounded by $a$. Then, the semi-algebraic set
$\Reali(\mathcal{J}^{p}_{\pi_{R,R'}}(\Phi^{\dagger}_{R,R'})^\star)$
can be defined by a
$\mathcal{P}'$-formula with
$\mathcal{P}' \in
\mathcal{A}^{\mathrm{div-free}}_
{
5M^2,N+1}$,
$$\begin{aligned}
M&= (p+1)(6k+6a+1)+2\textstyle{\binom{p+1}{2}}(4k+2a+3)\\
N&=  (p+1)(2k+a+3)+\textstyle{\binom{p+1}{2}}.
\end{aligned}$$
\end{proposition}

Finally, we obtain

\begin{proposition}
\label{prop:closedandbounded}
The number of distinct homotopy types of
semi-algebraic subsets of
$\mathbb{R}^k$ defined by $\mathcal{P}$-closed formulas
with $\mathcal{P} \in \mathcal{A}_{a,k}$ is
bounded by
$2^{(k(k+a))^{O(1)}}$.
\end{proposition}
\begin{proof}
Let $\mathcal{P} \in \mathcal{A}_{a,k}$.
By the conical structure at infinity of semi-algebraic sets
(see, for instance \cite[page 188]{BPRbook2})
there exists
$R_{\mathcal{P}}>0$
such that, for all $R>R_\mathcal{P}$
and every $\mathcal{P}$-closed formula $\Phi$,
the
semi-algebraic
sets $\Reali(\Phi_R), \Reali(\Phi)$ are
semi-algebraically homeomorphic.

For each $a,k\in \mathbb{N}$, there are only finitely many semi-algebraic homeomorphism types
of semi-algebraic sets described by a $\mathcal{P}$-formula having additive complexity at most $(a,k)$ \cite[Theorem 3.5]{Dries}.
Let $\ell\in \mathbb{N}$, $\mathcal{P}_i\in \mathcal{A}_{a,k}$, and $\Phi_i$ a $\mathcal{P}_i$-formula, $1\leq i \leq \ell$, such that every semi-algebraic set described by a formula
of additive complexity at most $(a,k)$ is semi-algebraically homeomorphic to $\Reali(\Phi_i)$ for some $i$, $1\leq i \leq \ell$.
Let $R=\max_{1\leq i \leq \ell} \{R_{\mathcal{P}_i}\}$ and $R'\gg R$.

Let $\Phi\in \{\Phi_i\}_{1\leq i \leq \ell}$.
By Proposition \ref{prop:reduction2} it suffices to bound the number
of distinct
homotopy types of the semi-algebraic
set
$\Reali(\mathcal{J}^{k+1}_{\pi_{R,R'}}(\Phi^\dagger_{R,R'}))$.
By Proposition \ref{prop:level1}, we have that
$$\Reali\left(
\mathcal{J}^{k+1}_{\pi_{R,R'}}(\Phi^\dagger_{R,R'})^\star
 \right)_{\limit}
=\Reali(\mathcal{J}^{k+1}_{\pi_{R,R'}}(\Phi^\dagger_{R,R'})).$$
By Proposition \ref{prop:boundonT},
the division-free additive format of the formula
$\mathcal{J}^{k+1}_{\pi_{R,R'}}(\Phi^\dagger_{R,R'})^\star$
is bounded by
$(2M,N)$,
where $p=k+1$.
The proposition now follows immediately from
 Theorem \ref{thm:main_weak}.
\end{proof}

\begin{proof}[Proof of Theorem \ref{thm:main1}]
Using the construction of Gabrielov and Vorobjov
\cite{GV07} one can reduce
the case of arbitrary semi-algebraic sets to that of
a
closed and bounded one,
defined by a $\mathcal{P}$-closed formula,
without changing asymptotically the complexity estimates (see for example
\cite{BV06}).
The theorem then follows directly from Proposition \ref{prop:closedandbounded}
above.
\end{proof}

\subsection{Proof of Theorem \ref{thm:main}}
\begin{proof}[Proof of Theorem \ref{thm:main}]
The proof is identical to that of the proof of Theorem \ref{thm:main_weak},
except that we use
Theorem \ref{thm:main1} instead of
Theorem \ref{thm:additive}.
\end{proof}



\chapter{Semi-cylindrical decomposition}{\label{ch:semi}}

Decomposition theorems are ubiquitous in real algebraic geometry as well as in computational geometry, used to study the topology of semi-algebraic sets or used, for example, in point location algorithms.  In the latter case, given a subset $X$ of $\R^k$ in some well-behaved class of sets (e.g., $X$ is a polygon, definable, semi-algebraic, etc.), the goal is to partition $\R^k$ into sets, typically called \textit{cells} of the decomposition, such that the union of cells in some sub-collection of this partition is $X$.  


For a concrete example, we turn to the well-known cylindrical decomposition for semi-algebraic sets.  
This decomposition theorem provides a partition of $\R^k$ into finitely many cells so that the union of some sub-collection of the cells is $X$.  
The cells in this case are semi-algebraic and have \textit{bounded description complexity} in that they are all semi-algebraic and can be described using a finite list of polynomials, whose number and degree are bounded by a function of the number and maximum degree of the polynomials appearing in a description of $X$.  
Partitions of this type are sometimes called \textit{Tarski decompositions} in the literature.  
A typical quantitative question one can ask is, given $X_1,\ldots,X_n$ subsets of $\R^k$, can one prove existence of decomposition of $\R^k$ into sets of bounded description complexity such that the number of the sets in the decomposition is as few as possible, bounded by a function of $n,k$ and the description complexity of the sets $X_1,\ldots,X_n$.  

Another example of a Tarski decomposition, a finite partition of $\R^k$ into sets of bounded description complexity, coming from computational geometry is that of $\ep$-cuttings \cite{matousek1995approximations}.  A \textit{cutting} is a collection of (possibly unbounded) simplices with disjoint interiors which cover $\R^k$. Given a set $\mathcal{H}$ of $n$ hyperplanes in $\R^k$, an $\ep$-cutting is a cutting $\Delta$ such that for every simplex $s\in \Delta$ the numbef of hyperplanes in $\mathcal{H}$ which intersect the interior of $s$ is at most $\ep n$.  It was proved by Chazelle and Friedman in \cite{chazelle1990deterministic} that for any $n$ hyperplanes and any $r\leq n$ an $(1/r)$-cutting exists of size $O(r^k)$, and that there exists an algorithm for computing this cutting.  The technique of $\ep$-cuttings have been used by many authors to make progress on several problems in computational geometry 
as the \textit{divide step} in several divide-and-conquer arguments where a problem (e.g., point location, incidence problems) is broken up into several easier to handle problems.  


  

It is often the case that one is primarily concerned with the asymptotic behaviour of the 
number of sets in the partition 
as the number of sets 
to be partitioned 
increases, while 
treating 
 the description complexity of 
$X$ 
the sets 
and the dimension of the ambient space $k$ as constants.  
In these settings, it is the \textit{combinatorial part} of the bounds that are relevant, as opposed to the \textit{algebraic part} of the bounds.  Loosely speaking, the combinatorial part of a bound is the part that depends on $n$ the number of sets and the dimension $k$ of the ambient space, while the algebraic part is the part that depends on the description complexity of the sets involved.  

When one is primarily concerned with the combinatorial part of the complexity of a decomposition, then the setting of o-minimal geometry and definable sets is an appropriate framework for analyzing the problem.  In this setting, the role of bounded description complexity is replaced by having all the sets involved arise as fibers of a single definable set in a larger space, which parametrizes the family from which the sets involved are taken. In this way the combinatorial part of the bound is brought into focus while the algebraic part is subsumed as part of the dimension of the ambient space.  When one reads the statement of the following theorems it may help to keep this in mind.  In order to state the theorems we define the structure of a cylindrical decomposition.  
  




A cylindrical definable decomposition, or cdd, of $\R^k$ is a partition of $\R^k$ into finitely many definable sets of a particular type, namely that these sets are defined inductively as follows.

\begin{itemize}

\item A cdd of $\R$ consists of a finite set of points $a_1<\ldots<a_m$ of $\R$ and the set of maximal open intervals $(-\infty,a_1),(a_1,a_2),\ldots,(a_{m-1},a_m),(a_m,\infty)$ in the complement of those points in $\R$.  

\item A cdd of $\R^k$ with $k>1$ consists of a cdd, $\{C_\alpha\}_{\alpha\in I}$, of $\R^{k-1}$ and for each $\alpha \in I$ a cdd $\{A_\beta\}_{\beta\in J}$ of $\R$ and definable homeomorphisms $\phi_{\alpha,\beta}: C_\alpha \times A_\beta \to Y_{\alpha,\beta}$ satisfying  $\pi^{k}_{\leq k-1}\circ \phi_{\alpha,\beta}=\pi^{k}_{\leq k-1}$ and which satisfy $\{Y_{\alpha,\beta}\}_{\alpha,\beta}$ is a partition of $\R^k$. The cdd of $\R^k$ is then given by $\{Y_{\alpha,\beta}\}_{\alpha\in I,\beta \in J}$. 

\end{itemize}

\noindent We make the following inductive definition, that of a \textit{cylindrical set}.

\begin{itemize}

\item A cylindrical set of $\R$ is a point $a\in \R$ or the open interval $(a,b)$ between two points $a,b\in \R$. 
\item Given a cylindrical set $X$ of $\R^{k-1}$ and definable continuous functions $\xi,\xi_1,\xi_2:Y\to \R$ satisfying $\xi_1(x)<\xi_2(x)$ for all $x\in X$, a cylindrical set $Y$ of $\R^k$ is either a \textit{graph} of the form $Y=\{(x,a)\in \R^k|\; \xi(x)=a\}$ or a \textit{band} of the form $\{(x,a)\in \R^k|\; \xi_1(x)<a<\xi_2(x)\}$.  

\end{itemize}

A partition of $\R^k$ into definable sets is called a \textit{semi-cylindrical decomposition} if each cell of the decomposition is a cylindrical set.  As we will see, a semi-cylindrical decomposition need not have a \textit{cylindrical structure}, in that given a cell $Y$ of a semi-cylindrical decomposition of $\R^k$, it need not be the case that the cylinder $\pi^k_{\leq k-1}(Y)\times R$ is a union of cells of the decomposition.  

A subset $A$ of $\R^k$ is said to be \textit{compatible} with a set $X$ if $A\subseteq X$ or $A\cap X=\emptyset$, and   
if $\mathcal{A}$ is a partition of $\R^k$ and $X_1,\ldots,X_n$ are subsets of $\R^k$ such that $A$ is compatible with $X_1,\ldots,X_n$ for every $A\in \mathcal{A}$, then $\mathcal{A}$ is said to be \textit{adapted} to $X_1,\ldots,X_n$.

Let $R$ be a real closed field.

The following theorem is proved in \cite{Basu9}. 

\begin{theorem}[Quantitative cylindrical definable decomposition \textnormal{\cite{Basu9}}]{\label{thm:cdd}}
Let $T\subseteq \R^{k+\ell}$ be a definable set, then there exists constants $C_1,C_2$ and a finite family of definable sets $\mathcal{A}$, $|\mathcal{A}|\leq C_1$, and $$A \subseteq \R^{k}\times \R^{2(2^k-1)\ell }, \; A\in \mathcal{A},$$ which satisfy the following property.  For any $n$ points $y_1,\ldots,y_n\in \R^{\ell}$, some sub-collection of the sets $$\pi^{k+2(2^k-1)\ell}_{\leq k}\left({(\pi^{k+2(2^k-1)\ell}_{>k})}^{-1}(y_{i_1},\ldots,y_{i_{2(2^k-1)}})\quad \cap \quad A \right),$$ $A\in \mathcal{A}$ and $1\leq i_1,\ldots,i_{2(2^k-1)}\leq n$, form a cylindrical definable decomposition adapted to $T_{y_1},\ldots,T_{y_n}$.  Also, the number of sets in the decomposition is bounded by 
$C_2 n^{2^k-1}$. 
\end{theorem}

We prove a theorem which is weaker in that it does not ensure a cylindrical structure to the decomposition, yet the number of sets required is asymptotically better and still the cells of the decomposition are of constant description complexity. 


\begin{theorem}[Quantitative semi-cylindrical definable decomposition]{\label{thm:scdd}}
Let $T\subseteq \R^{k+\ell}$ be a definable set, then there exists constants $C_1,C_2$ and a finite family of definable sets $\mathcal{A}$, $|\mathcal{A}|\leq C_1$, and $$A \subseteq \R^{k}\times \R^{2(2^k-1)\ell }, \; A\in \mathcal{A},$$ which satisfy the following property.  For any $n$ points $y_1,\ldots,y_n\in \R^{\ell}$, some sub-collection of the sets $$\pi^{k+2(2^k-1)\ell}_{\leq k}\left({(\pi^{k+2(2^k-1)\ell}_{>k})}^{-1}(y_{i_1},\ldots,y_{i_{2(2^k-1)}})\quad \cap \quad A \right),$$ $A\in \mathcal{A}$ and $1\leq i_1,\ldots,i_{2(2^k-1)}\leq n$, form a 
semi-cylindrical decomposition 
of $\R^k$ adapted to $T_{y_1},\ldots,T_{y_n}$.  Also, the number of sets in the decomposition is bounded by $C_2 n^{2k-1}$.
\end{theorem}

\section{Some preliminaries}{\label{subsec:prelim}}

We first make some preliminary definitions.

\subsection{o-minimal structures}

A \textit{structure} $\mathcal{S}(\R)$ on 
$\R$ is
a collection families of subsets of $\R^\ell$, $\{\mathcal{S}_\ell\}_{\ell\geq 1}$,
satisfying: 
  
\begin{enumerate}

\item[$(i)$] all algebraic subsets of $\R^\ell$ are in $\mathcal{S}_\ell$,
\item[$(ii)$] each family $\mathcal{S}_\ell$ is a boolean algebra, \textit{i.e.}, closed under finite unions, intersections and taking complements,
\item[$(iii)$] for $A\in \mathcal{S}_m$ and $B\in \mathcal{S}_n$, we have $A\times B\in
  \mathcal{S}_{m+n}$.  
\item[$(iv)$] if $\pi^{n+1}_{ \leq n}:R^{n+1} \to \R^n$ is the projection map which
  forgets the last coordinate and $A\in S_{n+1}$, then $\pi^{n+1}_{\leq 
    n}(A)\in S_n$.
\end{enumerate}
If additionally
  \begin{enumerate}
\item[$(v)$] the elements of $\mathcal{S}_1$ are precisely finite unions of points
  and open intervals, 
\end{enumerate}
then the structure is said to be \textit{o-minimal}.
For any structure $\mathcal{S}(\R)$, a set $A\subseteq \R^\ell$ is said to be \textit{definable} in the
structure $\mathcal{S}(\R)$ if $A\in \mathcal{S}_\ell$.

\begin{remark}
It is well known that the collection of semi-algebraic subsets over $\R$, sets defined by a boolean combination of real polynomial equalities and inequalities, is an o-minimal structure.  The non-trivial but well known fact that the collection of semi-algebraic subsets is closed under taking projections follows from the Tarski-Seidenberg quantifier elimination Theorem \cite{BPRbook2}.

\end{remark}

\subsection{Endpoints and adapted partitions}{\label{sec:end}}

Two definitions that will be important in our exposition are that of the endpoints of a definable subset of $\R$ and the common refinement of a definable subset of $\R$. Endpoints of a definable subset $A$ of $\R$ are those elements in $\R$ which are on the boundary of the definable set.  Their number is finite (see Lemma \ref{lem:finite}), and the set of endpoints of $A$ and the set of maximal open intervals in their complement comprise the coarsest partition of $\R$ which is a cdd and which is adapted to $A$, which we refer to as the common refinement of $A$.  More precisely,   
\begin{definition}{\label{def:endpoints}}
The set of \textnormal{endpoints} of $A$ is 
the set $\textrm{clos}{A}\smallsetminus \textrm{int}{A}$, where the closure $\textrm{clos}{A}$ of $A$ and interior $\textrm{int}{A}$ of $A$ are taken in the order topology  of $\R$ (or the Euclidean topology if the language contains $+,\times$).  
\end{definition}

\begin{lemma}{\label{lem:end1}}
Let $A_1,\ldots,A_n\subseteq \R$ be definable.  Then the set of endpoints of $\cup_{1\leq i \leq n} A_i$ is contained in  the union of the sets of endpoints of $A_i$, $1\leq i \leq n$.  In symbols,
\begin{equation}{\label{eqn:end1}}\displaystyle\textrm{clos}\left(\cup_{1\leq i \leq n} A_i\right)\smallsetminus \textrm{int}\left(\cup_{1\leq i \leq n} A_i \right) \subseteq  \bigcup_{1\leq i \leq n} \textrm{clos}(A_i)\smallsetminus \textrm{int}(A_i).\end{equation}
\end{lemma}
\begin{proof}
Equation \ref{eqn:end1} can be shown to hold for arbitrary sets $A_i$, $1\leq i \leq n$, in any topological space $X$ using elementary point set topology (e.g., Munkres).  The lemma then follows from the definitions. 

\end{proof}

\begin{lemma}{\label{lem:finite}}
Given a definable subset $A\subseteq \R$, the set of endpoints of $A$ is a finite definable subset of $\R$.

\end{lemma}
\begin{proof}
The set of endpoints of a point $\{a\}$ is the set $\{a\}$, and the set of endpoints of 
an interval $(a,b)$ is the set $\{a,b\}$.  By Lemma \ref{lem:end1} and the fact that $A$ is a finite union of points and intervals, the set of endpoints of $A$ is finite, hence definable by definition.

\end{proof}



The common refinement of a definable subset $A$ of $\R$ (see Definition \ref{def:common} below) is the coarsest cdd of $\R$ adapted to $A$.  
We also provide definitions of other adapted partitions, beginning with the basic boolean algebra generated by an arrangement, which is the coarsest partition of a set $X$ adapted to subsets $A_1,\ldots,A_n$, and whose   definition already appeared in \cite{Basu9}. 

\begin{notation}\textnormal{\cite{Basu9}}{\label{notaBOOLEAN}}
Let $\mathcal{A}=\{A_1,\dots,A_n\}$ be a collection of
subsets of a set $X$.  Let $I\subseteq \{1,\dots,n\}$, allowing
$I=\emptyset$ with the convention that if $I=\emptyset$ then 
the intersection over an empty index
set is $X$, $\cap_{i\in \emptyset} A_i= X$.  Let $\mathcal{A}(I)$ denote the
set $$\left(\bigcap_{i\in I} A_i\right) \setminus \left(\bigcup_{i\notin
  I} A_i\right),$$ and we call each such set a basic $\mathcal{A}$-set.

\end{notation}
\begin{definition}
[Boolean algebra generated by an arrangement \textnormal{\cite{Basu9}}] \label{defnBOOLEAN} 

\; The \; \textnormal{basic
    boolean algebra} of $X$ generated by $\mathcal{A}=\{A_1,\dots,A_n\}$,
  denoted $\mathcal{B}^\ast(A_1,\dots,A_n)$ or just $\mathcal{B}^\ast(\mathcal{A})$, is
  the collection of basic
  $\mathcal{A}$-sets $$\mathcal{B}^\ast(A_1,\dots,A_n)=\{\mathcal{A}(I)| \ I\subset
  \{1,\dots,n\}\}.$$ We use the notation $\mathcal{B}(A_1,\dots,A_n)$ to
  denote the set of connected components of sets in $\mathcal{B}^\ast(A_1,
  \dots,A_n)$ and call the set of connected components the
  \textnormal{boolean algebra} of $X$ generated by $A_1,\dots,
  A_n$.  \end{definition}

We will use the next definition in the proof of the main theorem in an essential way.  We need to have a slightly finer partition than the boolean algebra generated by an arrangement in order to make cylindrical cells; this is related to the separate roles of graphs and bands in the more familiar cylindrical decomposition theorem for semi-algebraic sets. 

\begin{definition}[Common refinement]{\label{def:common}}
Let $A_1,\ldots, A_n \subseteq \R$ be definable sets and define the \textnormal{common refinement}
 of the sets $A,B$ as the union of the sets of endpoints of $A_i$, $1\leq i\leq n$, and the set of maximal open intervals in the complement in $\R$ of these endpoints.  
\end{definition}

\begin{definition}
[Adapted partition] \label{defnADAP}For a partition of a set $X=\sqcup_{\alpha\in
    I} C_\alpha$ into pairwise disjoint sets $C_\alpha$, 
    we say that the partition $\{C_\alpha\}$ is
  \textnormal{adapted} to the sets $S_1,S_2,\dots,S_n$ with $S_i \subseteq
  X$ if for each $i=1,\dots,n$ we have $$\forall \alpha\in I \quad
  C_\alpha \cap S_i = \emptyset \quad
   \text{or} \quad C_\alpha \subseteq
  S_i. $$
\end{definition}

Given a collection of subsets $\mathcal{A}$ of a set $X$, the collection of sets $\mathcal{B}^\ast(A_1,\ldots,A_n)$ is the coarsest partition of $X$ which is adapted to the sets in $\mathcal{A}$.  Additionally, the sets in the common refinement of definable sets $A,B\subseteq \R$ comprise the coarsest partition of $\R$ adapted to $A$ and $B$ into sets which are either points or open intervals.  For example, $\mathcal{B}([0,1))=\{(-\infty,0),[0,1),[1,\infty)\}$ and the common refinement of $[0,1)$ is the partition $\{(-\infty,0),\{0\},(0,1),\{1\},(1,\infty)\}$.  


\subsection{Hardt's Theorem} 

We next state the o-minimal version of Hardt's Theorem, which provides the main technical tool for the proof of the main result of this section.  The version stated here is taken from \cite{Basu9}. 

\begin{definition}[Definably trivial maps]
Let $X\subseteq \R^{n+m}$ and $C\subseteq \R^m$ be definable, and let
$\pi:X \to \R^k$ denote the projection map which forgets the last $m$ coordinates.
We
say that \textnormal{$X$ is definably trivial over $C$} if there exists a definable set $A\subseteq \R^{n}$
and map $h$, called the trivialization of $X$ over $C$, such that
\begin{enumerate}
\item[$(1)$] $h$ is a definable homeomorphism $h:A\times C \to X\cap
\pi^{-1}(C)$, and
\item[$(2)$] $\pi \circ h = \pi$.

\end{enumerate}
If $Y$ is a definable subset of $X$, we say that the trivialization of $X$ over $C$ is compatible with $Y$ if there exists a definable subset $B$ of $A$ satisfying $h(B\times C)=Y\cap \pi^{-1}(C)$.

\end{definition}

\begin{theorem}[Hardt's theorem for definable families
\textnormal{\cite{Basu9}}]\label{thm1}
Let $X\subseteq \R^{k+\ell}$ be a definable set and let $Y_1,\dots,Y_m$
be definable subsets of $X$.  Then, there exists a finite partition of
$R^\ell$ into definable sets $C_1,\dots,C_N$ such that $X$ is
definably trivial over each $C_i$, and moreover the trivialization
over each $C_i$ is compatible with $Y_1,\dots,Y_m$. \end{theorem}

The following two corollaries follow immediately from Hardt's theorem and the proofs are omitted. 

\begin{corollary}{\label{cor2}}
Let $X\subseteq \R^{k+\ell}$ be a definable set.  There exists $N\geq 0$ and a finite family $\{Z_i\}_{i=1}^N$ of definable subsets of $\R^{k+\ell}$ satisfying: for each $y\in \R^\ell$ there exists an $i$, $1\leq i \leq N$, such that $Z_i$ is definably homeomorphic to $X\cap {(\pi^{k+\ell}_{>k})}^{-1} (y)$.

\end{corollary}

\begin{notation}[Definable fiber \textnormal{\cite{Basu9}}]
Let $T\subseteq \R^{k+\ell}$
be a definable set, and let $\pi^{k+\ell}_{\leq  k}:\R^{k+\ell}\to \R^k$ and
$\pi^{k+\ell}_{>k}:\R^{k+d}\to \R^\ell$ be the projection maps onto the
first $k$ and last $\ell$ components of $\R^{k+\ell}$ respectively. For
$y\in \R^\ell$ we set $T_y=\pi^{k+\ell}_{\leq  k}((\pi^{k+\ell}_{>k})^{-1}
(y)\cap T)$ and say that $T_y$ is a \textnormal{definable fiber of
$T$}.\end{notation}

\begin{corollary}{\label{cor:N}}
Given a definable set $T\subseteq \R^{m}$, there exists $N=N(T)\geq 0$ such that for any
$w\in \R^{m-1}$ the number of connected components of $T_w\subseteq \R$ is at most $N/5$.  
In particular, the cardinality of the common refinement of $T_w$ is at most $N$ for all $w\in \R^{m-1}$, 
and 
the cardinality of the common refinement of $T_{w_1},\ldots,T_{w_n}$ is bounded by $n\cdot N$, 
for any $w_1,\ldots,w_n\in \R^{m-1}$.
\end{corollary}


\section{Proof of the main result}

In this section we prove the main result of the chapter, the quantitative semi-cylindrical definable decomposition.  We also provide the proof of the quantitative cylindrical definable decomposition theorem given in \cite{Basu9} for the purpose of introducing notation and many of the same ideas which will be used in the proof of the main result.  Since many of the same ideas and notation will be used in the proof of the main result while the more technical ideas are not present in the cylindrical theorem, we provide the proof given in \cite{Basu9} for the convenience of the reader as a comparison and introduction to the proof of the main result.

Before giving the proof of the quantitative cylindrical definable decomposition, we need two lemmas.  The second lemma is very similar to Lemma 3.11 given in \cite{Basu9}.  
This lemma corresponds to the first projection phase in the more familiar cylindrical algebraic decomposition algorithm for semi-algebraic sets (see, for example \cite{BPRbook2}).

\begin{lemma}\label{lem:combine}
Let $X_1,\ldots,X_n\subset \R^k$ and $\mathcal{N}\subseteq \{1,\ldots,n\}$.  Suppose that, for each $i,j\in \mathcal{N}$, 
there exists a partition $\{B^{i,j}_l\}_{1\leq l\leq M}$ of $\R^{k-1}$ 
such that $X_i\cup X_j$ is definably trivial over $B^{i,j}_l$, $1\leq l\leq M$, and such that the trivialization is compatible with $X_i,X_j$.  Then, 
$\R^k$ is definably trivial over $A$ for every $A$ in the boolean algebra $\mathcal{B}(B^{i,j}_l|
i,j\in \mathcal{N}, 1\leq l\leq M)$, and the trivialization is compatible with 
$X_i$, $i\in \mathcal{N}$. 
\end{lemma}

\begin{proof}
Let $N_i$ satisfy the conclusion of Corollary \ref{cor:N} so that the number of sets in the common refinement of $(X_i)_{x'}$ is at most $N$ for any $x'\in \R^{k-1}$, $1\leq i\leq n$, and let $N=\max_{1\leq i \leq n} N_i$.   
\hide{
Let $A$ be a set in the boolean algebra $\mathcal{B}(B^{i,j}_l|1\leq i,j\leq n, 1\leq l \leq M)$. Since $B^{i,j}_l\cap B^{i,j}_{l'}=\emptyset$ if $l\neq l'$ for each pair $(i,j)$ by assumption, the set $A\in \mathcal{B}(B^{i,j}_l|1\leq l\leq M, 1\leq i,j\leq n)$ is a connected component of a set of the form  
$$\left(\bigcap_{(i,j,l)\in I} B^{i,j}_l\right) $$
for a subset $I$ of $\{1,\ldots,n\}^2\times \{1,\ldots,M\}$ with the property the elements of $I$ are exactly tuples of the form $(i,j,f(i,j))$, $1\leq i,j\leq n$, for some well-defined map $f:\{1,\ldots,n\}^2 \to \{1,\ldots,M\}$.  
}

Let $\mathcal{N}\subseteq \{1,\ldots,n\}$ and $A$ be a set in the boolean algebra $\mathcal{B}(B^{i,j}_l|
i,j\in \mathcal{N}, 1\leq l \leq M)$. Since $B^{i,j}_l\cap B^{i,j}_{l'}=\emptyset$ if $l\neq l'$ for each pair $(i,j)$ by the assumption that $\{B^{i,j}_l\}_{1\leq l\leq M}$ is a partition, the set $A\in \mathcal{B}(B^{i,j}_l|1\leq l\leq M, 1\leq i,j\leq n)$ is a connected component of a set of the form  
$
\bigcap_{(i,j,l)\in I} B^{i,j}_l 
$ 
for a subset $I$ of $
\mathcal{N}^2 \times \{1,\ldots,M\}$ with the property the elements of $I$ are exactly tuples of the form $(i,j,f(i,j))$, 
$i,j\in \mathcal{N}$, for some well-defined map $f:
\mathcal{N}^2 \to \{1,\ldots,M\}$.  
In particular, $X_i \cup X_j$ is definably trivial over $A$, since $A$ is a subset of $\cap_{(i,j,l)\in I} B^{i,j}_l$ and $X_i\cup X_j$ is definably trivial over each $B^{i,j}_l$.  

Since $X_i\cup X_j$ is definably trivial over $B^{i,j}_l$, for each pair $(i,j)$, 
$i,j\in \mathcal{N}$, 
there exists a definable subset $C_{i,j}$ of $\R$ such that the following diagram commutes.  
$$\xymatrix{
C_{i,j}\times B^{i,j}_{f(i,j)} \ar[r]^-{\phi^{i,j}}_-\approx \ar[d]^\pi & \pi^{-1}(B^{i,j}_{f(i,j)})\cap \left(X_i \cup X_j\right) \ar[ld]^\pi \\
B^{i,j}_{f(i,j)} & }
$$
Furthermore, since this trivialization is compatible with $X_i,X_j$, there exists definable subsets $D_{i,j},E_{i,j}$ of $C_{i,j}$ such that the following diagrams commute.
$$\xymatrix{
D_{i,j}\times B^{i,j}_{f(i,j)} \ar[r]^-{{\phi^{i,j}}}_-{\approx} \ar[d]^\pi & \pi^{-1}(B^{i,j}_{f(i,j)})\cap X_i \ar[ld]^\pi \\
B^{i,j}_{f(i,j)} & }
$$
$$\xymatrix{
E_{i,j}\times B^{i,j}_{f(i,j)} \ar[r]^{\phi^{i,j}}_-\approx \ar[d]^\pi & \pi^{-1}(B^{i,j}_{f(i,j)})\cap X_j \ar[ld]^\pi \\
B^{i,j}_{f(i,j)} & }
$$
The definable subsets $(X_i)_{x},(X_j)_{x}$ of $\R$, $x\in \R^{k-1}$, induce a partition
\begin{equation*}\label{eq:V}
V_1^{i,j}(x),\ldots,V_{2N}^{i,j}(x)
\end{equation*}
of $\R$ into the common refinement of $(X_i)_{x},(X_j)_{x}$, where the set $V_{2m+1}^{i,j}(x)$ is a maximal open interval contained in one of 
$$
\begin{aligned}
(X_i)_{x}\cap (X_j)_{x}&,(X_i)_{x}\smallsetminus (X_j)_{x} \\ 
(X_j)_{x}\smallsetminus (X_i)_{x} &, \R\smallsetminus \left((X_i)_{x}\cup (X_j)_{x}\right),
\end{aligned}
$$ 
and $V_{2m}^{i,j}(x)$ is the left endpoint of the interval $V_{2m+1}^{i,j}(x)$, $m\geq 0$ and $x\in \R^{k-1}$, with the understanding that $V_h^{i,j}(x)$ can be empty for all $h\geq h_0$, for some $0\leq h_0\leq 2N$. 
Clearly, the sets $V_h^{i,j}(x)$ are definable and depend only on $X_i,X_j$ and on $h$. 

\hide{
Let $\mathcal{N}\subseteq \{1,\ldots,n\}$ and $A$ be a set in the boolean algebra $\mathcal{B}(B^{i,j}_l|
i,j\in \mathcal{N}, 1\leq l \leq M)$. Since $B^{i,j}_l\cap B^{i,j}_{l'}=\emptyset$ if $l\neq l'$ for each pair $(i,j)$ by the assumption that $\{B^{i,j}_l\}_{1\leq l\leq M}$ is a partition, the set $A\in \mathcal{B}(B^{i,j}_l|1\leq l\leq M, 1\leq i,j\leq n)$ is a connected component of a set of the form  
$$\left(\bigcap_{(i,j,l)\in I} B^{i,j}_l\right) $$
for a subset $I$ of $
\mathcal{N}^2 \times \{1,\ldots,M\}$ with the property the elements of $I$ are exactly tuples of the form $(i,j,f(i,j))$, 
$i,j\in \mathcal{N}$, for some well-defined map $f:
\mathcal{N}^2 \to \{1,\ldots,M\}$.  
}
So we have seen that $X_i\cup X_j$ is definably trivial over $A$ for every pair $(i,j)\in \mathcal{N}^2$, and we claim that 
\begin{equation}{\label{eq:VV}}
\phi^{i,j}(C\times B^{i,j}_{f(i,j)})=\{(a,x)| a\in V_h^{i,j}(x), x\in A\} 
\end{equation}
 for some $C$ in the common refinement of $D_{i,j},E_{i,j}$, which suffices to prove the lemma. 

Fix $h,i,j$, $1\leq h\leq M$ and 
$i,j\in \mathcal{N}$.  Since the set $X_i\cup X_j$ is definably trivial over $A$, by definition the set $V_{h}^{i,j}(x)$ is either an interval or a point for all $x\in A$.  If $V_h^{i,j}(x)$ is an interval for $x\in A$,  let $x_0\in A$. 
It is clear that $(\phi^{i,j})^{-1}(V_{h-1}^{i,j}(x_0)),(\phi^{i,j})^{-1}(V_{h+1}^{i,j}(x_0))$ are endpoints of of some interval $C$ in the common refinement of $D_{i,j},E_{i,j}$, and that $V_h^{i,j}(x_0)=\phi^{i,j}(C\times \{x_0\})$.  
If $V_h^{i,j}(x)$ is a point for all $x\in A$ and $x_0\in A$, then $(\phi^{i,j})^{-1}(V_h^{i,j}(x_0))$ is a point $C$ in the common refinement of $D_{i,j},E_{i,j}$, and $V_h^{i,j}(x_0)=\phi^{i,j}(C\times \{x_0\})$.  

\end{proof}


\begin{lemma}[Cylindrical extension] \label{lem:extension}
Let $T\subseteq \R^{k+\ell}$ be a definable set, $k>1$.  Then there exists a definable partition $A_1,\ldots,A_M$ of $\R^{k-1+2\ell}$ satisfying the following.  For each $l$, $1\leq l \leq M$, and $y_1,y_2\in \R^\ell$, let $$B_l(y_1,y_2)=\pi^{k-1+2\ell}_{\leq k-1}\left({(\pi^{k-1+2\ell}_{>k-1})}^{-1}(y_1,y_2)\cap A_l\right).$$  The projection $\pi^k_{>1}$ restricted to the sets $T_{y_1}\cup T_{y_2}$ is definably trivial over $B_l(y_1,y_2)$, $1\leq l \leq M$, and the trivialization is compatible with $T_{y_1}$ and $T_{y_2}$.  
In particular, for any $y_1,\ldots,y_n\in \R^\ell$, any subset $\mathcal{N}\subseteq \{1,\ldots,n\}$, and any $A$ in the boolean algebra $\mathcal{B}:=\mathcal{B}(B_l(y_i,y_j)|
i,j\in \mathcal{N}, 1\leq l\leq M)$ 
there exists a cdd $\mathcal{A}_A$ of $\R$ and definable homeomorphisms $\phi_{C,A}:C\times A \to Y_{C,A}$ where $\{Y_{C,A}\}_{C\in \mathcal{A}_A, A\in \mathcal{B}}$ is a definable partition of $\R^k$ adapted to 
$T_{y_i}$, $i\in \mathcal{N}$.  
\end{lemma}

\begin{proof} 

We follow the proof of Lemma 3.11 in \cite{Basu9}.  

Let $T\subseteq \R^{k+\ell}$ be a definable set and define
$$\begin{aligned}
V_1&=\{(x,y_1,y_2)\in \R^k \times \R^\ell \times \R^\ell| \; (x,y_1)\in T\}\\
V_2&=\{(x,y_1,y_2)\in \R^k \times \R^\ell \times \R^\ell| \; (x,y_2)\in T\}.
\end{aligned}$$
Applying the Hardt Triviality Theorem to $(\R^{k+2\ell},V_1,V_2)$ and the projection map $\pi:=\pi^{k+2\ell}_{>1}$ we obtain a finite definable partition $A_1,\ldots,A_M$ of $\R^{k-1+2\ell}$ such that for each $A_i$, $1\leq i \leq M$, there exists a definable homeomorphism $h:\R\times A_i \to \pi^{-1}(A_i)$ such that the following diagram commutes,
$$
\xymatrix{
\R\times A_i \ar[r]^h \ar[d]^\pi & \pi^{-1}(A_i) \ar[ld]^\pi \\
A_i & }
$$
and such that the trivialization $h$ is compatible with $V_1,V_2$ (note that this last assumption means that $h$ is not typically the identity map, but we can take $h$ to be the identity map whenever $\pi^{-1}(A_i)\cap (V_1\cup V_2)=\emptyset$).  In particular, for each $i$, $1\leq i\leq M$, there exists subsets $B_1,B_2$ of $\R$ such that the following diagram commutes for each $j$, $1\leq j \leq 2$.
$$
\xymatrix{
B_j\times A_i \ar[r]^-{h|_{B_j\times A_i}} \ar[d]^\pi & \pi^{-1}(A_i) \cap V_j \ar[ld]^\pi \\
A_i & }
$$
Let $\mathcal{A}_i$ be the common refinement of $B_1,B_2\subseteq \R$ and $C\in \mathcal{A}_i$, and let $X_{C,i}=h(C\times A_i)$ \textsl{(we may want to say that if $C$ is a point then $X_{C,i}$ is a \textit{graph over $A_i$} and if $C$ is an interval then $X_{C,i}$ is a \textit{band over $A_i$} here)}.  Then, the collection $\{X_{C,i}\}_{C\in \mathcal{A}_i,1\leq i \leq M}$ is a finite definable partition of $\R^{k+2\ell}$ which is adapted to $V_1,V_2$.

Now let $y_1,y_2\in \R^\ell$.  Note that $V_1=T\times \R^\ell$ and similarly for $V_2$ after rearrangement of the last $2\ell$ components.  In particular, $T_{y_1}\cup T_{y_2}$ is definably trivial over $B_i(y_1,y_2)$ and this trivialization is compatible with $T_{y_1},T_{y_2}$ using the trivialization $h$.  More specifically, when one fixes the last $2\ell$ components to be $(y_1,y_2)$ in $\R^{k+2\ell}$, then the definable fibers $T_{y_1}=(V_1)_{(y_1,y_2)},T_{y_2}=(V_2)_{(y_1,y_2)}$ are definably trivial over the sets $B_i(y_1,y_2)=(A_i)_{(y_1,y_2)}$, $1\leq i\leq M$, where the trivialization is given by restricting the homeomorphism $h$ to $\R^k$. This proves the first part of the lemma. 
The second part of the lemma now follows immediately from Lemma \ref{lem:combine}. 

\end{proof}

\begin{proof}[Proof of Theorem \ref{thm:cdd}]
The proof is by induction.  We first prove the theorem in the case $k=1$.  Let $T\subseteq \R^{1+\ell}$ be a definable set and $N\geq 0$ as in Corollary \ref{cor:N}.  We follow the proof of Lemma \ref{lem:extension}.  
Define $$\begin{aligned}
V_1&=\{(x,y_1,y_2)\in \R^{1+2\ell}|\; x\in T_{y_1}\}, \\
V_2&=\{(x,y_1,y_2)\in \R^{1+2\ell}|\; x\in T_{y_2}\}.
\end{aligned}
$$
Applying Hardt triviality theorem to $\R^{1+2\ell}$ and the map $\pi=\pi^{1+2\ell}_{>1}$ 
we obtain a partition of $\R^{2\ell}$ into definable sets $A_1,\ldots,A_M$  and trivializations $h:\R\times A_i \to \pi^{-1} (A_i)$ such that these trivializations are compatible with $V_1,V_2$, $1\leq i \leq M$. In particular, for each $i$, $1\leq i \leq M$, 
there exist definable sets $B_1,B_2$ such that $h|_{B_j\times A_i}:B_j\times A_i \to \pi^{-1}(A_i)\cap V_j$ are definable homeomorphisms satisfying $\pi \circ h|_{B_j\times A_i} = \pi$, $1\leq j\leq 2$.  
Let $\mathcal{A}_i$ be the common refinement of $B_1,B_2$, and we denote for $C\in \mathcal{A}_i$ the image $X_{C,i}:=h(C\times A_i)$, and call $X_{C,i}$ a \textit{graph} if $C$ is a point and a \textit{band} if $C$ is an interval.  Let $\mathcal{A}=\{X_{C,i}|\; C\in \mathcal{A}_i, 1\leq i \leq M\}$.  Then $|\mathcal{A}|$ is bounded by a constant depending only on $T$ by Corollary \ref{cor:N}.  
Let $y_1,\ldots,y_n\in \R^\ell$ so that $T_{y_1},\ldots,T_{y_n}$ are definable subsets of $\R$.  We show that some subcollection of the sets \begin{equation}{\label{eq:case1}}\pi^{1+2\ell}_{\leq 1}((\pi^{1+2\ell}_{>1})^{-1}(y_i,y_j)\cap A),\end{equation} $A\in \mathcal{A}$ and $1\leq i,j\leq n$, form a cdd of $\R$ adapted to $T_{y_1},\ldots,T_{y_n}$, in fact that a subcollection of this form is the common refinement of $T_{y_1},\ldots,T_{y_n}$.  Let $W\subseteq \R$ be a set of the common refinement of $T_{y_1},\ldots,T_{y_n}$. If $W$ is a point assume without loss of generality that $W$ is an endpoint of $T_{y_1}$, then it is easy to see that $W=(X_{C,i})_{(y_1,y_1)}$ where $y_1$ is in the partitioning set $A_i$, for some $i$, $1\leq i \leq M$, and for some endpoint $C$ of $\mathcal{A}_i$.  If $W$ is an interval, $W=(a_1,a_2)$, then $a_1,a_2$ are endpoints of $T_{y_i},T_{y_j}$ respectively, for some $i,j$, $1\leq i,j\leq n$, and in this case 
it is easy to see that 
$W=(X_{C,l})_{(y_i,y_j)}$ for some interval $C$ of $ \mathcal{A}_l$, where $y_i\in A_l$.  So, we have shown that some subcollection of \ref{eq:case1} is the common refinement of $T_{y_1},\ldots,T_{y_n}$, which has at most $nN$ sets by Corollary \ref{cor:N}.  This finishes the case $k=1$.

\hide{
Note that 
$$\begin{aligned}(A_1\smallsetminus A_2)_{(y_i,y_j)}&=T_{y_i}\smallsetminus T_{y_j}, \\
(A_2\smallsetminus A_1)_{(y_i,y_j)}&=T_{y_j}\smallsetminus T_{y_i}, \\
(A_1\cap A_2)_{(y_i,y_j)}&=T_{y_i}\cap T_{y_j}, \\
(\R^{1+2\ell}\smallsetminus (A_1 \cup A_2))_{(y_i,y_j)}&=\R\smallsetminus (T_{y_i}\cup T_{y_j}).
\end{aligned}$$
The common refinement of $T_{y_1},\ldots,T_{y_n}$ is adapted to the sets in $\mathcal{B}(A_{(y_i,y_j)}|A\in \mathcal{A}, 1\leq i,j \leq n)$, and the common refinement is a cdd of $\R$.  Finally, the number of sets in the common refinement is at most $Nn$.  This finishes the case where $k=1$.
}

Now assume $k>1$ and by induction hypothesis that the theorem holds for smaller values of $k$.  Let $T\subseteq \R^{k+\ell}$ be definable, $N\geq 0$ as in Corollary \ref{cor:N}.  
Let $A_1,\ldots,A_M\subseteq \R^{k-1+2\ell}$ satisfy the conclusion of Lemma \ref{lem:extension} so that,  for every $y_1,y_2\in \R^\ell$, $T_{y_1}\cup T_{y_2}$ is definably trivial over $(A_i)_{(y_1,y_2)}$ and compatible with $T_{y_1}$ and $T_{y_2}$ for each $A_i$.  Let $T'=\cup_{1\leq i\leq M} A_i \times \{e_i\}$, where $e_i$ is the $i$-th standard basis vector of $\R^M$.  Note that $T'\subseteq \R^{k-1+2\ell+M}$.  

Applying the induction hypothesis to $T'$ 
we obtain definable subsets $\{A'_\beta \}_{\beta \in I'}$ of $\R^{k-1+2(2^{k-1}-1)(2\ell+M)}$, depending only on $T$ and
having the property that, for any $y_1,\ldots,y_n\in \R^{\ell}$ and $\textrm{a}=(a_1,\ldots,a_{2(2^{k-1}-1)})$ with each $a_i$ a standard basis vector of $\R^M$,  some sub-collection of the sets 
\begin{equation}{\label{eq:case2}}
\pi^{k-1+2(2^{k-1}-1)(2\ell+M)}_{\leq k-1}\left((\pi^{k-1+2(2^{k-1}-1)(2\ell+M)}_{>k-1})^{-1}(y_{i_1},\ldots,y_{i_{2^2(2^{k-1}-1)}},\textrm{a}) \; \cap \; A'_\beta\right),
\end{equation} $\beta \in I'$ and $1\leq i_1,\ldots,i_{2^2(2^{k-1}-1)}\leq n$, form a cdd of $\R^{k-1}$ adapted to the family $\{T'_{(y_i,y_j,a)}|\; 1\leq i,j\leq n,\; a \in \R^M \text{ standard basis vector}\}$.  Note that $T'_{(y_i,y_j,a)}$ is nothing more than $B_l(y_i,y_j)$ where $a$ has a $1$ in the $l$-th spot (cf. Lemma \ref{lem:extension}). 

Let $y_1,\ldots,y_n\in \R^\ell$.  Some subcollection of \ref{eq:case2} is a cdd of $\R^{k-1}$ adapted to $B_l(y_i,y_j)$, $1\leq i,j\leq n$, $1\leq l\leq M$ by induction hypothesis.   
Denote this cdd by $\{C_\alpha\}_{\alpha \in I}$ and note that $|I|\leq C_2'n^{2^{k}-2}$ for some constant $C'_2$ depending only on $T$ using the induction hypothesis and the fact that the cardinality of the set $\{B_{l}(y_i,y_j)|1\leq l\leq M, 1\leq i,j\leq n\}$ is at most $Mn^2$.  Since the cdd is adapted to $B_l(y_i,y_j)$, for each $\alpha \in I$ we have 
$C_\alpha \subseteq A$ or $C_\alpha \cap A=\emptyset$ for all $A\in \mathcal{B}(B_l(y_i,y_j)|1\leq l\leq M, 1\leq i,j\leq n)$.  
Recall the trivializations 
$\phi_{C,A}:C\times A\to Y_{C,A}$, for $C\in \mathcal{A}_A$ (cf. Lemma \ref{lem:extension}).
Denote the image of the restriction to $C\times C_\alpha$ of the trivialization $\phi_{C,A}$ by $X_{C,\alpha}$, where $C_\alpha \subseteq A$. It now follows from Lemma \ref{lem:extension} that $\{X_{C,\alpha}\}_{C\in \mathcal{A}_A, \alpha \in I}$ is a cdd of $\R^k$ adapted to $T_{y_1},\ldots,T_{y_n}$.  
Note that the sets $X_{C,\alpha}$ have the desired form (using Equation \ref{eq:VV})
\begin{equation}{\label{eq:sets}}
\{(a,x)\in \R^k| 
\; a\in \phi_{C,A}(C\times \{x\}) \wedge (x,y_{i_1},\ldots,y_{i_{2^2(2^{k-1}-1)}},\textrm{a})\in A'_\beta
\}, 
\end{equation}
for some $1\leq i_1,\ldots,i_{2^2(2^{k-1}-1)}\leq n$, where $x\in C_\alpha\subseteq A$, $\beta \in I'$ and $\textrm{a}=(a_1,\ldots,a_{2(2^{k-1}-1)})$, with each $a_i$ a standard basis vector of $\R^M$.  
Finally, the cardinality of 
$\{X_{C,\alpha}|\; C\in \mathcal{A}_A, \alpha \in I\}$   
is at most $|\mathcal{A}_A|\cdot |I| \leq C_2 n^{2^k-1}$, since $|\mathcal{A}_A|\leq Nn$ for each $A\in \mathcal{B}(B_l(y_i,y_j)|1\leq l\leq M, 1\leq i,j\leq n)$ by Corollary \ref{cor:N}, which finishes the proof. 





\end{proof}

The proof of the main theorem shares many of the same ideas as those we have just seen in the proof of the cylindrical definable decomposition theorem.  We give the main ideas of the proof now.  

The proof is by induction.  The case $k=1$ is identical to the cylindrical case.  
For $k>1$, we begin by using Lemma \ref{lem:extension} (Cylindrical Extension) to obtain a set $T'$ which has the property that certain definable fibers of $T'$ are the sets $B_l(y_i,y_j)$, over which $T_{y_i}\cup T_{y_j}$ is trivial and such that the trivialization is compatible with $T_{y_i},T_{y_j}$.  
We will use the induction hypothesis $n^2$ times, once for each pair $(i,j)$, on the collection $\{B_l(y_q,y_p)|\; 1\leq l\leq M, q\in \{i,j\}, 1\leq p \leq n \}$ which has only $O(n)$ sets,  
to find a scdd $\mathcal{A}$ of $\R^{k-1}$ into about $O(n^{2k-3})$ sets which have the property that the sets in $\mathcal{A}$ can be cylindrically extended to $\R^k$ using Lemma \ref{lem:extension} (Cylindrical Extension) to form a partition adapted to $T_{y_i},T_{y_j}$. 
 We make 
$n^2$ such partitions of $\R^k$ for each pair $(i,j)$, $1\leq i, j \leq n$, and call the partition associated to the pair $i,j$ the $(i,j)$-partition.  
Note that the total number of sets among all $(i,j)$-partitions is $O(n^{2k-1})$ since each of the $n^2$ partitions had about $O(n^{2k-3})$ sets.  
Now all the sets in these $(i,j)$-partitions are not necessarily adapted to $T_{y_p}$, $p\neq i,j$, and furthermore the sets of different partitions may overlap. 
We prove that it is possible to select out of all the sets in all the $(i,j)$-partitions a sub-collection of sets which form a partition of $\R^k$ which is adapted to $T_{y_1},\ldots,T_{y_p}$, obtaining the desired scdd of $\R^k$.  Since this selection procedure is the main difference between the cylindrical and semi-cylindrical decompositions, we next describe the selection procedure in some detail.  

Given $x\in \R^k$, we prove that it is possible to select a unique set $X=X(x)$ among all the $(i,j)$-partitions which contains $x$ and which has the property that for any $x'\in X$ we have $X(x')=X$, and furthermore so that $X\cap T_{y_p}=\emptyset$ or $X\subseteq T_{y_p}$ for all $p$, $1\leq p \leq n$.  Thus, the sets which are selected in this way form a partition of $\R^k$ adapted to $T_{y_1},\ldots,T_{y_n}$.  The selection procedure is as follows.  Let $\pi=\pi^k_{>1}$ and $\pi_2=\pi^k_{\leq 1}$.  For each $x\in \R^k$, considering the common refinement of the one-dimensional sets $T_{(\pi(x),y_1)},\ldots,T_{(\pi(x),y_n)}$, we have two cases to consider: either $\pi_2(x)$ is an endpoint of the common refinement or it belongs to one of the intervals of the common refinement.  If $\pi_2(x)$ corresponds to an endpoint, then $x$ belongs to a graph $X$ of some $(i,j)$-partition, and the set of indices of the sets $T_{(\pi(x),y_1)},\dots,T_{(\pi(x),y_n)}$ for which $x$ is an endpoint is invariant for any other $x'$ in the set $X$.  If $\pi_2(x)$ belongs to an interval, then $x$ belongs to a band $X$ in each $(i,j)$-partition, and if $i,j$ are among indices of $T_{(\pi(x),y_1)},\dots,T_{(\pi(x),y_n)}$ corresponding to the endpoints of the interval containing $x$, then this set $X$ is adapted to $T_{y_1},\ldots,T_{y_n}$ and furthermore for any choice $x'$ in $X$ the same indices $i,j$ correspond to the endpoints of the interval containing $x'$ in the common refinement of $T_{(\pi(x'),y_1)},\ldots,T_{(\pi(x'),y_n)}$.  By taking the minimum of the indices which satisfy the corresponding property in each case we obtain the unique $X$ which will be selected to be included in the final scdd.
We now provide the proof of the main result of the chapter.




\begin{proof}[Proof of Theorem \ref{thm:scdd}]
The proof is by induction on $k$.  The case $k=1$ is identical to the cylindrical theorem and is not repeated here.

Let $k>1$ and suppose by induction hypothesis that the claim holds for $k-1$.  	As in the proof for the cylindrical theorem above, let $T\subseteq \R^{k+\ell}$ be definable and $N\geq 0$ as in Corollary \ref{cor:N}.  
Let $A_1,\ldots,A_M\subseteq \R^{k-1+2\ell}$ satisfy the conclusion of Lemma \ref{lem:extension} so that,  for every $y_1,y_2\in \R^\ell$, $T_{y_1}\cup T_{y_2}$ is definably trivial over $(A_i)_{(y_1,y_2)}$ and compatible with $T_{y_1}$ and $T_{y_2}$ for each $A_i$.  Let $T'=\cup_{1\leq i\leq M} A_i \times \{e_i\}$, where $e_i$ is the $i$-th standard basis vector of $\R^M$.  Note that $T'\subseteq \R^{k-1+2\ell+M}$.  

Fix $y_1,\ldots,y_n\in \R^\ell$. We apply the induction hypothesis to the set $T'$ to obtain definable subsets $\{A'_{i,j,\alpha}\}_{\alpha \in I_{i,j}}$ of $\R^{k-1+2(2^k-1)(2\ell +M)}$, depending only on $T$ and having the property that
some sub-collection of the sets 
$$\pi^{k-1+2(2^{k-1}-1)(2\ell+M)}_{\leq k-1}\left((\pi^{k-1+2(2^{k-1}-1)(2\ell+M)}_{>k-1})^{-1}(y_{i_1},\ldots,y_{i_{2^2(2^{k-1}-1)}},\textrm{a}) \; \cap \; A'_{\beta}\right),
$$
$\beta \in I'_{i,j}$ and $1\leq i_1,\ldots,i_{2^2(2^{k-1}-1)}\leq n$ and $\textrm{a}=(a_1,\ldots,a_{2(2^{k-1}-1)})$ with each $a_i$ a standard basis vector of $\R^M$, form a scdd of $\R^{k-1}$ into definable sets which are adapted to the family $\{T'_{(y_q,y_p,a)}|\; q\in\{i,j\}, 1\leq p\leq n, a \in \R^M \text{ standard basis vector}\}$, and by induction hypothesis we have that the number of sets in the scdd is bounded by 
$C_2'n^{2(k-1)-1}$ 
for some $C_2'$ depending only on $T$. Denote this scdd by 
\begin{equation}{\label{eq:scdd}}
\{C^{i,j}_\alpha\}_{\alpha \in I}.
\end{equation}  Note that $T'_{(y_i,y_j,a)}$ is nothing more than $B_l(y_i,y_j)$ where $a$ has a $1$ in the $l$-th spot (cf. Lemma \ref{lem:extension}).  In particular, for each pair 
$(i,j)$, $1\leq i,j\leq n$, recalling 
the homeomorphisms $\phi_{C,A}^{i,j}:C\times A\to Y_{C,A}$, for $C\in \mathcal{A}_A$ and $A$ in the boolean algebra $\mathcal{B}^{i,j}:=\mathcal{B}(B_l(y_q,y_p)|q\in \{i,j\},1\leq p\leq n\}$ (cf. Lemma \ref{lem:extension}), and denoting the image of the restriction to $C\times C^{i,j}_\alpha$ of these trivializations by $X^{i,j}_{C,\alpha}$, recall that $\{X^{i,j}_{C,\alpha}\}_{C\in \mathcal{A}_A, 
A\in \mathcal{B}^{i,j}, \alpha\in I}$ is a partition of $\R^k$ adapted to $T_{y_i},T_{y_j}$.  
For convenience we refer to this partition as the $(i,j)$-partition 
and denote 
\begin{equation}{\label{eq:B}}
\mathcal{B}:=\bigcup_{1\leq i,j\leq n} \mathcal{B}^{i,j}
.  
\end{equation}
Note that the sets $X^{i,j}_{C,\alpha}$ have the desired form in Equation \ref{eq:sets} (using Equation \ref{eq:VV}),
\hide{
\begin{equation}{\label{eq:sets}}
\{(a,x)\in \R^k| 
 a\in \phi_{C,l}^{(i,j)}(C\times \{x\}) \wedge (x,y_{i_1},\ldots,y_{i_{2^2(2^{k-1}-1)}},\textrm{a})\in A'_{i,j,\alpha}
\}, 
\end{equation}
}
for some $1\leq i_1,\ldots,i_{2^2(2^{k-1}-1)}\leq n$, where $x\in C_\alpha$, $\beta \in I'$ and $\textrm{a}=(a_1,\ldots,a_{2(2^{k-1}-1)})$, with each $a_i$ a standard basis vector of $\R^M$.  
\hide{
using Lemma \ref{lem:extension} we extend the partition of $\R^{k-1}$ adapted to $B_l(y_i,y_j)$ to a partition of $\R^k$ adapted to $T_{y_i},T_{y_j}$. For convenience we will refer to this partition as the $(i,j)$-partition and denote it by 
$\{X_{C,i,j,l\alpha}|\; C\in \mathcal{A}_l, 1\leq i,j\leq n, 1\leq l \leq M, \alpha \in I\}$.
 and note each set has the form \ref{eq:sets}.  
}
Since by Corollary \ref{cor:N} the number of sets in the common refinement of $T_{y_i},T_{y_j}$ is at most $2N$, note that each $(i,j)$-partition consists of at most
$2NC_2' n^{2(k-1)-1}$ sets for some constant $C'_2$ depending only on $T$. 

We now select from each $(i,j)$-partition the sets which will comprise the semi-cylindrical decomposition.  In order to describe the selection process, we need the following definition and lemma.  We define a map from $\R^k$ to the collection of sets among all $(i,j)$-partitions which associates to each point $x\in \R^k$ a unique set in some $(i,j)$-partition which contains it.  We first define maps which pick out the indices of the sets $T_{y_1},\ldots,T_{y_n}$ which are closest vertically to the point $x$ and then use these maps to define the selection criterion. 

Let $x\in \R^k$ 
and let $\{a_1,\ldots,a_p\}\subseteq \R$ be the finite set of endpoints among the sets $T_{(\pi(x),y_1)},\dots,T_{(\pi(x),y_n)}$, and suppose $a_j<a_{j+1}$, $0\leq j\leq p$ with the convention that $a_0=-\infty$ and $a_{p+1}=+\infty$.  Let $\mathcal{I}:\{a_1,\ldots,a_p\}\to \mathcal{P}(\{1,\ldots,n\})$ be the map which associates to each $a_j$, $1\leq j\leq p$, the set of indices $\{i| \;  a_j \text{ is an endpoint of } T_{(\pi_2(x),y_i)}\}$.  Let $\pi_2:\R^{k}\to \R$ be the projection onto the first coordinate.  If $\pi_2(x)=a_j$ for some $j=1,\ldots,n$, define $\xi(x)=\mathcal{I}(a_j)$ and $\mu(x)=\nu(x)=\emptyset$, and otherwise if $a_j<\pi_2(x)<a_{j+1}$ for some $0\leq j\leq p$, 
 define $\xi(x)=\emptyset$, $\mu(x)=\mathcal{I}(a_j)$ and $\nu(x)=\mathcal{I}(a_{j+1})$ with the convention that $\mathcal{I}(a_0)=\mathcal{I}(a_{p+1})=\{n\}$.  Then we have the following lemma.
 
\begin{lemma}{\label{lem:well}}
Let $x\in \R^k$.  If $\xi(x)\neq \emptyset$ let $i_0=\min \xi(x)$ and let $X$ be the unique set in the $(i_0,i_0)$-partition which contains $x$, then the set $\xi(x)$ is invariant for each $x\in X$.  If $\xi(x)=\emptyset$ let $i_0=\min(\min \mu(x), \min \nu(x))$ and $j_0=\max(\min\mu(x),\min \nu(x))$ and let $X$ be the unique set in the $(i_0,j_0)$-partition which contains $x$, then the sets $\mu(x),\nu(x)$ are invariant for each $x\in X$.   

\end{lemma}

\begin{proof}
Suppose $\xi(x)\neq \emptyset$ and $X$ as above.  
We show that the map $\xi(x)$ is constant for all $x\in X$.  
The proof to show that the maps $\mu(x)$ and $\nu(x)$ are constant for all $x\in X$ is nearly identical and is not repeated.  

By definition, $T_{y_{i_0}}\cup T_{y_p}$ is definably trivial over $\pi(X)$, and this trivialization is compatible with $T_{y_{i_0}},T_{y_p}$ for all $1\leq p\leq n$, since $\pi(X)\subseteq B_l(y_{i_0},y_p)$ for some $l$, $1\leq l\leq M$.  
Suppose, seeking a contradiction, that there exists $x,x'\in X$ such that $\xi(x)\neq \xi(x')$, so there exists an index $p\in (\xi(x)\cup \xi(x'))\smallsetminus (\xi(x)\cap \xi(x'))$.  Since $X$ is path connected, there exists a path in $X$ from $x$ to $x'$, that is, a continuous map $\Gamma:[0,1]\to X$ such that $\Gamma(0)=x$ and $\Gamma(1)=x'$.  
But $T_{y_{i_0}}\cup T_{y_p}$ is not definably trivial over $\pi(\Gamma([0,1]))$, and $\pi(\Gamma([0,1]))$ is contained in $\pi(X)$, a contradiction.  


\end{proof}
\hide{
We next show that the map $\mu(x)$ is constant for all $x\in X_\gamma$. 
Suppose, seeking a contradiction, that there exists $x,x'\in X_\gamma$ such that $\xi(x)\neq \xi(x')$, so there exists an index $i_0\in (\xi(x)\cup \xi(x'))\smallsetminus (\xi(x)\cap \xi(x'))$.  Since $X_\gamma$ is path connected, there exists a path in $X_\gamma$ from $x$ to $x'$, that is, a continuous map $\Gamma:[0,1]\to X_\gamma$ such that $\Gamma(0)=x$ and $\Gamma(1)=x'$.  But $T_{y_{i_0}}$ is not definably trivial over $\pi(\Gamma([0,1]))$ by choice of $i_0$, and in particular this implies that $\pi(\Gamma([0,1]))$ is not contained in $\pi(X_\gamma)$, a contradiction.  
}

We now define a map which will aid in the selection process to decide which sets of the $(i,j)$-partitions are be included in the final scdd.  Define the map $\theta$ from $\R^k$ to the collection of sets among all $(i,j)$-partitions as follows.   For each $x\in \R^k$, 
 if $\xi(x)\neq \emptyset$, let $i_0=\min(\xi(x))$, and let $\theta(x)$ be the unique $X$ in the $(i_0,i_0)$-partition which contains $x$. Otherwise, if $\xi(x)=\emptyset$ then let $i_0=\min(\min(\mu(x)),\min(\nu(x))$ and $j_0=\max(\min(\mu(x)),\min(\nu(x))$ and let $\theta(x)$ be the unique $X$ in the $(i_0,j_0)$-partition which contains $x$.  We say that a set $X$ in an $(i,j)$-partition satisfies the inclusion criterion if $X=\theta(x)$ for some $x\in \R^k$.  
 

Now collect all sets from all $(i,j)$-partitions which satisfy the selection criterion, denote this collection by $\mathcal{A}=\{X_\gamma\}_{\gamma\in J}$, and 
by the selection criterion we see that each $x\in \R^k$ is contained in a unique set in the collection $\mathcal{A}$.  
So, $\mathcal{A}=\{X_\gamma\}_{\gamma \in J}$ 
is a partition of $\R^k$ and, by definition and Lemma \ref{lem:well}, it is adapted to $T_{y_1},\ldots,T_{y_n}$, and we have shown that each set is definable and has the form in the statement of the theorem (via Equation \ref{eq:sets} and Equation \ref{eq:VV}).  It remains to show that the partition has the right number of sets.
Each of the $n^2$ K-decompositions consists of at most $2NC_2' n^{2(k-1)-1}$ sets, so there are at most $n^2\cdot 2NC_2' n^{2(k-1)-1}\leq C_2n^{2k-1}$ sets in the partition $\mathcal{A}$ for some $C_2$ depending only on $T$.



\end{proof}

\begin{remark}
In certain settings, in particular for semi-algebraic sets, a slightly better bound on the number of sets in the semi-cylindrical decomposition can be proven as follows.  First one proves a case by case bound which is better than the proposed bound on the number of sets in the decomposition in low dimension using the additional tools available for semi-algebraic sets in low dimension, e.g., for dimension $k=2,3,4$.  Using the better bound as a point of departure and the same provable growth rate in the induction step, the technique of the proof given above yields an improved bound taking advantage of the fact that better bounds can be proven in low dimension in these certain situations.

\end{remark}









\bibliography{master}



\end{document}